\documentclass[a4paper]{amsart}

\usepackage[french, english]{babel} 
\usepackage[T1]{fontenc}
\usepackage[ansinew]{inputenc}

\date{}
\title[Polar decomposition and regularizing effects]{Polar decomposition of semigroups generated by non-selfadjoint quadratic differential operators and regularizing effects}

\author{Paul Alphonse}
\address{(Paul Alphonse) Univ Rennes, CNRS, IRMAR - UMR 6625, F-35000 Rennes}
\email{paul.alphonse@ens-rennes.fr}

\author{Joackim Bernier}
\address{(Joackim Bernier) Univ Rennes, CNRS, IRMAR - UMR 6625, F-35000 Rennes}
\email{joackim.bernier@ens-rennes.fr}



\usepackage[top=3cm, bottom=2cm, left=3cm, right=3cm]{geometry}		

\usepackage{mathtools}

\usepackage{enumitem}

\usepackage{array}											

\usepackage{amsfonts}
\usepackage{amsmath}
\usepackage{amsthm}
\usepackage{amssymb}

\usepackage{bbm}

\usepackage{stmaryrd}

\usepackage[scr]{rsfso}

\usepackage{comment}

\usepackage{hyperref}

\usepackage{xcolor}

\numberwithin{equation}{section}

\newtheorem{thm}{Theorem}[section]
\newtheorem{prop}[thm]{Proposition}
\newtheorem{lem}[thm]{Lemma}
\newtheorem{cor}[thm]{Corollary}
\newtheorem{dfn}[thm]{Definition}
\theoremstyle{definition}
\newtheorem{ex}[thm]{Example}
\newtheorem{rk}[thm]{Remark}

\DeclareMathOperator{\Reelle}{Re}
\DeclareMathOperator{\Tr}{Tr}
\DeclareMathOperator{\Ker}{Ker}
\DeclareMathOperator{\Imag}{Im}

\DeclareMathOperator{\GL}{GL}
\DeclareMathOperator{\Symm}S
\DeclareMathOperator{\Symp}{Sp}
\DeclareMathOperator{\Supp}{Supp}
\DeclareMathOperator{\Op}{Op}
\DeclareMathOperator{\Her}H
\DeclareMathOperator{\atanh}{atanh}

\DeclareMathOperator{\ad}{ad}
\DeclareMathOperator{\atan}{atan}
\DeclareMathOperator{\Log}{Log}

\begin{document}

\sloppy

\selectlanguage{english}

\begin{abstract}
We characterize geometrically the regularizing effects of the semigroups generated by accretive non-selfadjoint quadratic differential operators. As a byproduct, we establish the subelliptic estimates enjoyed by these operators, being expected to be optimal. These results prove conjectures by M. Hitrik, K. Pravda-Starov and J. Viola. The proof relies on a new representation of the polar decomposition of these semigroups. In particular, we identify the selfadjoint part as the evolution operator generated by the Weyl quantization of a time-dependent real-valued nonnegative quadratic form for which we prove a sharp anisotropic lower bound.
\end{abstract}

\maketitle

\section{Introduction}\label{intro}

We consider the semigroups generated by accretive non-selfadjoint quadratic differential operators. They are the evolution operators associated with partial differential equations of the form
\begin{equation}\label{myPde}
\left\{\begin{array}{l}
	\partial_t u + q^w(x,D_x) u = 0, \\[2pt]
	u(0,\cdot) = u_0,
\end{array}\right.
\end{equation}
where $u_0\in L^2(\mathbb R^n)$, $n\geq 1$ is a fixed number and $q^w(x,D_x)$ is the \emph{Weyl quantization} of  a complex-valued quadratic form $q:\mathbb R^{2n}\rightarrow\mathbb C$ with a nonnegative real part. Denoting $Q\in S_{2n}(\mathbb C)$ the matrix of $q$ in the canonical basis of $\mathbb R^{2n}$, $q^w(x,D_x)$ is nothing but the differential operator 
$$q^w(x,D_x) = \begin{pmatrix} x & -i\nabla \end{pmatrix} Q \begin{pmatrix} x \\ -i\nabla
\end{pmatrix}.$$
This operator is equipped with the domain $D(q^w) = \{u\in L^2(\mathbb R^n) : q^w(x,D_x)u\in L^2(\mathbb R^n)\}.$ Note that this definition coincides with the classical definition of $q^w(x,D_x)$ as an oscillatory integral. We recall that since the real part of the quadratic form $q$ is nonnegative, the quadratic operator $q^w(x,D_x)$ is shown in \cite[pp. 425-426]{MR1339714} to be maximal accretive and to generate a strongly continuous contraction semigroup $(e^{-t q^w})_{t\geq0}$ on $L^2(\mathbb R^n)$. 

In this paper, proving a conjecture of M. Hitrik, K. Pravda-Starov and J. Viola in \cite{MR3710672}, we characterize and quantify geometrically the regularizing effects of  $(e^{-t q^w})_{t\geq0}$ in the asymptotic $0< t \ll1$. Basically, we determine how smooth and localized are the mild (i.e.\ semigroup) solutions of \eqref{myPde}. This problematic is natural and interesting in itself but it is also motivated by its applications in control theory (see Remark \ref{rk_control} below). Furthermore, it is not trivial because, since our operators are non-selfadjoint, we have to deal with nonlinear interactions between phenomena of diffusions and transports (understood in some very weak senses). For example, considering the Kolmogorov operator $x_2\partial_{x_1} - \partial_{x_2}^2$,  it can be proven that its associated semigroup is smoothing super-analytically both with respect to the variables $x_1$ and $x_2$ (see e.g.\ \cite{AB}). In the more general framework of the quadratic differential operators, this problematic was widely studied (see e.g.\ \cite{A, MR2507625, MR3710672, MR3841852, MR3756858}) but results were established only for some specific subclasses of these operators. As a byproduct, using interpolation theory, we establish sharp subelliptic estimates that were also conjectured in \cite{MR3710672} and widely studied (see e.g.\ \cite{AB,MR3710672,MR3841852,MR2752935}).

Beyond our results, we believe that one of the main interests of this paper consists in the methods we introduce, their possible applications and the links we highlight between the analysis of the properties of semigroups and the study of splitting methods in geometric numerical integration. Our proof relies on a new representation of the polar decomposition of the evolution operators:
\begin{equation}\label{intro_thedec}
	 e^{-tq^w} =  e^{-ta_t^w} e^{-itb_t^w},
\end{equation}
where $a_t,b_t$ are some real valued quadratic forms depending analytically on $0\leq t\ll 1$, $a_t$ is nonnegative and $e^{-ta_t^w}$ (resp. $e^{-itb_t^w}$) denotes the evolution operator generated by $a_t^w$ (resp. $ib_t^w$) at time $t$. The existence of such a representation relies on the \emph{exact classical-quantum correspondence} (through the theory of Fourier Integral Operators developed by L. H\"ormander in \cite{HFIO}). This correspondance allows to identify a semigroup generated by the Weyl quantization of a quadratic form with the Hamiltonian flow of this quadratic form (i.e.\ the exponential of a matrix). The key observation in this paper is that, since  $e^{-t i b_t^w} $ is unitary, the regularizing effects are entirely driven by  $e^{-t a_t^w}$. In other words, $a_t$ encodes all the regularizing effects generated by the nonlinear interactions between the phenomena of diffusions and transports. For example, a formula of Kolmogorov (see e.g.\ \cite{AB}) proves that for the Kolmogorov operator, the factorization \eqref{intro_thedec} becomes
$$\forall t\geq 0,\quad e^{t (\partial_{x_2}^2-x_2\partial_{x_1})} = e^{t (\partial_{x_2} + t\partial_{x_1}/2)^2 + t^3 \partial_{x_1}^2/12 }e^{-tx_2\partial_{x_1}}.$$
As a consequence, the smoothing properties of this semigroup become as explicit as for the heat equation. Obviously, in general, there is no elementary explicit formula giving $a_t$. The main technical result of this paper is the derivation of a sharp anisotropic lower bound for $a_t$ in the asymptotic of $0< t \ll 1$. The starting of this derivation is the observation that $a_t^w$ results from the \emph{backward error analysis} of the Lie splitting method\footnote{This is a classical problematic in geometric numerical integration, we refer the reader to \cite{HLW} for a presentation of this topic.} associated with the decomposition $2(\Reelle q)^w = q^w + \bar q^w$:
$$e^{-tq^w}e^{-t\bar q^w} = e^{-2ta_t^w}.$$ 
This formula provides a direct way to determine $a_t$ as a function of $t$ and $q$. Using the generalization \cite{MR3880300} of the results of L. H\"ormander \cite{MR1339714}, our results could be extended to non-autonomous equations. Furthermore, in view of \cite{B,Joe}, we expect that our results could be extended to deal with semigroups generated by inhomogeneous quadratic differential operators. However, these extensions would require some important technicalities. Consequently, they would deserve some further analysis in future works. For the moment, it is not clear how our methods could be extended to deal with non-quadratic operators. It would also deserve some further investigations.
We believe that our representation \eqref{intro_thedec} could also be useful to analyse some other properties of the semigroups like the propagation of coherent states or singularities. Finally, our methods seem promising to design and analyse rigorously some splitting methods to solve numerically equations of the form \eqref{myPde}, see~\cite{B}.

\subsubsection*{Outline of the work} Section \ref{sec_relou} is devoted to present the main results contained in this paper, put in their bibliographic context and illustrated with examples. In Section \ref{Polar}, we establish the polar decomposition of quadratic semigroups in any positive times whereas Section \ref{short} is devoted to the study of the selfadjoint part for small times. As a byproduct of this decomposition, we study the regularizing effects of semigroups generated by non-selfadjoint quadratic differential operators in Section \ref{Reg} from which we derive subelliptic estimates enjoyed by accretive quadratic operators in Section \ref{Sub}. Section \ref{Appendix} is an appendix containing the proofs of some technical results.

\subsubsection*{Convention} Any complex-valued quadratic form $q:\mathbb R^{2n}\rightarrow\mathbb C$ will be implicitly extended to the complex phase space $\mathbb C^{2n}$ in the following way:
\begin{equation}\label{07052019E1}
	\forall X\in\mathbb C^{2n},\quad q(X) = X^TQ\overline X = q(\Reelle X) + q(\Imag X),
\end{equation}
where $Q\in\Symm_{2n}(\mathbb C)$ denotes the matrix of the quadratic form $q$ in the canonical basis of $\mathbb R^{2n}$.

\subsubsection*{Notations} The following notation will be used all over the work:
\begin{enumerate}[label=\textbf{\arabic*.},leftmargin=* ,parsep=2pt,itemsep=0pt,topsep=2pt]
\item For all complex matrix $M\in M_n(\mathbb C)$, $M^T$ denotes the transpose matrix of $M$ while $M^* = \overline M^T$ denotes its adjoint.
\item $\langle\cdot,\cdot\rangle$ denotes the inner product on $\mathbb C^n$ as defined in \eqref{15022019E4}.
\item We set $\vert\cdot\vert$ the Euclidean norm on $\mathbb R^n$ extended to $\mathbb C^n$ as explained in the previous convention.
\item The notation $\Vert\cdot\Vert$ stands for the matrix norm on $M_{2n}(\mathbb C)$ induced by the norm $\Vert\cdot\Vert_2$ on $\mathbb C^{2n}$. From there, we introduce the norm $\Vert\cdot\Vert_{\infty}$ on $M_{2n}(\mathbb C)\times M_{2n}(\mathbb C)$ defined by 
\begin{equation}\label{29032021E1}
	\Vert(M,N)\Vert_{\infty} = \max(\Vert M\Vert,\Vert N\Vert).
\end{equation}
\item When $\mathbb K = \mathbb R$ or $\mathbb C$, we denote by $\Symp_{2n}(\mathbb K)$ the symplectic group whose definition is recalled at the beginning of Subsection \ref{symp}.
\item We denote by $\mathbb C\langle X,Y \rangle$ the ring of the non-commutative polynomials in $X$ and $Y$, as defined e.g.\ in \cite[Chapter 6]{MR3308118}. For all nonnegative integer $k\geq0$, we set $\mathbb C_{k,0}\langle X,Y \rangle$ the subspace of $\mathbb C\langle X,Y \rangle$ of non-commutative polynomials of degree smaller than or equal to $k$ vanishing at $(0,0)$.
\item For all vector subspaces $V\subset \mathbb K^n$, with $\mathbb K=\mathbb R$ or $\mathbb C$, the notation $V^{\perp}$ is devoted for the orthogonal complement of $V$ with respect to the canonical Euclidean (when $\mathbb K = \mathbb R$) or Hermitian (when $\mathbb K = \mathbb C$) structure of $\mathbb K^n$.
\item If $f:(-\alpha,\alpha)\rightarrow M_n(\mathbb C)$ is an analytic function such that $f(0) = 0$, with $\alpha\in(0,+\infty]$, there exists another analytic function $g:(-\alpha,\alpha)\rightarrow M_n(\mathbb C)$ such that for all $t\in(-\alpha,\alpha)$, $f(t) = tg(t)$. With an abuse of notation, we will denote
\begin{equation}\label{15052019E1}
	\forall t\in(-\alpha,\alpha),\quad g(t) = f(t)/t.
\end{equation}
\end{enumerate}

\section{Formalism and main results}
\label{sec_relou}

\subsection{Hamiltonian formalism and singular space} Before stating the main results contained in this paper, we need to introduce the Hamilton map and the singular space associated with the quadratic form $q$, which will play a key role in the following. According to \cite[Definition 21.5.1]{MR2304165}, the Hamilton map $F$ of the quadratic form $q$ is defined as the unique matrix $F\in M_{2n}(\mathbb C)$ satisfying the identity
\begin{equation}\label{14022019E9}
	\forall X,Y\in\mathbb R^{2n},\quad q(X,Y) = \sigma(X,FY),
\end{equation}
with $q(\cdot,\cdot)$ the polarized form associated with $q$ and $\sigma$ the standard symplectic form given by
\begin{equation}\label{09052019E12}
	\sigma((x,\xi),(y,\eta)) = \langle\xi,y\rangle -\langle x,\eta\rangle,\quad (x,y),(\xi,\eta)\in\mathbb C^{2n},
\end{equation}
where $\langle\cdot,\cdot\rangle$ denotes the inner product on $\mathbb C^n$ defined by
\begin{equation}\label{15022019E4}
	\langle x,y\rangle = \sum_{j=0}^nx_jy_j,\quad x =(x_1,\ldots,x_n),y=(y_1,\ldots,y_n)\in\mathbb C^n.
\end{equation}
Note that this inner product $\langle\cdot,\cdot\rangle$ is linear in both variables but not sesquilinear. By definition, the matrix $F$ is given by
\begin{equation}\label{05112018E7}
	F = JQ,
\end{equation}
where $Q\in S_{2n}(\mathbb C)$ is the symmetric matrix associated with the bilinear form $q(\cdot,\cdot)$,
\begin{equation}\label{29062018E3}
	\forall X,Y\in\mathbb R^{2n},\quad q(X,Y) = \langle QX,Y\rangle = \langle X,QY\rangle,
\end{equation}
and $J\in\GL_{2n}(\mathbb R)$ stands for the symplectic matrix defined by
\begin{equation}\label{08112018E1}
J = \begin{pmatrix}
	0_n & I_n \\
	-I_n & 0_n
\end{pmatrix}\in\Symp_{2n}(\mathbb R).
\end{equation}

The notion of singular space was introduced in \cite[formula (1.1.14)]{MR2507625} by M. Hitrik and K. Pravda-Starov by pointing out the existence of a particular vector subspace $S$ in the phase space $\mathbb R^{2n}$, which is intrinsically associated with the quadratic symbol $q$, and defined as the
following intersection of kernels
\begin{equation}\label{13112018E6}
	S = \bigcap_{j=0}^{+\infty}\Ker(\Reelle F(\Imag F)^j)\cap\mathbb R^{2n},
\end{equation}
where the notations $\Reelle F$ and $\Imag F$ stand respectively for the real part and the imaginary part of the Hamilton map $F$ associated with $q$. Note that the subspace $S$ readily satisfies the following two properties
\begin{equation}\label{23012019E2}
	(\Reelle F)S = \{0\}\quad\text{and}\quad(\Imag F)S\subset S.
\end{equation}
Furthermore, the intersection defining $S$ in \eqref{13112018E6} being an intersection of subspaces of a finite dimensional vector space, this intersection is finite. More precisely, we may consider the smallest integer $k_0\geq0$ satisfying 
\begin{equation}\label{12102018E7}
	S = \bigcap_{j=0}^{k_0}\Ker(\Reelle F(\Imag F)^j)\cap\mathbb R^{2n}.
\end{equation}
Notice that as a consequence of the Cayley-Hamilton theorem, we have $0\le k_0 \le 2n-1$. This integer $k_0$ will play a key role in the following. Since the quadratic symbol has a nonnegative real part $\Reelle q\geq0$, the singular space can be defined in an equivalent way as the subspace in the phase space where all the Poisson brackets
$$H_{\Imag q}^k\Reelle q = \big(\partial_{\xi}\Imag q\cdot\partial_x-\partial_x\Imag q\cdot\partial_{\xi}\big)^k\Reelle q,\quad k\geq0,$$
are vanishing
$$S = \big\{X\in\mathbb R^{2n} : (H^k_{\Imag q}\Reelle q)(X) = 0,\ k\geq0\big\}.$$
This dynamical definition shows that the singular space corresponds exactly to the set of points $X\in\mathbb R^{2n}$, where the function $t\mapsto(\Reelle q)(e^{tH_{\Imag q}}X)$ vanishes to infinite order at $t=0$. This is also equivalent to the fact that this function is identically zero on $\mathbb R$. As pointed out in \cite{MR2507625, MR2752935, MR3244980}, the singular space is playing a basic role in understanding the spectral and hypoelliptic properties of non-elliptic quadratic operators, as well as the spectral and pseudospectral properties of certain classes of degenerate doubly characteristic pseudodifferential operators \cite{MR2753626, MR3137478}.

\subsection{Polar decomposition of semigroups generated by non-selfadjoint quadratic differential operators} We begin by giving a sharp description of the polar decomposition of the evolution operators $e^{-tq^w}$. More precisely, we aim at establishing that for any $t\geq0$, the operator $e^{-tq^w}$ admits the decomposition
\begin{equation}\label{01022019E1}
	e^{-tq^w} = e^{-ta^w_t}e^{-itb^w_t},
\end{equation}
where $a_t,b_t:\mathbb R^{2n}\rightarrow\mathbb R$, with $t\geq0$, are real-valued time-dependent quadratic forms, $a_t$ being nonnegative. In formula \eqref{01022019E1}, the linear operators $e^{-ta^w_t}$ and $e^{-itb^w_t}$ are defined as follows: for some fixed $t\geq0$, the quadratic operators $a^w_t(x,D_x)$ and $ib^w_t(x,D_x)$ respectively generate a semigroup $(e^{-sa^w_t})_{s\geq0}$ and a group $(e^{-isb^w_t})_{s\in\mathbb R}$ of contraction operators on $L^2(\mathbb R^n)$ (since the quadratic form $a_t$ is nonnegative and the quadratic form $ib_t$ is purely imaginary) and the operators $e^{-ta^w_t}$ and $e^{-itb^w_t}$ are respectively defined by
\begin{equation}\label{14052019E1}
	e^{-ta^w_t} = e^{-sa^w_t}\big\vert_{s=t}\quad\text{and}\quad e^{-itb^w_t} = e^{-isb^w_t}\big\vert_{s=t}.
\end{equation}
Notice that if the quadratic operators $(\Reelle q)^w$ and $(\Imag q)^w$ commute, then the relation \eqref{01022019E1} is satisfied with $a_t = \Reelle q$ and $b_t = \Imag q$. Moreover, \eqref{01022019E1} is the polar decomposition of the evolution operator $e^{-tq^w}$ as defined in Subsection \ref{polardecomposition}. In fact, the equality \eqref{01022019E1} will be proven only for small times $0\le t\ll1$. In the case where $t\gg1$, a formula similar to \eqref{01022019E1} will be established with the operator $e^{-itb^w_t}$ replaced by a unitary operator $U_t$ which \textit{a priori} cannot be written as an operator defined in \eqref{14052019E1}. The main result contained in this paper is the following:

\begin{thm}\label{12102018T1} Let $q:\mathbb R^{2n}\rightarrow\mathbb C$ be a complex-valued quadratic form with a nonnegative real part $\Reelle q\geq0$. Then, there exist a family $(a_t)_{t\in\mathbb R}$ of nonnegative quadratic forms $a_t:\mathbb R^{2n}\rightarrow\mathbb R_+$ depending analytically on the time-variable $t\in\mathbb R$ and a family $(U_t)_{t\in\mathbb R}$ of metaplectic operators such that
$$\forall t\geq0,\quad e^{-tq^w} = e^{-ta_t^w}U_t.$$
Moreover, there exists a positive constant $T>0$ and a family $(b_t)_{-T<t< T}$ of real-valued quadratic forms $b_t:\mathbb R^{2n}\rightarrow\mathbb R$ also depending analytically on the time-variable $-T<t< T$, such that 
$$\forall t\in[0,T),\quad e^{-tq^w} = e^{-ta_t^w}e^{-itb^w_t}.$$
\end{thm}

We refer the reader to Definition \ref{17052019D1} in the appendix where the metaplectic operators (and more generally the Fourier integral operators associated with nonnegative complex symplectic transformations) are defined.

The principal application of this decomposition will be to describe the regularizing effects of the semigroup $(e^{-tq^w})_{t\geq0}$, which requires a precise knowledge of the selfadjoint part $e^{-ta^w_t}$ given by Theorem \ref{12102018T1}. More precisely, we will need an estimate from below of the time-dependent quadratic form $a_t$. This is the purpose of the following theorem:

\begin{thm}\label{14102018T2} Let $q:\mathbb R^{2n}\rightarrow\mathbb C$ be a complex-valued quadratic form with a nonnegative real part $\Reelle q\geq0$. We consider $F$ the Hamilton map of $q$ and $S$ its singular space. Let $(a_t)_{t\in\mathbb R}$ be the family of nonnegative quadratic forms given by Theorem \ref{12102018T1}. Then, there exist some positive constants $0<T<1$ and $c>0$ such that for all $0\le t\le T$ and $X\in\mathbb R^{2n}$,
\begin{equation}
\label{pour le coup cest bien d'avoir un numero car les gens risquent de me citer}
a_t(X)\geq c\sum_{j=0}^{k_0} t^{2j} \Reelle q\big((\Imag F)^j X\big),
\end{equation}
where $0\le k_0\le 2n-1$ is the smallest integer such that \eqref{12102018E7} holds.
\end{thm}

Theorem \ref{14102018T2} implies in particular that for all $0\le t\ll1$, the quadratic form $a_t$ enjoys degenerate anisotropic coercive estimates in the phase space. This corollary is proven in  Lemma \ref{23102018L1}. In the particular case when $S=\{0\}$, this lemma implies that the quadratic form $a_t$ is positive definite for all $0\le t\ll1$.  Moreover, it highlights the role of the singular space $S$ in the polar decomposition given by Theorem \ref{12102018T1} through the index $0\le k_0\le 2n-1$ which is intrinsically related to its structure.

The calculation of the quadratic forms $a_t$ and $b_t$ is quite difficult in practice (except for example for the Ornstein-Uhlenbeck operators see e.g.\ \cite{AB}). The Kramers-Fokker-Planck operator without external potential also makes an exception as illustrated in the following example:

\begin{ex} Let $K$ be the Kramers-Fokker-Planck operator without external potential defined by
\begin{equation}\label{10012019E5}
	K = -\Delta_v + \vert v\vert^2 + \langle v,\nabla_x\rangle,\quad (x,v)\in\mathbb R^{2n},
\end{equation}
and equipped with the domain $D(K) = \{u\in L^2(\mathbb R^{2n}) : Ku\in L^2(\mathbb R^{2n})\}$. The operator $K$ is quadratic since its Weyl symbol is the quadratic form $q:\mathbb R^{4n}\rightarrow\mathbb C$ given by $q(x,v,\xi,\eta) = \vert\eta\vert^2 + \vert v\vert^2 + i\langle v,\xi\rangle,$ with $(x,v,\xi,\eta)\in\mathbb R^{4n}$. Moreover, for all $t\geq0$, the evolution operator $e^{-tK}$ can be written as 
\begin{equation}\label{14022019E2}
	e^{-tK} = e^{-ta^w_t}e^{-itb^w_t},
\end{equation}
where the time-dependent quadratic operators $a^w_t$ and $b^w_t$ are defined for all $t\geq0$ by
$$a^w_t = -\Delta_v + \vert v\vert^2 - \frac{\sinh(2t)}{\cosh(2t) + 1}\langle\nabla_x,\nabla_v\rangle - \frac{2t\cosh(2t) - \sinh(2t)}{4t(\cosh(2t) + 1)}\Delta_x\quad \text{and}\quad  b^w_t = \frac{\tanh t}{it}\langle v,\nabla_x\rangle.$$
Indeed, as we will see in the proof of Theorem \ref{12102018T1}, establishing the relation \eqref{14022019E2} is equivalent to proving the following equality between matrices:
\begin{equation}\label{14022019E3}
	e^{-2itJQ} = e^{-2itJA_t}e^{2tJB_t},
\end{equation}
where $J\in\Symp_{4n}(\mathbb R)$ is the symplectic matrix defined in \eqref{08112018E1}, $Q\in\Symm_{4n}(\mathbb C)$ is the matrix of the quadratic form $q$ in the canonical basis of $\mathbb R^{4n}$, and the time-dependent matrices $A_t,B_t\in\Symm_{4n}(\mathbb R)$ are respectively defined for all $t\geq0$ by
$$A_t = \begin{pmatrix}
	0_n & 0_n & 0_n & 0_n \\
	0_n & I_n & 0_n & 0_n \\
	0_n & 0_n & \frac{2t\cosh(2t) - \sinh(2t)}{4t(\cosh(2t) + 1)}I_n & \frac{\sinh(2t)}{2(\cosh(2t) + 1)}I_n \\
	0_n & 0_n & \frac{\sinh(2t)}{2(\cosh(2t) + 1)}I_n & I_n
\end{pmatrix}
\
B_t = \begin{pmatrix}
	0_n & 0_n & 0_n & 0_n \\
	0_n & 0_n & \frac{\tanh t}{2t}I_n & 0_n \\
	0_n & \frac{\tanh t}{2t}I_n & 0_n & 0_n \\
	0_n & 0_n & 0_n & 0_n
\end{pmatrix}.$$
Moreover, \eqref{14022019E3} follows from a direct calculus.
\end{ex}

\begin{rk} The technics used to derive the polar decompositions of semigroups generated by accretive non-selfadjoint quadratic differential operators can also be used to obtain other splitting formulas. For example, let us consider the harmonic oscillator $\mathcal H = -\Delta_x+\vert x\vert^2$, with $x\in\mathbb R^n$. We prove in Proposition \ref{11022019P1} (in dimension $1$, but the proof works the same in any dimension by tensorization) with the same arguments as the ones used in the proof of Theorem \ref{12102018T1} that for all $t\geq0$, the evolution operator $e^{-t\mathcal H}$ generated by $\mathcal H$ is written as
$$e^{-t\mathcal H} = e^{-\frac12(\tanh t)\vert x\vert^2}e^{\frac12\sinh(2t)\Delta_x}e^{-\frac12(\tanh t)\vert x\vert^2}.$$
The method can be generally used for all semigroups generated by accretive non-selfadjoint quadratic differential operators.
\end{rk}

\begin{rk}\label{exfrac} The polar decomposition provided by Theorem \ref{12102018T1} for the semigroups generated by accretive non-selfadjoint quadratic differential operators is as well valid for an other general class of semigroups called fractional Ornstein-Uhlenbeck semigroups defined as follows: given $s>0$ a positive real number, $B$ and $Q$ real $n\times n$ matrices, with $Q$ symmetric positive semidefinite, we define the fractional Ornstein-Uhlenbeck operator $L_s$ as
$$L_s = \frac12\Tr^s(-Q\nabla^2_x) + \langle Bx,\nabla_x\rangle,$$
and equipped with the domain $D(L_s) = \{u\in L^2(\mathbb R^n) : L_su\in L^2(\mathbb R^n)\}.$ The operator $\Tr^s(-Q\nabla^2_x)$ stands for the Fourier multiplier with symbol $\langle Q\xi,\xi\rangle^s$. The two authors proved in \cite[Theorem 1.1]{AB} that the operator $L_s$ generates a strongly continuous semigroup $(e^{-tL_s})_{t\geq0}$ on $L^2(\mathbb R^n)$ and that for all $t\geq0$, the evolution operator $e^{-tL_s}$ is explicitly given by the following formula:
\begin{equation}\label{01022019E2}
	\forall t\geq0,\quad e^{-tL_s} = \exp\Big(-\frac12\int_0^t\big\vert\sqrt Qe^{\tau B^T}D_x\big\vert^{2s}\ \mathrm d\tau\Big)e^{-t\langle Bx,\nabla_x\rangle}.
\end{equation}
For all $t\geq0$, the relation \eqref{01022019E2} is the polar decomposition of the operator $e^{-tL_s}$.
\end{rk}

\subsection{Regularizing effects of semigroups generated by accretive non-selfadjoint quadratic differential operators} As an application of the splitting formula given by Theorem \ref{12102018T1} and the estimate given by Theorem \ref{14102018T2}, we  investigate the regularizing properties of the evolution operators $e^{-tq^w}$ for all $t\geq0$. As pointed out in the works \cite{A, MR2507625, MR3710672, MR3841852, MR3756858}, the understanding of this smoothing effect is closely related to the structure of the singular space $S$. Indeed, the notion of singular space allows to study the propagation of Gabor singularities for the solutions of the quadratic differential equations
$$\left\{\begin{array}{l}
	\partial_tu + q^w(x,D_x)u = 0, \\[5pt]
	u(0) = u_0\in L^2(\mathbb R^n).
\end{array}\right.$$
We recall from \cite[Section 5]{MR3880300} that the Gabor wave front set $WF(u)$ of a tempered distribution $u\in\mathscr S'(\mathbb R^n)$ measures the directions in the phase space in which a tempered distribution does not behave like a Schwartz function. In particular, when $u\in\mathscr S'(\mathbb R^n)$, its Gabor wave front set $WF(u)$ is empty if and only if $u\in\mathscr S(\mathbb R^n)$. The following microlocal inclusion was proven in \cite[Theorem 6.2]{MR3756858}:
\begin{equation}\label{25042018E14}
	\forall u\in L^2(\mathbb R^n),\forall t>0,\quad WF(e^{-tq^w}u)\subset e^{tH_{\Imag q}}(WF(u)\cap S)\subset S,
\end{equation}
where $(e^{tH_{\Imag q}})_{t\in\mathbb{R}}$ is the flow generated  by the Hamilton vector field associated with the imaginary part of the quadratic form $q$, $H_{\Imag q} = (\partial_{\xi}\Imag q)\cdot\partial_x-(\partial_x\Imag q)\cdot\partial_{\xi}$. This result points out that the possible Gabor singularities of the solution $e^{-tq^w}u$ can only come from Gabor singularities of the initial datum $u$ localized in the singular space $S$ and are propagated along the curves given by the flow of the Hamilton vector field $H_{\Imag q}$ associated with the imaginary part of the symbol. The microlocal inclusion \eqref{25042018E14} was shown to hold as well for other types of wave front sets, as Gelfand-Shilov wave front sets \cite{MR3649471} or polynomial phase space wave front sets \cite{MR3767155}. 

Drawing our inspiration from the work \cite{MR3710672}, we consider the vector subspaces $V_0,\ldots,V_{k_0}\subset\mathbb R^{2n}$ defined by
\begin{equation}\label{12102018E8}
	V_k = \bigcap_{j=0}^k\Ker(\Reelle F(\Imag F)^j)\cap\mathbb R^{2n},\quad 0\le k\le k_0,
\end{equation}
where $0\le k_0\le 2n-1$ is the smallest integer such that \eqref{12102018E7} holds. According to \eqref{12102018E7}, the family of vector subspaces $V_0^{\perp},\ldots,V^{\perp}_{k_0}$ is increasing for the inclusion and satisfies
\begin{equation}\label{12102018E9}
	V_0^{\perp}\subsetneq\ldots\subsetneq V_{k_0}^{\perp} = S^{\perp},
\end{equation}
where the orthogonality is taken with respect to the canonical Euclidean structure of $\mathbb R^{2n}$. This stratification allows one to define the index with respect to the singular space of any point $X_0\in S^{\perp}$ as
\begin{equation}\label{23012019E4}
	k_{X_0} = \min\big\{0\le k\le k_0 : X_0\in V^{\perp}_k\big\}.
\end{equation}
When the singular space of $q$ is reduced to zero $S=\{0\}$, the microlocal inclusion \eqref{25042018E14} implies that the semigroup $(e^{-tq^w})_{t\geq0}$ is smoothing in any positive time $t>0$ in the Schwartz space $\mathscr S(\mathbb R^n)$, but this result does not provide any control of the blow-up of the associated seminorms as $t\rightarrow0^+$. However, the notion of index was shown in \cite{MR3710672} to allow to determine the short-time asymptotics of the regularizing effect induced by the semigroup $(e^{-tq^w})_{t\geq0}$ in the phase space direction given by the vector $X_0\in\mathbb R^{2n}$. More precisely, \cite[Theorem 1.1]{MR3710672} states that when the singular space is trivial $S=\{0\}$, there exists a positive constant $C>1$ such that for all $X_0\in\mathbb R^{2n} = S^{\perp}$, $0< t\le 1$ and $u\in L^2(\mathbb R^n)$,
\begin{equation}\label{28012019E1}
	\big\Vert\langle X_0,X\rangle^we^{-tq^w}u\big\Vert_{L^2(\mathbb R^n)}\le \frac{C\vert X_0\vert}{t^{k_{X_0}+\frac12}}\ \Vert u\Vert_{L^2(\mathbb R^n)},
\end{equation}
where $0\le k_{X_0}\le k_0$ denotes the index of the point $X_0\in\mathbb R^{2n}=S^{\perp}$ with respect to the singular space and where the pseudodifferential operator $\langle X_0,X\rangle^w$ is defined as the differential operator whose Weyl symbol is given by the linear form $\langle X_0,X\rangle$, that is
\begin{equation}\label{23012019E5}
	\langle X_0,X\rangle^w = \langle x_0,x\rangle + \langle\xi_0,D_x\rangle,\quad X_0 = (x_0,\xi_0)\in\mathbb R^{2n}.
\end{equation}
This result shows that the structure of the singular space accounting for the family of vector subspaces $(V_k)_{0\le k\le k_0}$, allows one to sharply describe the short-time asymptotics of the regularizing effect induced by the semigroup $(e^{-tq^w})_{t\geq0}$. The degeneracy degree of the phase space direction $X_0\in\mathbb R^{2n} = S^{\perp}$ given by the index with respect to the singular space directly accounts for the blow-up upper bound $t^{-k_{X_0} -\frac12}$, for small times $t\rightarrow0^+$. As a corollary, the same three authors proved in \cite[Corollary 1.2]{MR3710672} that still under the assumption $S=\{0\}$, there exists a positive constant $C>1$ such that for all $m\geq1$ and $X_1,\ldots,X_m\in\mathbb R^{2n} = S^{\perp}$, $0< t\le 1$ and $u\in L^2(\mathbb R^n)$,
\begin{equation}\label{25012019E1}
	\big\Vert\langle X_1,X\rangle^w\ldots\langle X_m,X\rangle^we^{-tq^w}u\big\Vert_{L^2(\mathbb R^n)}
	\le \frac{C^m}{t^{(k_0+\frac12)m}}\ \Bigg[\prod_{j=1}^m\vert X_j\vert\Bigg]\ (m!)^{k_0+\frac12}\ \Vert u\Vert_{L^2(\mathbb R^n)}.
\end{equation}
This implies in particular that when $S=\{0\}$, the semigroup $(e^{-tq^w})_{t\geq0}$ is smoothing in any positive time $t>0$ in the Gelfand-Shilov space $S^{k_0+1/2}_{k_0+1/2}(\mathbb R^n)$. We recall that when $\mu$ and $\nu$ are two positive real numbers satisfying $\mu+\nu\geq1$, the Gelfand-Shilov space $S^{\mu}_{\nu}(\mathbb R^n)$ consists in all the Schwartz functions $f\in\mathscr S(\mathbb R^n)$ satisfying that 
$$\exists C>1,\forall(\alpha,\beta)\in\mathbb N^{2n},\quad \big\Vert x^{\alpha}\partial^{\beta}_xf(x)\big\Vert_{L^2(\mathbb R^n)}
\le C^{1+\vert\alpha\vert+\vert\beta\vert}\ (\alpha!)^{\nu}\ (\beta!)^{\mu}.$$
We refer to \cite[Chapter 6]{MR2668420} for an extensive discussion about the Gelfand-Shilov spaces. This result was sharpened by the same three authors in \cite[Theorem 1.2]{MR3841852} with a different approach based on FBI technics, where they proved that $S=\{0\}$ implies that the semigroup $(e^{-tq^w})_{t\geq0}$ is actually smoothing in any positive time $t>0$ in the Gelfand-Shilov space $S^{1/2}_{1/2}(\mathbb R^n)$ with a control of the blow-up of the associated seminorms in the asymptotics $t\rightarrow0^+$. Moreover, estimates similar to \eqref{25012019E1} in the asymptotics $t\rightarrow+\infty$ were obtained in the case where $S=\{0\}$, see again Theorem 1.1 and Corollary 1.2 in \cite{MR3710672}. We also refer the reader to \cite{MR2899986, MR2390087} where quadratic semigroups are studied in long-time asymptotics.

On the other hand, when the singular space $S$ of $q$ is possibly non-zero but still has a symplectic structure, that is, when the restriction of the canonical symplectic form to the singular space $\sigma_{\vert S}$ is non-degenerate, the above result can be easily extended but only when differentiating the semigroup in the directions of the phase space given by the symplectic orthogonal complement of the singular space 
$$S^{\sigma\perp} = \big\{X\in\mathbb R^{2n} : \forall Y\in S,\quad \sigma(X,Y) = 0\big\}.$$
Indeed, when the singular space $S$ has a symplectic structure, it was proven in \cite[Subsection 2.5]{MR3710672} that the quadratic form $q$ can be written as $q = q_1 + q_2$ with $q_1$ a purely imaginary-valued quadratic form defined on $S$ and $q_2$ another one defined on $S^{\sigma\perp}$ with a nonnegative real part and a zero singular space. The symplectic structures of $S$ and $S^{\sigma\perp}$ imply that the operators $q^w_1(x,D_x)$ and $q^w_2(x,D_x)$ commute as well as their associated semigroups
$$\forall t>0,\quad e^{-tq^w} = e^{-tq^w_1}e^{-tq^w_2} = e^{-tq^w_2}e^{-tq^w_1}.$$
Moreover, since $\Reelle q_1=0$, $(e^{-tq^w_1})_{t\geq0}$ is a contraction semigroup on $L^2(\mathbb{R}^n)$ and the partial smoothing properties of the semigroup $(e^{-tq^w})_{t\geq0}$ can be deduced from a symplectic change of variables and the result known for zero singular spaces applied to the semigroup $(e^{-tq^w_2})_{t\geq0}$. We refer the reader to \cite[Subsection 2.5]{MR3710672} for more details about the reduction by tensorization of the non-zero symplectic case to the case when the singular space is zero. 

In the case when the singular space $S$ is not necessary trivial nor symplectic but is spanned by elements of the canonical basis of $\mathbb R^{2n}$ satisfies the condition $S\subset\Ker(\Imag F)$, with $F$ the Hamilton map of the quadratic form $q$, some partial Gelfand-Shilov smoothing effects in any positive time $t>0$ for the semigroup $(e^{-tq^w})_{t\geq0}$ were obtained by the first author in \cite[Theorem 1.4]{A}, with some control of the associated seminorms as $t\rightarrow0^+$. Moreover, we mention that the two authors, in \cite[Theorem 1.2]{AB}, described the regularizing effects of the Ornstein-Uhlenbeck operator,  whose singular space is not symplectic nor satisfies the condition $S\subset\Ker(\Imag F)$.

In this paper, we investigate the smoothing properties of the evolution operators $e^{-tq^w}$ for any positive times $t>0$, and we aim at sharpening and generalizing the estimates \eqref{25012019E1} without making any assumptions on the singular space $S$. As in the work \cite{MR3710672}, the notion of index plays a key role in understanding the blow-up of the seminorms associated with the smoothing effects of the semigroup $(e^{-tq^w})_{t\geq0}$:

\begin{thm}\label{22102018T1} Let $q:\mathbb R^{2n}\rightarrow\mathbb C$ be a complex-valued quadratic form with a nonnegative real part $\Reelle q\geq0$. We consider $S$ the singular space of $q$ and $0\le k_0\le 2n-1$ the smallest integer such that \eqref{12102018E7} holds. Then, there exist some positive constants $c>1$ and $t_0>0$ such that for all $m\geq1$, $X_1,\ldots,X_m\in S^{\perp}$, $0<t<t_0$ and $u\in L^2(\mathbb R^n)$,
$$\big\Vert\langle X_1,X\rangle^w\ldots\langle X_m,X\rangle^we^{-tq^w}u\big\Vert_{L^2(\mathbb R^n)}
\le \frac{c^m}{t^{k_{X_1}+\ldots+k_{X_m}+\frac m2}}\ \Bigg[\prod_{j=1}^m\vert X_j\vert\Bigg]\ (m!)^{\frac12}\ \Vert u\Vert_{L^2(\mathbb R^n)},$$
where $0\le k_{X_j}\le k_0$ stands for the index of the point $X_j\in S^{\perp}$ with respect to the singular space.
\end{thm}

In the case when $m=1$, Theorem \ref{22102018T1} recovers the estimate \eqref{28012019E1}. The short-time asymptotics given by \eqref{25012019E1} of $m$ differentiations of the semigroup $(e^{-tq^w})_{t\geq0}$, as for it, is sharpened in $\mathcal O(t^{-k_{X_1}-\ldots k_{X_m} - \frac m2})$, which was the bound conjectured by the three authors of \cite{MR3710672} in page $622$. This result discloses that these short-time asymptotics depend on the phase space directions of differentiations. Moreover, the power over $(m!)^{k_0+\frac12}$ is sharpened in $(m!)^{\frac12}$, which in particular allows one to recover the Gelfand-Shilov $S^{1/2}_{1/2}(\mathbb R^n)$ regularizing effect of the semigroup $(e^{-tq^w})_{t\geq0}$ in any positive time $t>0$ when $S=\{0\}$ already established in \cite[Theorem 1.2]{MR3841852}, with now a precise control in short-time of the associated seminorms.

\begin{ex}\label{ex} Let $Q,R$ and $B$ be real $n\times n$ matrices, $Q$ and $R$ being symmetric positive semidefinite. We consider the generalized Ornstein-Uhlenbeck operator
\begin{equation}\label{15012019E5}
	P =-\frac12\Tr(Q\nabla^2_x) + \frac12\langle Rx,x\rangle + \langle Bx,\nabla_x\rangle,
\end{equation}
equipped with the domain $D(P) = \{u\in L^2(\mathbb R^n) : Pu\in L^2(\mathbb R^n)\}$. Notice that $P$ is a pseudodifferential operator whose Weyl symbol $p$ is given by
\begin{equation}\label{15012019E9}
	p(x,\xi) = \frac12\langle Q\xi,\xi\rangle + \frac12\langle Rx,x\rangle + i\langle Bx,\xi\rangle - \frac12\Tr(B).
\end{equation}
The operator $\tilde P = P+\frac12\Tr(B)$ is therefore a quadratic operator and it follows from a straightforward computation, see e.g.\ \cite[Section 5]{A}, that the Hamilton map $F$ and the singular space $S$ of $\tilde P$ are respectively given by
$$F = \frac12\begin{pmatrix}
	iB & Q \\
	-R & -iB^T
\end{pmatrix}\quad\text{and}\quad S=\bigcap_{j=0}^{n-1}\big(\Ker(RB^j)\times\Ker(Q(B^T)^j)\big).$$
We can consider $0\le k_0\le n-1$ the smallest integer such that $S$ is written as
\begin{equation}\label{28012019E4}
	S = \bigcap_{j=0}^{k_0}\big(\Ker(RB^j)\times\Ker(Q(B^T)^j)\big).
\end{equation}
We notice that the singular space of $\tilde P$ has a decoupled structure in the phase space in the sense that $S$ is written as the cartesian product $S=S_x\times S_{\xi}$, where the two vector subspaces $S_x\subset\mathbb R^n_x$ and $S_{\xi}\subset\mathbb R^n_{\xi}$ are respectively defined by
$$S_x = \bigcap_{j=0}^{k_0}\Ker(RB^j)\subset\mathbb R^n_x\quad\text{and}\quad S_{\xi} = \bigcap_{j=0}^{k_0}\Ker(Q(B^T)^j)\subset\mathbb R^n_{\xi}.$$
For all $x\in S^{\perp}_x$ and $\xi\in S^{\perp}_{\xi}$, we can define the indexes $0\le k_x\le k_0$ and $0\le k_{\xi}\le k_0$ of the points $x$ and $\xi$ with respect to the spaces $S_x$ and $S_{\xi}$ respectively by
$$k_x = \min\bigg\{0\le k\le k_0 : x\in\bigg(\bigcap_{j=0}^k\Ker(RB^j)\bigg)^{\perp}\bigg\},$$
and
$$k_{\xi} = \min\bigg\{0\le k\le k_0 : \xi\in\bigg(\bigcap_{j=0}^k\Ker(Q(B^T)^j)\bigg)^{\perp}\bigg\}.$$
Notice that the integer $k_x$ (resp. $k_{\xi}$) coincides with the index of the point $(x,0)\in S^{\perp}_x\times\{0\}\subset S^{\perp}$ (resp. of the point $(0,\xi)\in\{0\}\times S^{\perp}_{\xi}\subset S^{\perp}$) with respect to the singular space. Theorem \ref{22102018T1} implies in particular that there exist some positive constants $c>1$ and $t_0>0$ such that for all $m,p\geq0$, $x_1,\ldots,x_m\in S^{\perp}_x$, $\xi_1,\ldots,\xi_p\in S^{\perp}_{\xi}$, $0<t<t_0$ and $u\in L^2(\mathbb R^n)$,
\begin{multline}\label{30012019E1}
	\big\Vert\langle x_1,x\rangle\ldots\langle x_m,x\rangle\langle \xi_1,\nabla_x\rangle\ldots\langle \xi_p,\nabla_x\rangle e^{-tP}u\big\Vert_{L^2(\mathbb R^n)} \\
\le \frac{c^{1+m+p}}{t^{k_{x_1}+\ldots+k_{x_m}+k_{\xi_1}+\ldots+k_{\xi_p}+\frac m2+\frac p2}}\ \Bigg[\prod_{j=1}^m\vert x_j\vert\Bigg]\ \Bigg[\prod_{j=1}^p\vert\xi_j\vert\Bigg]\ (m!)^{\frac12}\ (p!)^{\frac12}\ \Vert u\Vert_{L^2(\mathbb R^n)},
\end{multline}
where the integers $0\le k_{x_j}\le k_0$ (resp. $0\le k_{\xi_j}\le k_0$) denote the indexes of the points $x_j$ (resp. $\xi_j$) with respect to $S_x$ (resp. $S_{\xi}$). This proves that the semigroup $(e^{-tP})_{t\geq0}$ enjoys partial Gelfand-Shilov regularity in any positive time $t>0$.
\end{ex}

Theorem \ref{22102018T1} implies in particular that for all $X_0\in S^{\perp}$ and $t>0$, the linear operator $\langle X_0,X\rangle^we^{-tq^w}$ is bounded on $L^2(\mathbb R^n)$. In fact, the reciprocal assertion also holds as shown in the following theorem:

\begin{thm}\label{10122018T1} Let $q:\mathbb R^{2n}\rightarrow\mathbb C$ be a complex-valued quadratic form with a nonnegative real part $\Reelle q\geq0$. We consider $S$ the singular space of $q$. If there exist $t>0$ and $X_0\in\mathbb R^{2n}$ such that the linear operator $\langle X_0,X\rangle^we^{-tq^w}$ is bounded on $L^2(\mathbb R^n)$, then $X_0\in S^{\perp}$.
\end{thm}

Notice that if $t>0$ and $X_0\in \mathbb R^{2n}$ are such that the operator $\langle X_0,X\rangle^we^{-tq^w}$ is bounded on $L^2(\mathbb R^n)$, then $X_0\in S^{\perp}$ according to Theorem \ref{10122018T1} and then Theorem \ref{22102018T1} can be applied to obtain that for all $m\geq1$, the operators $(\langle X_0,X\rangle^w)^me^{-tq^w}$ are also bounded on $L^2(\mathbb R^n)$. 

\begin{rk}
\label{rk_control} In the study of the null-controllability of quadratic differential equations, a key ingredient is to obtain some dissipation estimates for the semigroup $(e^{-tq^w})_{t\geq0}$ in order to use a Lebeau-Robbiano strategy, see e.g.\ \cite{A, AB, BEPS, BJKPS, MR3732691, MR3816981}. The regularizing effects given by Theorem \ref{22102018T1} allow to give a sufficient geometric condition on the singular space $S$ of $q$ so that such dissipation estimates hold. More precisely, let $\pi_k:L^2(\mathbb R^n)\rightarrow E_k$ be the frequency cutoff projection defined as the orthogonal projection onto the vector subspace $E_k\subset L^2(\mathbb R^n)$ given by $E_k = \{u\in L^2(\mathbb R^n) : \Supp\widehat u\subset[-k,k]^n\}$, with $k\geq1$ a positive integer. It can be proven while using Theorem \ref{22102018T1} and the strategy used in \cite[Section 4.2]{A}, that when the singular space $S$ of $q$ takes the form $S = \Sigma\times\{0_{\mathbb R^n_{\xi}}\}$, with $\Sigma\subset\mathbb R^n_x$ a vector subspace, there exist some positive constants $c_1,c_2>0$ and $0<t_0<1$ such that for all $k\geq1$, $0<t<t_0$ and $u\in L^2(\mathbb R^n)$,
\begin{equation}\label{28012019E2}
	\big\Vert(1-\pi_k)e^{-tq^w}u\big\Vert_{L^2(\mathbb R^n)}\le c_1e^{-c_2t^{2k_0+1}k^2}\Vert u\Vert_{L^2(\mathbb R^n)}.
\end{equation}
When the singular space of $q$ is reduced to zero $S=\{0\}$, dissipative estimates similar to \eqref{28012019E2} were obtained with $\pi_k$ some cutoff projections with respect to the Hermite basis of $L^2(\mathbb R^n)$, see e.g.\ \cite{BJKPS, MR3732691}.
\end{rk}

\subsection{Subelliptic estimates enjoyed by quadratic operators} Finally, we study the subelliptic estimates enjoyed by accretive non-selfadjoint quadratic differential operators. When the singular space of the quadratic form $q$ is reduced to zero $S = \{0\}$, K. Pravda-Starov proved in \cite{MR2752935} that the quadratic operator $q^w(x,D_x)$ satisfies specific global subelliptic estimates with a loss of derivatives with respect to the elliptic case directly depending on the structural parameter of the singular space $0\le k_0\le 2n-1$ defined in \eqref{12102018E7}. More precisely, \cite[Theorem 1.2.1]{MR2752935} states that when the singular space is equal to zero $S=\{0\}$, there exists a positive constant $c>0$ such that for all $u\in D(q^w)$,
\begin{equation}\label{29012019E1}
	\big\Vert\langle(x,D_x)\rangle^{\frac2{2k_0+1}}u\big\Vert_{L^2(\mathbb R^n)}\le c\big[\Vert q^w(x,D_x)u\Vert_{L^2(\mathbb R^n)} + \Vert u\Vert_{L^2(\mathbb R^n)}\big],
\end{equation}
where $0\le k_0\le 2n-1$ is the smallest integer such that \eqref{12102018E7} holds, with 
$$\langle(x,D_x)\rangle^{\frac2{2k_0+1}} = (1+x^2+D^2_x)^{\frac1{2k_0+1}},$$
being the operator defined by the functional calculus of the harmonic oscillator. The estimate \eqref{29012019E1} was first proven in \cite{MR2752935} with a technical multiplier method, and recovered in the two papers \cite[Theorem 1.1]{MR3841852} and \cite[Corollary 1.3]{MR3710672} respectively by using techniques of FBI transforms and the interpolation theory. Moreover, the three authors of \cite{MR3710672} and \cite{MR3841852} sharpened this result by improving it in the directions of the phase space which are less degenerate, that is with smaller indices with respect to the singular space. In order to recall their result, we need to consider the following quadratic forms
\begin{equation}\label{26102018E1}
	p_k(X) = \sum_{j=0}^k\Reelle q\big((\Imag F)^jX\big),\quad 0\le k\le k_0,
\end{equation}
where $0\le k_0\le 2n-1$ is the smallest integer such that \eqref{12102018E7} holds. We also consider the quadratic operators $\Lambda^2_k$ defined for all $0\le k\le k_0$ by
\begin{equation}\label{25102018E12}
	\Lambda_k^2 = 1 + p^w_k(x,D_x),
\end{equation}
and equipped with the domains $D(\Lambda_k^2) = \{u\in L^2(\mathbb R^n) : \Lambda_k^2u\in L^2(\mathbb R^n)\}.$
Since $\Reelle q\geq0$ is a nonnegative quadratic form, it can be proven by using for example Lemma~\ref{06072018L1} that the operators $\Lambda^2_k$ are positive and as a consequence, we can consider the fractional powers of those operators. When the singular space $S$ of $q$ is reduced to zero, Theorem 1.4 in \cite{MR3710672} states that there exists a positive constant $c>0$ such that for all $u\in D(q^w)$,
\begin{equation}\label{29012019E2}
	\big\Vert\Lambda_0u\big\Vert_{L^2(\mathbb R^n)} + \sum_{k=1}^{k_0}\big\Vert\Lambda^{\frac2{2k+1}}_ku\big\Vert_{L^2(\mathbb R^n)}\le c\big[\Vert q^w(x,D_x)u\Vert_{L^2(\mathbb R^n)} + \Vert u\Vert_{L^2(\mathbb R^n)}\big].
\end{equation}
The authors of \cite{MR3710672} expected the powers $2/(2k+1)$ over the operators $\Lambda_k$ to be sharp but also expected the power over the term $\Lambda_0$ to be equal to $2$ and not to $1$.

No general theory has been developed when the singular space $S$ is not necessarily equal to zero. However, let us mention that some subelliptic estimates were obtained for the Kramers-Fokker-Planck operator without external potential $K$ defined in \eqref{10012019E5} by F. H\'erau and K. Pravda-Starov in \cite[Proposition 2.1]{MR2786222} with a multiplier method and for the Ornstein-Uhlenbeck operator (under the Kalman rank condition) by the two authors in \cite[Corollary 1.15]{AB} while using the interpolation theory as in the work \cite{MR3710672}.

In this paper, we aim at extending and sharpening the subelliptic estimates \eqref{29012019E2} to all quadratic forms $q:\mathbb R^{2n}\rightarrow\mathbb C$ with nonnegative real parts $\Reelle q\geq0$, without making any assumptions on their singular spaces $S$.

\begin{thm}\label{22102018T2} Let $q:\mathbb R^{2n}\rightarrow\mathbb C$ be a complex-valued quadratic form with a nonnegative real part $\Reelle q\geq0$. We consider $S$ the singular space of $q$ and $0\le k_0\le 2n-1$ the smallest integer such that \eqref{12102018E7} holds. Then, there exists a positive constant $c>0$ such that for all $u\in D(q^w)$,
$$\sum_{k=0}^{k_0}\big\Vert\Lambda^{\frac2{2k+1}}_ku\big\Vert_{L^2(\mathbb R^n)}\le c\big[\Vert q^w(x,D_x)u\Vert_{L^2(\mathbb R^n)} + \Vert u\Vert_{L^2(\mathbb R^n)}\big].$$
\end{thm}

As in the case when the singular space is trivial, this result shows that the quadratic operator $q^w(x,D_x)$ enjoys anisotropic subelliptic estimates, this anisotropy being directly related to the structure \eqref{12102018E7} of the singular space $S$. Moreover, Theorem \ref{22102018T2} confirms that the power over the operator $\Lambda_0$ associated with the real part of the quadratic form $q$ is actually equal to $2$.

\begin{ex} Let $P$ be the generalized Ornstein-Uhlenbeck operator defined in \eqref{15012019E5}. It follows from a straightforward calculation that for all $0\le k\le k_0$, the operator $\Lambda^2_k$ associated with the quadratic operator $P+\frac12\Tr(B)$ is given by
$$\Lambda^2_k = 1 + \sum_{j=0}^k\frac1{2^{j+1}}\big\vert\sqrt RB^jx\big\vert^2 + \sum_{j=0}^k\frac1{2^{j+1}}\big\vert\sqrt Q(B^T)^jD_x\big\vert^2,$$
where $0\le k_0\le n-1$ is the smallest integer such that \eqref{28012019E4} holds. It therefore follows from Theorem \ref{22102018T2} that there exists a positive constant $c>0$ such that for all $0\le k\le k_0$ and $u\in D(P)$,
$$\Big\Vert\Big(1 + \sum_{j=0}^k\frac1{2^{j+1}}\big\vert\sqrt RB^jx\big\vert^2 + \sum_{j=0}^k\frac1{2^{j+1}}\big\vert\sqrt Q(B^T)^jD_x\big\vert^2\Big)^{\frac1{2k+1}}u\Big\Vert_{L^2(\mathbb R^n)}
\le c\big[\Vert Pu\Vert_{L^2(\mathbb R^n)} + \Vert u\Vert_{L^2(\mathbb R^n)}\big].$$
\end{ex}

\section{Splitting of semigroups generated by non-selfadjoint quadratic differential operators}\label{Polar}

This section is devoted to the proof of Theorem \ref{12102018T1}.  Let $q:\mathbb R^{2n}\rightarrow\mathbb C$ be a complex-valued quadratic form with a nonnegative real part $\Reelle q\geq0$. We consider $Q\in\Symm_{2n}(\mathbb C)$ the matrix of $q$ in the canonical basis of $\mathbb R^{2n}$. We also consider  $J$ the symplectic matrix defined in \eqref{08112018E1}. Our goal is first to construct a family $(a_t)_{t\in\mathbb R}$ of nonnegative quadratic forms $a_t:\mathbb R^{2n}\rightarrow\mathbb R_+$ depending analytically on the time-variable $t\in\mathbb R$ and a family $(U_t)_{t\in\mathbb R}$ of metaplectic operators such that for all $t\geq0$,
\begin{equation}\label{14012019E1}
	e^{-tq^w} = e^{-ta_t^w}U_t,
\end{equation}
and then to prove that there exist a positive constant $T>0$ and a family $(b_t)_{-T<t<T}$ of real-valued quadratic forms $b_t:\mathbb R^{2n}\rightarrow\mathbb R$ also depending analytically on the time-variable $-T<t<T$, such that for all $0\le t<T$,
\begin{equation}\label{14012019E2}
	e^{-tq^w} = e^{-ta_t^w}e^{-itb^w_t}.
\end{equation}

To that end, we begin by establishing that proving \eqref{14012019E1} and \eqref{14012019E2} is actually equivalent to solving a finite-dimensional problem involving matrices. First of all, in order to give an intuition of this equivalence, let us formally prove that given some $t>0$, the factorization \eqref{14012019E2} is equivalent to the finite dimensional matrix relation
\begin{equation}\label{11022019E2}
	e^{-2itJQ} = e^{-2itJA_t} e^{2tJB_t},
\end{equation}
where $A_t$ (resp. $B_t$) is the matrix of the quadratic form $a_t$ (resp. $b_t$) in the canonical basis of $\mathbb R^{2n}$. The equivalence between \eqref{14012019E2} and \eqref{11022019E2}  will be justified rigorously shortly later with the theory of Fourier integral operators. By applying the Baker-Campbell-Hausdorff formula introduced in \cite{MR1575931} and \cite H, the relation \eqref{14012019E2} is formally equivalent to
\begin{equation}\label{11022019E3}
	-tq^w = \sum_{m=0}^{+\infty}\sum_{p\in\{a_t,ib_t\}^m}(\ad_{tp^w_1})\ldots(\ad_{tp^w_m})(\alpha_mta_t^w+\beta_mitb_t^w),
\end{equation}
where $\alpha_m,\beta_m\in\mathbb Q$ are explicit rational coefficients and
$$\ad_{\mathcal P_1}\mathcal P_2 := [\mathcal P_1,\mathcal P_2] = \mathcal P_1 \mathcal P_2 - \mathcal P_2\mathcal P_1,$$
denotes the commutator between the operators $\mathcal P_1$ and $\mathcal P_2$. However, if $q_1,q_2:\mathbb R^{2n}\rightarrow\mathbb C$ are two quadratic forms, elements of Weyl calculus, see e.g.\ \cite[Theorem 18.5.6]{MR2304165}, show that the commutator $[q^w_1,q^w_2]$ is also a differential operator given by
\begin{equation}\label{11022019E4}
	[q^w_1,q^w_2] = -i\{q_1,q_2\}^w, \quad\text{where}\quad \{q_1,q_2\} = \nabla_\xi q_1\cdot\nabla_xq_2 - \nabla_xq_1\cdot\nabla_\xi q_2.
\end{equation}
Note that $\{q_1,q_2\}$ is the canonical Poisson bracket between the quadratic forms $q_1$ and $q_2$. We therefore deduce that \eqref{11022019E3} is equivalent to the equality between quadratic forms
\begin{equation}\label{11022019E5}
	-tq = \sum_{m=0}^{+\infty}\sum_{p\in\{-ia_t,b_t\}^m}(\ad_{tp_1})\ldots(\ad_{tp_m})(\alpha_mta_t + \beta_mtib_t),
\end{equation}
where we set $\ad_{p_1}p_2 := \{p_1,p_2\}$. Moreover, we observe that if $q_1,q_2:\mathbb R^{2n}\rightarrow\mathbb C$ are two quadratic forms, the Hamilton map of the Poisson bracket $\{q_1,q_2\}$ is $-2[F_1,F_2]$, with $[F_1,F_2]$ the commutator of $F_1$ and $F_2$ the Hamilton maps of $q_1$ and $q_2$, see e.g.\ \cite[Lemma 3.2]{MR2390087}. As a consequence, we deduce while using \eqref{05112018E7} and multiplying by $2i$ that \eqref{11022019E5} is equivalent to the matrix relation
\begin{equation}\label{11022019E6}
	-2itJQ =  \sum_{m=0}^{+\infty}\sum_{P\in\{2iA_t,-2B_t\}^m}(\ad_{tJP_1})\ldots(\ad_{t JP_m})(\alpha_m 2itJA_t-\beta_m 2tJB_t).
\end{equation}
Thus, by applying once again the Baker-Campbell-Hausdorff formula, the relation \eqref{14012019E2} is equivalent to \eqref{11022019E2}. Obtaining the quadratic forms $a_t$ and $b_t$ is then far easier henceforth the equivalence between \eqref{14012019E2} and \eqref{11022019E2} is established. Indeed, let us check that the relation \eqref{11022019E2} is equivalent to the following triangular system
\begin{equation}\label{14022019E4}
\left\{\begin{array}{lll} 
	e^{-4itJA_t} &=& e^{-2itJQ}e^{-2itJ\overline Q}, \\[1pt]
	e^{2tJB_t} &=& e^{2itJA_t} e^{-2itJQ}.
\end{array}\right.
\end{equation}
Obviously, if \eqref{14022019E4} holds, then \eqref{11022019E2} is satisfied. On the other hand, when \eqref{11022019E2} holds, we observe that  
$$e^{-2itJQ}e^{-2itJ\overline Q}= e^{-2itJA_t}e^{2tJB_t}e^{-2tJB_t}e^{-2itJA_t}= e^{-4itJA_t}.$$
Moreover, the equality $e^{2tJB_t}=e^{2itJA_t}e^{-2itJQ}$ is only a rewriting of \eqref{11022019E2} and hence, \eqref{14022019E4} holds. The first equation of \eqref{14022019E4} will be solved for any time $t\in\mathbb R$ by using the holomorphic functional calculus. The second one will only be solved for short times $\vert t\vert\ll 1$.

In order to justify rigorously this reduction to a finite-dimensional problem, we shall use the Fourier integral operator representation of the evolution operators $e^{-tq^w}$ proven in \cite[Theorem 5.12]{MR1339714} and recalled in the following proposition:

\begin{prop}\label{13112018P2} Let $\tilde q:\mathbb R^{2n}\rightarrow \mathbb C$ be a complex-valued quadratic form with a nonnegative real part $\Reelle \tilde q\geq0$. Then, for all $t\geq0$, the evolution operator $e^{-t\tilde q^w} =  \mathscr K_{e^{-2itJ\tilde Q}}$ generated by the quadratic operator $\tilde q^w(x,D_x)$ is a Fourier integral operator whose kernel is a Gaussian distribution associated with the nonnegative complex symplectic linear bijection $e^{-2itJ\tilde Q}\in\Symp_{2n}(\mathbb C)$, with $\tilde Q\in\Symm_{2n}(\mathbb C)$ the matrix of $\tilde q$ with respect to the canonical basis of $\mathbb R^{2n}$.
\end{prop}

\noindent We refer the reader to Subsection \ref{FIO} in the appendix for the definition of the Fourier integral operators $\mathscr K_T$ and their basic properties, where $T$ is a nonnegative complex symplectic linear bijection in $\mathbb C^{2n}$. The key property satisfied by the operators $\mathscr K_T$ that we will need here is that if $T_1$ and $T_2$ are two nonnegative complex symplectic linear bijections in $\mathbb C^{2n}$, then $T_1T_2$ is also a nonnegative complex symplectic linear bijection and
\begin{equation}\label{11022019E7}
	\mathscr K_{T_1T_2} = \pm\mathscr K_{T_1}\mathscr K_{T_2},
\end{equation}
see Proposition \ref{17052019P3}. The sign uncertainty in \eqref{11022019E7} will not be an issue in the following. As a consequence of \eqref{11022019E7} and Proposition \ref{13112018P2}, we shall on the one hand, to prove \eqref{14012019E1}, obtain the existence of two families $(A_t)_{t\in\mathbb R}$ and $(H_t)_{t\in\mathbb R}$ of real symmetric positive semidefinite matrices $A_t\in\Symm^+_{2n}(\mathbb R)$ and real symplectic matrices $H_t\in\Symp_{2n}(\mathbb R)$ respectively, whose coefficients depend analytically on the time variable $t\in\mathbb R$, such that for all $t\in\mathbb R$,
\begin{equation}\label{13112018E21}
	e^{-2itJQ} = e^{-2itJA_t}H_t.
\end{equation}
On the other hand, to establish \eqref{14012019E2}, we shall prove that there exist a positive constant $T>0$ and a family $(B_t)_{-T<t<T}$ of real symmetric matrices, whose coefficients also depend analytically on the time-variable $-T<t<T$, such that for all $-T<t<T$, the real symplectic matrix $H_t$ is given by
\begin{equation}\label{08012019E10}
	H_t = e^{2tJB_t}.
\end{equation}
Indeed, let us first assume that \eqref{13112018E21} holds and let us prove \eqref{14012019E1}. It follows from \eqref{11022019E7} that for all $t\geq0$, up to sign,
$$e^{-tq^w} = \mathscr K_{e^{-2itJQ}} = \mathscr K_{e^{-2itJA_t}H_t} = \pm\mathscr K_{e^{-2itJA_t}}\mathscr K_{H_t} = e^{-ta^w_t}U_t,$$
where $U_t = \varepsilon_t\mathscr K_{H_t}$ is a metaplectic operator on $L^2(\mathbb R^n)$, see Definition \ref{17052019D1}, with $\varepsilon_t\in\{-1,1\}$, and $a_t:\mathbb R^{2n}\rightarrow\mathbb R_+$ is the nonnegative time-dependent quadratic form associated with the matrix $A_t$ in the canonical basis of $\mathbb R^{2n}$. This proves that \eqref{14012019E1} holds. On the other hand, to derive \eqref{14012019E2} from \eqref{08012019E10}, we consider the time-dependent quadratic form $b_t:\mathbb R^{2n}\rightarrow\mathbb R$, with $0\le t<T$, associated with the time-dependent matrix $B_t$ in the canonical basis of $\mathbb R^{2n}$. Indeed, when \eqref{08012019E10} holds, it follows from the definition of the operators $U_t$ and Proposition \ref{13112018P2} that for all $0\le t<T$, 
\begin{equation}\label{12072019E1}
	U_t = \varepsilon_t\mathscr K_{H_t} = \varepsilon_t\mathscr K_{e^{2tJB_t}} = \varepsilon_te^{-itb^w_t}.
\end{equation}
We then deduce from \eqref{14012019E1} and \eqref{12072019E1} that for all $0\le t<T$,
\begin{equation}\label{14112018E2}
	e^{-tq^w} = \varepsilon_te^{-ta^w_t}e^{-itb^w_t},
\end{equation}
It only remains to check that $\varepsilon_t = 1$ for all $0\le t<T$. To that end, we consider $u\in\mathscr S(\mathbb R^n)$ a non-zero Schwartz function. We deduce from \eqref{14112018E2} that for all $0\le t<T$,
$$\big\langle e^{-tq^w}u,e^{-itb^w_t}u\big\rangle_{L^2(\mathbb R^n)} = \varepsilon_t\big\langle e^{-ta^w_t}e^{-itb^w_t}u,e^{-itb^w_t}u\big\rangle_{L^2(\mathbb R^n)}.$$
Since the quadratic form $a_t$ is nonnegative for all $t\geq0$, the operator $e^{-ta^w_t}$ is selfadjoint on $L^2(\mathbb R^n)$ and we therefore deduce by using the semigroup property of the family of operators $(e^{-sa^w_t})_{s\geq0}$ that for all $t\geq0$,
$$\big\langle e^{-tq^w}u,e^{-itb^w_t}u\big\rangle_{L^2(\mathbb R^n)} = \varepsilon_t\big\Vert e^{-\frac t2a^w_t}e^{-itb^w_t}u\big\Vert^2_{L^2(\mathbb R^n)}.$$
The operator $e^{-\frac t2a^w_t}$ is injective from Corollary \ref{11022019C1} and the operator $e^{-itb^w_t}$ is unitary for all $t\geq0$, since the quadratic form $b_t$ is real-valued. Thus, the Schwartz functions $e^{-\frac t2a^w_t}e^{-itb^w_t}u$ are non-zero and we have that for all $t\geq0$,
\begin{equation}\label{14022019E5}
	\varepsilon_t = \big\langle e^{-tq^w}u,e^{-itb^w_t}u\big\rangle_{L^2(\mathbb R^n)}\big\Vert e^{-\frac t2a^w_t}e^{-itb^w_t}u\big\Vert^{-2}_{L^2(\mathbb R^n)}.
\end{equation}
Moreover, it follows from \cite[Theorem 4.2]{MR1339714} that the maps $t\mapsto e^{-tq^w}u$, $t\mapsto e^{-itb^w_t}u$ and $t\mapsto e^{-ta^w_t}e^{-itb^w_t}u$ are continuous from $[0,+\infty)$ to $\mathscr S(\mathbb R^n)$. It follows from \eqref{14022019E5} that the map $t\mapsto\varepsilon_t$ is also continuous from $[0,T)$ to $\{-1,1\}$ and since $\varepsilon_0=1$, we have $\varepsilon_t = 1$ for all $0\le t<T$.  This ends the proof of \eqref{14012019E2}.

The present subsection is therefore devoted to the proof of \eqref{13112018E21} and \eqref{08012019E10}. We first focus on the identity \eqref{13112018E21}. As above, we can prove that this relation is equivalent to the following triangular system,
\begin{equation}\label{14022019E6}
\left\{\begin{array}{lll} 
	e^{-4itJA_t} &=& e^{-2itJQ}e^{-2itJ\overline Q}, \\[2pt]
	H_t &=& e^{2itJA_t} e^{-2itJQ}.
\end{array}\right.
\end{equation}
We begin by solving the first equation of \eqref{14022019E6}:

\begin{thm}\label{04122018T1} There exists a family $(A_t)_{t\in\mathbb R}$ of real symmetric positive semidefinite matrices $A_t\in\Symm^+_{2n}(\mathbb R)$ whose coefficients depend analytically on the time-variable $t\in\mathbb R$ such that for all $t\in\mathbb R$,
$$e^{-4itJA_t} = e^{-2itJQ}e^{-2itJ\overline Q}.$$
\end{thm}

To prove Theorem \ref{04122018T1}, we need some technical lemmas. The first of them investigates the spectrum of the symplectic matrices $e^{-2itJQ}e^{-2itJ\overline Q}$ appearing in Theorem \ref{04122018T1}:

\begin{lem}\label{08012019L1} For all $t\in\mathbb R$, the eigenvalues of the matrix $e^{-2itJQ}e^{-2itJ\overline Q}$ are positive real numbers.
\end{lem}

\begin{proof} For all $t\in\mathbb R$, we define $K_t = e^{-2itJQ}e^{-2itJ\overline Q}$. We first check that for all $t\in\mathbb R$ we have
\begin{equation}\label{07012019E13}
	K_t = I_{2n} - 4iJ\Gamma_t, \quad\text{where}\quad \Gamma_t = \int_0^t(e^{-2isJ\overline Q})^*(\Reelle Q)(e^{-2isJ\overline Q})\ \mathrm ds.
\end{equation}
It follows from a direct computation for all $t\in\mathbb R$,
$$\partial_t \big(e^{-2itJQ}e^{-2itJ\overline Q}\big)
= -2ie^{-2itJQ}J(Q+\overline Q)e^{-2itJ\overline Q}
= -4ie^{-2itJQ}J(\Reelle Q)e^{-2itJ\overline Q}.$$
Since $Q$ is a symmetric matrix, it follows from Lemma \ref{09112018L1} that for all $t\in\mathbb R$, $e^{-2itJQ}\in\Symp_{2n}(\mathbb C)$ is a symplectic matrix and as a consequence of the above identity,
$$\partial_t \big(e^{-2itJQ}e^{-2itJ\overline Q}\big)
= -4iJ(e^{2itJQ})^T(\Reelle Q)e^{-2itJ\overline Q}
= -4iJ(e^{-2itJ\overline Q})^*(\Reelle Q)e^{-2itJ\overline Q}.$$
This proves that \eqref{07012019E13} holds. Since the matrices $\Gamma_t\in\Her_{2n}(\mathbb C)$ are Hermitian positive semidefinite when $t\geq0$ and Hermitian negative semidefinite when $t\le0$, we deduce from Lemma \ref{07012019L1} that for all $t\in\mathbb R$, the spectra of the matrices $J\Gamma_t$ satisfy $\sigma(J\Gamma_t)\subset i\mathbb R$. This combined with \eqref{07012019E13} shows that for all $t\in\mathbb R$, $\sigma\big(K_t\big)\subset\mathbb R$. The matrices $K_t\in\GL_{2n}(\mathbb C)$ are non singular and therefore, these inclusions can be refined to $\sigma\big(K_t\big)\subset\mathbb R^*$. Moreover, $\sigma(K_0) = \{1\}$ and the eigenvalues of $K_t$ are continuous with respect to the time-variable $t\in\mathbb R$ since the coefficients of the matrix $K_t$ are themselves continuous with respect to the time-variable $t\in\mathbb R$, see \cite[Theorem II.5.1]{MR0203473}. Since $\mathbb R$ is connected, this proves that $\sigma(K_t)\subset\mathbb R^*_+$ and ends the proof of Lemma \ref{08012019L1}.
\end{proof}

In the following, we shall need to define some matrices through the holomorphic functional calculus. We refer the reader to \cite[VII - 3.]{MR1009162} where this theory is presented. As a first application of this theory, we consider the matrix square root function $\sqrt{\cdot}$ defined on the set of matrices whose spectrum is contained in $\mathbb C\setminus\mathbb R_-$, which is possible since the function $z\mapsto\sqrt z = e^{\frac12\Log z}$ is well-defined and holomorphic in $\mathbb C\setminus\mathbb R_-$, with $\Log$ the principal determination of the logarithm in $\mathbb C\setminus\mathbb R_-$. For all $t\in\mathbb R$, since the spectrum of the matrix $K_t$ is only composed of positive real numbers, we can consider the matrix $G_t$ defined by
\begin{equation}\label{04122018E1}
	G_t = \sqrt{e^{-2itJQ}e^{-2itJ\overline Q}}.
\end{equation}
We shall check that the matrices $G_t$ are symplectic:

\begin{lem}\label{09012019L1} For all $t\in\mathbb R$, $G_t\in\Symp_{2n}(\mathbb C)$ is a complex symplectic matrix.
\end{lem}

\begin{proof} Let $t\in\mathbb R$ and $K_t = e^{-2itJQ}e^{-2itJ\overline Q}$. %We consider $K_t$ the matrix defined in \eqref{09012019E6}. 
We first observe that since both matrices $Q$ and $\overline Q$ are symmetric, Lemma \ref{09112018L1} shows that the matrices $e^{-2itJQ}$ and $e^{-2itJ\overline Q}$ are symplectic and as a consequence, the matrices $K_t\in\Symp_{2n}(\mathbb C)$ are also symplectic. To prove that the matrix $G_t$ is also symplectic, we need to go back to the definition of the matrix square root given by the functional holomorphic calculus. Therefore, we consider $\Sigma_t\subset\mathbb C$ the following domain of the complex plane
$$\Sigma_t = \Big\{re^{i\theta} : c_{1,t}<r<c_{2,t},\ \theta\in\Big(-\frac{\pi}2,\frac{\pi}2\Big)\Big\},$$
where the positive constants $c_{1,t},c_{2,t}>0$ are chosen so that $\sigma(K_t)\subset(c_{1,t},c_{2,t})$ and $\sigma(K^{-1}_t)\subset(c_{1,t},c_{2,t})$. Notice that the existence of the constants $c_{1,t},c_{2,t}>0$ is given by Lemma \ref{08012019L1}. We assume that the boundary $\partial\Sigma_t$ of the domain $\Sigma_t$ is oriented counterclockwise. Then, it follows from \eqref{04122018E1} and the holomorphic functional calculus that the matrix $G_t$ is defined by 
\begin{equation}\label{10052019E3}
	G_t = \frac1{2i\pi}\int_{\partial\Sigma_t}\sqrt z\ (K_t-zI_{2n})^{-1}\ \mathrm dz,
\end{equation}
with $\sqrt z = e^{\frac12\Log z}$. Moreover, since the matrix $K_t$ is symplectic, we deduce that
\begin{align}\label{14012019E3}
	JG_t & = \frac1{2i\pi}\int_{\partial\Sigma_t}\sqrt z\ J(K_t-zI_{2n})^{-1}\ \mathrm dz = \frac{-1}{2i\pi}\int_{\partial\Sigma_t}\sqrt z\ (K_tJ-zJ)^{-1}\ \mathrm dz \\
	& = \frac{-1}{2i\pi}\int_{\partial\Sigma_t}\sqrt z\ (J(K^T_t)^{-1}-zJ)^{-1}\ \mathrm dz \nonumber  = \frac1{2i\pi}\int_{\partial\Sigma_t}\sqrt z\ ((K^T_t)^{-1}-zI_{2n})^{-1}J\ \mathrm dz \nonumber \\
	& = \bigg(\frac1{2i\pi}\int_{\partial\Sigma_t}\sqrt z\ (K^{-1}_t-zI_{2n})^{-1}\ \mathrm dz\bigg)^TJ = \Big(\sqrt{K^{-1}_t}\Big)^TJ. \nonumber
\end{align}
Finally, since the function $z\mapsto(\sqrt z)^{-1} = \sqrt{z^{-1}}$ is holomorphic on $\mathbb C\setminus\mathbb R_-$ and that the eigenvalues of the matrices $K_t$ are positive real numbers, it follows from the holomorphic functional calculus, see e.g.\ \cite[VII.3.12, Theorem 12]{MR1009162}, that $\sqrt{K^{-1}_t} = \big(\sqrt{K_t}\big)^{-1} = G^{-1}_t.$
This, combined with \eqref{14012019E3}, proves that $JG_t = (G^T_t)^{-1}J$, that is $G_t\in\Symp_{2n}(\mathbb C)$ is a symplectic matrix. This ends the proof of Lemma \ref{09012019L1}.
\end{proof}

We can now construct the matrices $A_t$. Since the function $z\mapsto\atanh((z - 1)(z+1)^{-1})$ is holomorphic on a neighborhood of $\mathbb R^*_+$, where $\atanh$ denotes the hyperbolic $\atan$ function (whose definition and properties can be found in \cite[Section 4.6]{MR0167642}), and that $\sigma(G_t)\subset\mathbb R^*_+$ for all $t\in\mathbb R$ from \eqref{04122018E1} and Lemma \ref{08012019L1}, the functional holomorphic calculus also allows to consider the family of matrices $(A_t)_{t\in\mathbb R}$ defined for all $t\in\mathbb R$ by
\begin{equation}\label{07012019E9}
	A_t = -(itJ)^{-1}\atanh\big(\big(G_t-I_{2n}\big)\big(G_t+I_{2n}\big)^{-1}\big).
\end{equation}
By construction, the function $t\in\mathbb R\mapsto\atanh((G_t-I_{2n})(G_t+I_{2n})^{-1})$ is real analytic and vanishes in $t=0$, since $G_0 = I_{2n}$ from \eqref{04122018E1} and $\atanh(0_{2n}) = 0_{2n}$. The matrix $A_t$ is therefore well-defined for all $t\in\mathbb R$ and the function $t\in\mathbb R\mapsto A_t$ is as well analytic according to \eqref{15052019E1}. This family $(A_t)_{t\in\mathbb R}$ satisfies the algebraic part of Theorem \ref{04122018T1}, as proved in the

\begin{lem}\label{08012019L6} For all $t\in\mathbb R$, the matrix $A_t$ satisfies $e^{-4itJA_t} = e^{-2itJQ}e^{-2itJ\overline Q}.$
\end{lem}

\begin{proof} We first observe that 
\begin{equation}\label{08012019E13}
	\forall x>0,\quad \exp\Big(4\atanh\Big(\frac{x-1}{x+1}\Big)\Big) = x^2.
\end{equation}
Indeed, if $x>0$ a positive real number and $y\in\mathbb R$ is a real number such that $x = e^{2y}$, we have
$$\exp\Big(4\atanh\Big(\frac{x-1}{x+1}\Big)\Big) = \exp\Big(4\atanh\Big(\frac{e^{2y}-1}{e^{2y}+1}\Big)\Big) 
= \exp\big(4\atanh\big(\tanh y\big)\big) = e^{4y} = x^2.$$
Moreover, both functions $z\mapsto \exp(4\atanh((z-1)(z+1)^{-1}))$ and $z\mapsto z^2$ are holomorphic on a connected open neighborhood of $\mathbb R^*_+$ and $\sigma(G_t)\subset\mathbb R^*_+$ from \eqref{04122018E1} and Lemma \ref{08012019L1}. We therefore deduce from \eqref{07012019E9}, \eqref{08012019E13} and the holomorphic functional calculus that for all $t\in\mathbb R$,
$$e^{-4itJA_t} = \exp\big(4\atanh\big(\big(G_t-I_{2n}\big)\big(G_t+I_{2n}\big)^{-1}\big)\big) = G^2_t = e^{-2itJQ}e^{-2itJ\overline Q}.$$
This ends the proof of Lemma \ref{08012019L6}.
\end{proof}

Notice that the matrices $A_t$ can therefore be expressed by taking the logarithm of the matrices $e^{-2itJQ}e^{-2itJ\overline Q}$. Indeed, since the spectra of these matrices is contained in $\mathbb R^*_+$ from Lemma \ref{08012019L1} and that the function $\Log$ (which still denotes the principal determination of the logarithm in $\mathbb C\setminus\mathbb R_-$) is holomorphic in a neighborhood of $\mathbb R^*_+$, Lemma \ref{08012019L6} and the holomorphic functional calculus imply that for all $t\in\mathbb R$,
$$tA_t = -(4iJ)^{-1}\Log\big(e^{-2itJQ}e^{-2itJ\overline Q}\big).$$
Moreover, the function $t\in\mathbb R\mapsto\Log(e^{-2itJQ}e^{-2itJ\overline Q})$ is analytic by construction and vanishes in $t=0$. Consequently, the matrix $A_t$ is given for all $t\in\mathbb R$ by
\begin{equation}\label{08012019E9}
	A_t = -(4itJ)^{-1}\Log\big(e^{-2itJQ}e^{-2itJ\overline Q}\big).
\end{equation}

Now, it only remains to prove that the matrices $A_t$ are real and symmetric positive semidefinite. To that end, we introduce the family of matrices $(M_t)_{t\in\mathbb R}$ where $M_t$ is defined for all $t\in\mathbb R$ by
\begin{equation}\label{08012019E1}
	M_t = -(itJ)^{-1}\big(G_t-I_{2n}\big)\big(G_t+I_{2n}\big)^{-1}.
\end{equation}
Notice that the matrices $M_t$ are well-defined according to \eqref{15052019E1} since on the one hand, \eqref{04122018E1} and Lemma \ref{08012019L1} imply that $-1$ is not an eigenvalue of any matrix $G_t$ and on the other hand, the function $t\in\mathbb R\mapsto(G_t-I_{2n})(G_t+I_{2n})^{-1}$ is real analytic by construction and vanishes in $t=0$. Moreover, the function $t\in\mathbb R\mapsto M_t$ is analytic. We will prove in Lemma \ref{08012019L2} that the matrices $A_t$ can be expressed in terms of the matrices $M_t$ which will turn out to be real and symmetric. Moreover, the next lemma will imply that the matrices $M_t$ are positive semidefinite. The properties required for the matrices $A_t$ will then arise from the ones of the matrices $M_t$.

\begin{lem}\label{08012019L4} For all $t\in\mathbb R$, the matrix $M_t$ admits the following integral representation
$$M_t = \int_0^1(e^{-2i\alpha tJ\overline Q}\Phi_t)^*(\Reelle Q)(e^{-2i\alpha tJ\overline Q}\Phi_t)\ \mathrm d\alpha,$$
where the matrix $\Phi_t$ is given by
\begin{equation}\label{08012019E8}
	\Phi_t = 2 \bigg(\sqrt{e^{-2itJQ}e^{-2itJ\overline Q}}+I_{2n}\bigg)^{-1}.
\end{equation}
In particular, the matrix $M_t$ is Hermitian positive semidefinite.
\end{lem}

\begin{proof} Let $t\in\mathbb R$. We begin by checking that the matrix $M_t$ satisfies the relation
\begin{equation}\label{12112018E4}
	(G_t +I_{2n})^*tM_t(G_t +I_{2n}) = iJ(I_{2n} - G^2_t).
\end{equation}
We recall that the matrix $J$ satisfies $J^{-1} = J^T = -J$. On the one hand, the left-hand side of this equality can be computed with the definition \eqref{08012019E1} of $M_t$:
\begin{equation}\label{hehe}
	(G_t + I_{2n})^*tM_t(G_t + I_{2n}) = -i(G_t +I_{2n})^*J(G_t - I_{2n}).
\end{equation}
On the other hand, since the matrix square root given by the holomorphic functional calculus commutes with the complex conjugate (which can be readily checked by using \eqref{10052019E3}) and with the invert function defined for all non-singular matrix whose spectrum is composed of positive real numbers, it follows from \eqref{04122018E1} that the matrix $G_t$ satisfies
\begin{equation}\label{09012019E1}
	\overline G_t = \sqrt{e^{2itJ\overline Q}e^{2itJQ}} = \sqrt{(e^{-2itJQ}e^{-2itJ\overline Q})^{-1}} = G^{-1}_t.
\end{equation}
Moreover, $G_t\in\Symp_{2n}(\mathbb C)$ is a symplectic matrix from Lemma \ref{09012019L1} and we deduce that
\begin{multline}\label{09012019E11}
	(G_t +I_{2n})^*J = (\overline{G_t} + I_{2n})^TJ = (G^{-1}_t+I_{2n})^TJ = -(JG^{-1}_t+J)^T \\
	= -(G^T_tJ+J)^T = J(G^T_t+I_{2n})^T = J(G_t + I_{2n}).
\end{multline}
Hence, substituting this equality in \eqref{hehe}, we get that
\begin{equation}\label{potatoes}
	(G_t + I_{2n})^*tM_t(G_t +I_{2n}) = -iJ(G_t + I_{2n})(G_t - I_{2n}) = -iJ(G^2_t - I_{2n}).
\end{equation}
This proves that \eqref{12112018E4} holds. Then, we deduce from \eqref{07012019E13} and \eqref{04122018E1} that the right-hand side of \eqref{12112018E4} is written as
$$iJ(I_{2n} - G^2_t) = iJ(I_{2n} - e^{-2itJQ}e^{-2itJ\overline Q}) = 4\int_0^t(e^{-2isJ\overline Q})^*(\Reelle Q)(e^{-2isJ\overline Q})\ \mathrm ds.$$
Therefore, we derive the following expression for the matrix $tM_t$:
\begin{equation}\label{09012019E3}
	tM_t = 4\int_0^t(e^{-2isJ\overline Q}(G_t + I_{2n})^{-1})^*(\Reelle Q)(e^{-2isJ\overline Q}(G_t + I_{2n})^{-1})\ \mathrm ds.
\end{equation}
Since the matrix $\Phi_t$ defined in \eqref{08012019E8} can also be written as $\Phi_t = 2(G_t+I_{2n})^{-1},$
we deduce from \eqref{09012019E3} that the matrix $tM_t$ is given by
$$tM_t = \int_0^t(e^{-2isJ\overline Q}\Phi_t)^*(\Reelle Q)(e^{-2isJ\overline Q}\Phi_t)\ \mathrm ds.$$
A change of variable in the integral ends the proof of Lemma \ref{08012019L4}.
\end{proof}

We can now derive the end of the proof of Theorem \ref{04122018T1} from Lemma \ref{08012019L4}. This is done in the following Lemma which will also be key to prove Theorem \ref{14102018T2} in Section \ref{short}.

\begin{lem}\label{08012019L2} For all $t\in\mathbb R$, the matrix $A_t$ is real and symmetric positive semidefinite. Moreover, the matrices $A_t$ and $M_t$ satisfy the following estimate:
$$\forall t\in\mathbb R,\quad A_t\geq M_t\geq0.$$
\end{lem}

\begin{proof} To simplify the notations in the following, we consider the following matrices for all $t\in\mathbb R$,
\begin{equation}\label{07012019E14}
	\Psi_t = \big(G_t-I_{2n}\big)\big(G_t+I_{2n}\big)^{-1}.
\end{equation}
We recall that the matrix $\atanh$ function admits the following Taylor expansion for all matrices $R$ whose norm satisfies $\Vert R\Vert<1$,
\begin{equation}\label{09012019E7}
	\atanh R = \sum_{k=0}^{+\infty}\frac{R^{2k+1}}{2k+1}.
\end{equation}
We also recall from \eqref{07012019E9} that the matrices $A_t$ are defined for all $t\in\mathbb R$ with the convention \eqref{15052019E1} by $A_t = -(itJ)^{-1}\atanh \Psi_t$. It follows from the inequality
$$\forall x>0,\quad \bigg\vert\frac{\sqrt x-1}{\sqrt x+1}\bigg\vert<1,$$
the definitions \eqref{04122018E1} and \eqref{07012019E14} of the matrices $G_t$ and $\Psi_t$, and Lemma \ref{08012019L1} that the spectrum of the matrix $\Psi_t$ satisfies $\sigma(\Psi_t)\subset(-1,1)$ for all $t\in\mathbb R$. It therefore follows from \cite[Lemma 5.6.10]{MR1084815} that for all $t\in\mathbb R$, there exists a norm $\Vert\cdot\Vert_t$ on $M_n(\mathbb C)$ such that $\Vert\Psi_t\Vert_t<1$. This proves that the series $\sum\frac{\Psi^{2k+1}_t}{2k+1}$ converges in $M_n(\mathbb C)$ for all $t\in\mathbb R$ and we deduce from \eqref{09012019E7}  that for all $t\in\mathbb R$,
\begin{equation}\label{09012019E9}
	A_t = -(itJ)^{-1}\sum_{k=0}^{+\infty}\frac{\Psi^{2k+1}_t}{2k+1}.
\end{equation}
To prove that the matrices $A_t$ are real and symmetric, we need to derive a new expression for them. To that end, we compute the product $J\Psi_t$ by using the relation \eqref{09012019E11} (which also holds when the matrix $I_{2n}$ is replaced by $-I_{2n}$):
\begin{multline}\label{09012019E12}
	J\Psi_t = J\big(G_t-I_{2n}\big)\big(G_t+I_{2n}\big)^{-1} = \big(G_t-I_{2n}\big)^*J\big(G_t+I_{2n}\big)^{-1} \\
	= \big(G_t-I_{2n}\big)^*\big(\big(G_t+I_{2n}\big)^{-1}\big)^*J = \Psi^*_tJ.
\end{multline}
We deduce from \eqref{08012019E1}, \eqref{07012019E14}, \eqref{09012019E9} and \eqref{09012019E12} that for all $t\in\mathbb R$,
\begin{equation}\label{07012019}
	A_t = \sum_{k=0}^{+\infty}\frac1{2k+1}(\Psi^k_t)^*(-itJ)^{-1}\Psi_t(\Psi^k_t) = \sum_{k=0}^{+\infty}\frac1{2k+1}(\Psi^k_t)^*M_t(\Psi^k_t).
\end{equation}
We observe from \eqref{09012019E1} and \eqref{07012019E14} that for all $t\in\mathbb R$,
$$\overline{\Psi}_t = \big(\overline G_t-I_{2n}\big)\big(\overline G_t+I_{2n}\big)^{-1} = \big(G^{-1}_t-I_{2n}\big)\big(G^{-1}_t+I_{2n}\big)^{-1} 
= \big(I_{2n} - G_t\big)\big(I_{2n} + G_t\big)^{-1} = -\Psi_t.$$
Consequently, from \eqref{08012019E1}, \eqref{07012019E14}, \eqref{09012019E12}, the matrices $M_t$ satisfy the two relations
\begin{equation*}
\left\{\begin{array}{lll}
	\overline M_t &=& (itJ)^{-1}\overline{\Psi}_t = M_t, \\[5pt]
	M^*_t &=& (it)^{-1}\Psi^*_tJ = (it)^{-1}J\Psi_t = (-itJ)^{-1}\Psi_t = M_t.
\end{array}\right.
\end{equation*}
Thus, for all $t\in\mathbb R$ and $k\geq0$, the matrix $(\Psi^k_t)^*M_t(\Psi^k_t)$ is real and symmetric. As sums of such matrices, the matrices $A_t$ are also real and symmetric. Finally, we deduce from \eqref{07012019} and Lemma \ref{08012019L4} that for all $t\in\mathbb R$, $A_t\geq M_t\geq0$. This ends the proof of Lemma \ref{08012019L2}.
\end{proof}

We recall from \cite[Theorem 4.2]{MR1339714} that the evolution operators ${e^{-t\tilde q^w}}$, with $t\geq0$, generated by an accretive quadratic operator $\tilde q^w(x,D_x)$, with $\tilde q:\mathbb R^{2n}\rightarrow\mathbb C$ a complex-valued quadratic form with a nonnegative real-part $\Reelle\tilde q\geq0$, are pseudodifferential operators whose symbols are tempered distributions $p_t\in\mathscr S'(\mathbb R^{2n})$. More specifically, these symbols are $L^{\infty}(\mathbb R^{2n})$ functions explicitly given by the Mehler formula
\begin{equation}\label{12022019E2}
	p_t(X) = (\det(\cos(t\tilde F)))^{-\frac12}\exp(-\sigma(X,\tan(t\tilde F)X)),\quad X\in\mathbb R^{2n},
\end{equation}
whenever the condition $\det(\cos(t\tilde F))\ne0$ is satisfied, where $\tilde F$ denotes the Hamilton map associated with the quadratic form $\tilde q$. As a Corollary of Lemma \ref{08012019L2}, we can compute the Weyl symbol of the operator $e^{-ta^w_t}$ for all $t\geq0$, with $a_t:\mathbb R^{2n}\rightarrow\mathbb R_+$ the nonnegative quadratic form whose matrix in the canonical basis of $\mathbb R^{2n}$ is $A_t$, in terms of $m_t:\mathbb R^{2n}\rightarrow\mathbb R_+$ the nonnegative quadratic form whose matrix in the canonical basis of $\mathbb R^{2n}$ is $M_t$. By the way, this is a justification \textit{a posteriori} of the introduction of the matrices $M_t$.

\begin{cor}\label{09012019C1} For all $t\geq0$, the operator $e^{-ta^w_t}$ is a pseudodifferential operator whose Weyl symbol is given by
$$X\in\mathbb R^{2n}\mapsto (\det\cos(tJA_t))^{-\frac12}e^{-tm_t(X)}\in L^{\infty}(\mathbb R^{2n}).$$
\end{cor}

\begin{proof} Let $t\geq0$. It follows from Lemma \ref{08012019L2} that the matrix $A_t$ is real symmetric positive semidefinite and this combined with Lemma \ref{07012019L1} show that the spectrum of the matrix $tJA_t$ is purely imaginary. As a consequence, the matrix $\cos(tJA_t)$ is non-singular and it follows from the Mehler formula \eqref{12022019E2} that the operator $e^{-ta^w_t}$ is a pseudodifferential operator whose Weyl symbol is a $L^{\infty}(\mathbb R^{2n})$-function given for all $X\in\mathbb R^{2n}$ by
$$ (\det\cos(tJA_t))^{-\frac12}\exp(-\sigma(X,\tan(tJA_t)X)).$$
Moreover, we deduce from \eqref{04122018E1}, \eqref{08012019E1} and Lemma \ref{08012019L6} that 
\begin{multline*}
	(tJ)^{-1}\tan(tJA_t) = -(itJ)^{-1}\big(e^{-2itJA_t}-I_{2n}\big)\big(e^{-2itJA_t}+I_{2n}\big)^{-1} \\
	= -(itJ)^{-1}\big(G_t-I_{2n}\big)\big(G_t+I_{2n}\big)^{-1} = M_t.
\end{multline*}
We deduce from \eqref{14022019E9} and the above equality that for all $X\in\mathbb R^{2n}$,
$$\sigma(X,\tan(tJA_t)X)) = \sigma(X,tJM_tX) = t\langle X,M_tX\rangle = tm_t(X).$$
This ends the proof of Corollary \ref{09012019C1}.
\end{proof}

The study of the family $(A_t)_{t\in\mathbb R}$ is now ended. Still in order to prove \eqref{13112018E21} \textit{via} \eqref{14022019E6}, we consider the time-dependent matrices $H_t$ defined for all $t\in\mathbb R$ by
\begin{equation}\label{08012018E6}
	H_t = e^{2itJA_t}e^{-2itJQ}.
\end{equation}
Notice that the analyticity of the function $t\in\mathbb R\mapsto H_t$ is induced by the ones of the functions $t\in\mathbb R\mapsto A_t$ and $t\in\mathbb R\rightarrow e^{-2itJM}$ for all $M\in M_{2n}(\mathbb C)$. We only need to check that each matrix $H_t$ is real and symplectic.

\begin{lem}\label{08012018L5} For all $t\in\mathbb R$, $H_t$ is a real symplectic matrix.
\end{lem}

\begin{proof} Let $t\in\mathbb R$. Since both matrices $A_t$ and $Q$ are symmetric (from Lemma \ref{08012019L2} concerning $A_t$), Lemma \ref{09112018L1} shows that the matrices $e^{2itJA_t}$ and $e^{-2itJQ}$ are symplectic. As a consequence, the matrix $H_t$ is also symplectic. Moreover, it follows from Lemma \ref{08012019L6} that
\begin{multline}\label{10012019E1}
	\overline{H_t} = e^{-2itJA_t}e^{2itJ\overline Q} = e^{2itJA_t}e^{-4itJA_t}e^{2itJ\overline Q} \\
	= e^{2itJA_t}e^{-2itJQ}e^{-2itJ\overline Q}e^{2itJ\overline Q} = e^{2itJA_t}e^{-2itJQ} = H_t,
\end{multline}
which proves that $H_t$ is a real matrix. This ends the proof of Lemma \ref{08012018L5}.
\end{proof}

This ends the proof of \eqref{13112018E21} and the splitting of the symplectic matrices $e^{-2itJQ}$ in any time $t\in\mathbb R$.

The rest of this section is then devoted to prove \eqref{08012019E10} which sharpens the decomposition \eqref{13112018E21} for small times $\vert t\vert\ll1$ . The strategy will be different than the one used until now, since the holomorphic functional calculus will not be used anymore to define the different matrices at play. The identity \eqref{08012019E10} is proved in the following lemma:

\begin{lem}\label{09012019L2} There exist a positive constant $T>0$ and a family $(B_t)_{-T< t< T}$ of real symmetric matrices $B_t\in\Symm_{2n}(\mathbb R)$ whose coefficients  depend analytically on the time-variable $-T< t< T$ such that for all $-T< t< T$, the symplectic matrix $H_t$ is written as $H_t = e^{2tJB_t}$.
\end{lem}

\begin{proof} First, we recall that for all matrix $M\in M_{2n}(\mathbb C)$ satisfying $\Vert M-I_{2n}\Vert<1$, the matrix $\Log M$ is given by the following sum
\begin{equation}\label{10012019E3}
	\Log M = \sum_{k=1}^{+\infty}\frac{(-1)^{k-1}}k(M-I_{2n})^k.
\end{equation}
Since $H_t$ goes to $I_{2n}$ as $t$ goes to $0$, there exists a positive constant $T>0$ such that,
\begin{equation}\label{09012019E17}
	\forall t\in (-T,T),\quad \big\Vert H_t - I_{2n}\big\Vert<1 \quad\mathrm{and}\quad \big\Vert H^{-1}_t - I_{2n}\big\Vert<1.
\end{equation}
The estimate \eqref{09012019E17} allows to consider the matrix $B_t$ defined for all $-T< t< T$ by
\begin{equation}\label{09012019E18}
	B_t = (2tJ)^{-1}\Log H_t.
\end{equation}
 Notice that the function $t\in(-T,T)\mapsto\Log H_t$ is analytic by construction and vanishes in $t=0$ since $H_0 = I_{2n}$. The matrix $B_t$ is therefore well-defined for all $-T<t<T$ according to \eqref{15052019E1}. We deduce from \eqref{08012018E6}, \eqref{09012019E17} and \eqref{09012019E18}, that for all $-T< t< T$,
$e^{2tJB_t} = \exp\big(\Log H_t\big) = H_t.$
It remains to check that the matrices $B_t$ are real and symmetric. First we observe from \eqref{10012019E1} and \eqref{09012019E18} that for all $-T< t< T$,
$$\overline B_t = (2tJ)^{-1}\Log\overline H_t = (2tJ)^{-1}\Log H_t = B_t.$$
This proves that the matrices $B_t$ are real. Moreover, we deduce from \eqref{09012019E17}, \eqref{09012019E18}, Lemma \ref{08012018L5} and the binomial formula that for all $-T<t< T$,
\begin{align*}
	B^T_t & = (2t)^{-1}(\Log H_t)^TJ = (2t)^{-1}\sum_{k=1}^{+\infty}\frac{(-1)^{k-1}}k\big((H_t-I_{2n})^k\big)^TJ \\
	& = (2t)^{-1}\sum_{k=1}^{+\infty}\frac{(-1)^{k-1}}k\sum_{\ell=0}^k\binom kl(-1)^{k-l}(H^{\ell}_t)^TJ  = (2t)^{-1}\sum_{k=1}^{+\infty}\frac{(-1)^{k-1}}k\sum_{\ell=0}^k\binom kl(-1)^{k-l}J(H^{-1}_t)^{\ell} \\
	& = (2t)^{-1}\sum_{k=1}^{+\infty}\frac{(-1)^{k-1}}kJ(H^{-1}_t-I_{2n})^k  = -(2tJ)^{-1}\Log(H^{-1}_t) = (2tJ)^{-1}\Log H_t = B_t.
\end{align*}
The matrices $B_t$ are therefore symmetric. Moreover, the function $t\in(-T,T)\mapsto B_t$ is analytic by contruction. This ends the proof of Lemma \ref{09012019L2}.
\end{proof}

\section{Study of the real part for short times}\label{short}

In this section, we prove Theorem \ref{14102018T2}. Let $q:\mathbb R^{2n}\rightarrow\mathbb C$ be a complex-valued quadratic form with a nonnegative real part $\Reelle q\geq0$. We consider $F$ the Hamilton map associated with $q$, $S$ its singular space and $0\le k_0\le 2n-1$ the smallest integer such that \eqref{12102018E7} holds. Let $(a_t)_{t\in\mathbb R}$ be the family of nonnegative quadratic forms $a_t:\mathbb R^{2n}\rightarrow\mathbb R_+$ given by Theorem \ref{12102018T1} and $(m_t)_{t\in\mathbb R}$ be the family of nonnegative quadratic forms $m_t:\mathbb R^{2n}\rightarrow\mathbb R_+$ whose matrices in the canonical basis of $\mathbb R^{2n}$ are the matrices $M_t$ defined in \eqref{08012019E1}. We shall prove that the quadratic forms $m_t$ (and therefore the quadratic forms $a_t$) satisfy a sharp lower bound implying some degenerate anisotropic coercivity properties on the phase space. More precisely, we shall prove that there exist some positive constants $c>0$ and $T>0$ such that for all $0\le t\le T$ and $X\in\mathbb R^{2n}$,
\begin{equation}\label{12112018E11}
	a_t(X)\geq m_t(X)\geq c\sum_{k=0}^{k_0} t^{2k} \Reelle q\big((\Imag F)^k X\big).
\end{equation}
Notice that the left inequality in \eqref{12112018E11} is a consequence of Lemma \ref{08012019L2}. We are therefore interested in proving the right one. To that end, we consider the time-dependent quadratic form $\kappa_t:\mathbb C^{2n}\rightarrow\mathbb R$ defined in accordance with the convention \eqref{07052019E1} for all $t\geq0$ and $X\in\mathbb C^{2n}$ by
\begin{equation}\label{18102018E4}
	\kappa_t(X) = \sum_{k=0}^{k_0} t^{2k} \Reelle q\big((\Imag F)^k X\big) = \sum_{k=0}^{k_0}t^{2k}\big\vert\sqrt{\Reelle Q}(\Imag F)^kX\big\vert^2.
\end{equation}
We recall from Lemma \ref{08012019L4} that for all $t\geq0$, the matrix $M_t$ admits the following integral representation
\begin{equation}\label{14022019E8}
	M_t = \int_0^1(e^{-2i\alpha t\overline F}\Phi_t)^*(\Reelle Q)(e^{-2i\alpha t\overline F}\Phi_t)\ \mathrm d\alpha, \quad\mathrm{with}\quad \Phi_t = \bigg(\frac{\sqrt{e^{-2itF}e^{-2it\overline F}}+I_{2n}}2\bigg)^{-1}.
\end{equation}
We therefore deduce that for all $t\geq0$ and $X\in\mathbb R^{2n}$,
\begin{equation}\label{15022019E1}
	m_t(X) = X^TM_tX = \int_0^1(e^{-2i\alpha t\overline F}\Phi_tX)^*(\Reelle Q)(e^{-2i\alpha t\overline F}\Phi_tX)\ \mathrm d\alpha,
\end{equation}
and this equality can be written as
$$m_t(X) = \int_0^1\big\vert\sqrt{\Reelle Q}e^{-2i\alpha t\overline F}\Phi_tX\big\vert^2\ \mathrm d\alpha = \big\Vert\sqrt{\Reelle Q}e^{-2i\alpha t\overline F}\Phi_tX\big\Vert_{L^2(0,1)}^2.$$
By applying the Minkowski inequality, we therefore obtain that for all $t\geq0$ and $X\in\mathbb R^{2n}$,
\begin{equation}\label{05112018E4}
	\sqrt{m_t(X)}
	\geq\bigg\Vert\sum_{k=0}^{k_0}\frac{(-2t\alpha)^k}{k!}\sqrt{\Reelle Q}(i\overline F)^k\Phi_tX\bigg\Vert_{L^2(0,1)} 
	\! \! \!- \bigg\Vert\sum_{k>k_0}\frac{(-2t\alpha)^k}{k!}\sqrt{\Reelle Q}(i\overline F)^k\Phi_tX\bigg\Vert_{L^2(0,1)}\! \! \! \! \! \! \!.
\end{equation}
We then study separately the two terms of the right-hand side of the above estimate. \\[5pt]
\textbf{1.} First, we focus on controlling the first term in the right-hand side of \eqref{05112018E4}. On the finite-dimensional vector space $(\mathbb C_{k_0}[X])^{2n}$, the Hardy's norm $\Vert\cdot\Vert_{\mathcal H^1}$ defined by
$$\bigg\Vert\sum_{k=0}^{k_0}y_kX^k\bigg\Vert_{\mathcal H^1} = \sum_{k=0}^{k_0}k!2^{-k}\vert y_k\vert,\quad y_0,\ldots,y_{k_0}\in\mathbb C^{2n},$$
is equivalent to the standard Lebesgue's norm $\Vert\cdot\Vert_{L^2(0,1)}$ given by 
$$\bigg\Vert\sum_{k=0}^{k_0}y_kX^k\bigg\Vert^2_{L^2(0,1)} = \int_0^1\bigg\vert\sum_{k=0}^{k_0}y_k\alpha^k\bigg\vert^2\mathrm d\alpha,\quad y_0,\ldots,y_{k_0}\in\mathbb C^{2n}.$$
Thus, there exists a positive constant $c_1>0$ such that for all $t\geq0$ and $X\in\mathbb R^{2n}$,
\begin{equation}\label{05112018E5}
	\bigg\Vert\sum_{k=0}^{k_0}\frac{(-2t\alpha)^k}{k!}\sqrt{\Reelle Q}(i\overline F)^k\Phi_tX\bigg\Vert_{L^2(0,1)}
	\geq c_1\sum_{k=0}^{k_0} t^k \big\vert\sqrt{\Reelle Q}(i\overline F)^k\Phi_tX\big\vert.
\end{equation}
We develop the matrices $(i\overline F)^k$ in the following way:
\begin{equation}\label{12112018E13}
	(i\overline F)^k = (\Imag F)^k + B_k, \quad\mathrm{where}\quad B_k = \sum_{j=1}^{2^k-1}\varepsilon_{j,k}M_{j,k}(\Reelle F)(\Imag F)^{m_{j,k}},
\end{equation}
with $0\le m_{j,k}\le k-1$, $\varepsilon_{j,k}\in\{-1,1,-i,i\}$ and the matrices $M_{j,k}$ are finite products of $\Reelle F$ and $\Imag F$. Then, by putting \eqref{12112018E13} in \eqref{05112018E5} and using the triangle inequality, we obtain the following estimate for all $t\geq0$ and $X\in\mathbb R^{2n}$,
\begin{equation}\label{05112018E8}
	\sum_{k=0}^{k_0}t^k\big\vert\sqrt{\Reelle Q}(i\overline F)^k\Phi_tX\big\vert
	\geq\sum_{k=0}^{k_0}t^k \big\vert\sqrt{\Reelle Q}(\Imag F)^k\Phi_tX\big\vert
	- \sum_{k=0}^{k_0}t^k \big\vert\sqrt{\Reelle Q}B_k\Phi_tX\big\vert.
\end{equation}
Denoting by $\#$ the cardinality, we consider the two positive quantities
$$c_2 = \max_{0\le k\le k_0}\max_{1\le j\le 2^k-1}\big\Vert\sqrt{\Reelle Q}M_{j,k}J\sqrt{\Reelle Q}\big\Vert>0,$$
$$c'_2 = \max_{0\le k\le k_0}\max_{0\le m\le k-1}\#\big\{1\le j\le 2^k-1 : m_{j,k} = m\big\}>0,$$
Since $F = JQ$, it follows from the definition of $B_k$ that for all $t\geq0$ and $X\in\mathbb R^{2n}$,
\begin{align*}
	\sum_{k=0}^{k_0}t^k\big\vert\sqrt{\Reelle Q}B_k\Phi_tX\big\vert
	& \le \sum_{k=0}^{k_0}t^k\sum_{j=1}^{2^k-1}\big\vert\sqrt{\Reelle Q}M_{j,k}(\Reelle F)(\Imag F)^{m_{j,k}}\Phi_tX\big\vert \\
	& = \sum_{k=0}^{k_0}t^k\sum_{j=1}^{2^k-1}\big\vert\sqrt{\Reelle Q}M_{j,k}J\sqrt{\Reelle Q}\sqrt{\Reelle Q}(\Imag F)^{m_{j,k}}\Phi_tX\big\vert \\
	& \le c_2\sum_{k=0}^{k_0}t^k\sum_{j=1}^{2^k-1}\big\vert\sqrt{\Reelle Q}(\Imag F)^{m_{j,k}}\Phi_tX\big\vert.
\end{align*}
Then, we gather the integers $0\le m_{j,k}\le k-1$ taking the same value, which shows that for all $t\geq0$ and $X\in\mathbb R^{2n}$,
\begin{align*}
	\sum_{k=0}^{k_0}t^k\big\vert\sqrt{\Reelle Q}B_k\Phi_tX\big\vert
	& \le c_2\sum_{k=0}^{k_0}t^k\sum_{m=0}^{k-1}\sum_{\substack{1\le j\le 2^k-1 \\[1pt] m_{j,k} = m}}\big\vert\sqrt{\Reelle Q}(\Imag F)^m\Phi_tX\big\vert \\
	& \le c_2c'_2\sum_{k=0}^{k_0}\sum_{m=0}^{k-1}t^k\big\vert\sqrt{\Reelle Q}(\Imag F)^m\Phi_tX\big\vert. \nonumber
\end{align*}
Since $k-m\geq1$, we have that for all $0\le t\le 1$,
$t^k = t^{k-m}t^m\le t^{1+m}.$
The following inequality therefore holds for all $0\le t\le 1$ and $X\in\mathbb R^{2n}$,
$$\sum_{k=0}^{k_0}t^k\big\vert\sqrt{\Reelle Q}B_k\Phi_tX\big\vert\le c_2c'_2t\sum_{k=0}^{k_0}\sum_{m=0}^{k-1}t^m\big\vert\sqrt{\Reelle Q}(\Imag F)^m\Phi_tX\big\vert.$$
As a consequence, there exists a positive constant $c_3>0$ such that for all $0\le t\le 1$ and $X\in\mathbb R^{2n}$,
\begin{equation}\label{05112018E10}
	\sum_{k=0}^{k_0}t^k\big\vert\sqrt{\Reelle Q}B_k\Phi_tX\big\vert
	\le c_3t\sum_{k=0}^{k_0}t^k\big\vert\sqrt{\Reelle Q}(\Imag F)^k\Phi_tX\big\vert.
\end{equation}
It follows from \eqref{05112018E5}, \eqref{05112018E8} and \eqref{05112018E10} that for all $0\le t\le 1$ and $X\in\mathbb R^{2n}$,
\begin{equation}\label{13112018E2}
	\bigg\Vert\sum_{k=0}^{k_0}\frac{(-2t\alpha)^k}{k!}\sqrt{\Reelle Q}(i\overline F)^k\Phi_tX\bigg\Vert_{L^2(0,1)}
	\geq c_1(1-c_3t)\sum_{k=0}^{k_0}t^k\big\vert\sqrt{\Reelle Q}(\Imag F)^k\Phi_tX\big\vert.
\end{equation}
We recall from the third inequality of \eqref{13052019E20} (no assumption of smallness is required for $t\geq0$ to apply this estimate) that for all $0\le t\le 1$ and $X\in\mathbb R^{2n}$,
\begin{equation}\label{13112018E3}
	\sqrt{\kappa_t}(\Phi_tX)\le \sum_{k=0}^{k_0}t^k\big\vert\sqrt{\Reelle Q}(\Imag F)^k\Phi_tX\big\vert.
\end{equation}
As a consequence of \eqref{13112018E2} and \eqref{13112018E3}, there exist some positive constants $0<t_1<1$ and $c_4>0$ such that for all $0\le t\le t_1$ and $X\in\mathbb R^{2n}$,
\begin{equation}\label{19112018E12}
	\bigg\Vert\sum_{k=0}^{k_0}\frac{(-2t\alpha)^k}{k!}\sqrt{\Reelle Q}(i\overline F)^k\Phi_tX\bigg\Vert_{L^2(0,1)}\geq c_4\sqrt{\kappa_t}(\Phi_tX).
\end{equation}
In order to estimate from below the term $\sqrt{\kappa_t}(\Phi_tX)$, we would like to apply Lemma \ref{05112018L2} to the function 
\begin{equation}\label{04092019E4}
	G(M,N) = 2 ( \sqrt{e^{-2i(M+iN)}e^{-2i(M-iN)}} + I_{2n} )^{-1},
\end{equation}
in view of the definition of the matrices $\Phi_t$ in \eqref{14022019E8} . We prove in Lemma \ref{25122025L7} in the Appendix that the function $G$ actually satisfies the assumptions of Lemma \ref{05112018L2} and as a consequence, there exist some positive constants $c_5>0$ and $0<t_2<t_1$ such that for all $0\le t\le t_2$ and $X\in\mathbb R^{2n}$,
$$\bigg\Vert\sum_{k=0}^{k_0}\frac{(-2t\alpha)^k}{k!}\sqrt{\Reelle Q}(i\overline F)^k\Phi_tX\bigg\Vert_{L^2(0,1)}\geq c_5\sqrt{\kappa_t}(X).$$
This inequality, combined with \eqref{05112018E4}, leads to the following estimate for all $0\le t\le t_2$ and $X\in\mathbb R^{2n}$,
\begin{equation}\label{05112018E16}
	\sqrt{m_t}(X)\geq c_5\sqrt{\kappa_t}(X)
	- \bigg\Vert\sum_{k>k_0}\frac{(-2t\alpha)^k}{k!}\sqrt{\Reelle Q}(i\overline F)^k\Phi_tX\bigg\Vert_{L^2(0,1)}.
\end{equation}
\textbf{2.} The end of the proof consists in controlling the second term in the right hand side of \eqref{05112018E4}. The technics employed will be similar to the ones used in the end of the proof of Lemma \ref{05112018L2}. We begin by observing that for all $0\le t\le t_2$ and $X\in\mathbb R^{2n}$,
$$\bigg\Vert\sum_{k>k_0}\frac{(-2t\alpha)^k}{k!}\sqrt{\Reelle Q}(i\overline F)^k\Phi_tX\bigg\Vert^2_{L^2(0,1)}
= t^{2k_0+2}\bigg\Vert\sum_{k>k_0}t^{k-k_0-1}\frac{(-2\alpha)^k}{k!}\sqrt{\Reelle Q}(i\overline F)^k\Phi_tX\bigg\Vert^2_{L^2(0,1)}.$$
Note that the coefficients of the quadratic form above are continuous with respect to $t\in[0,t_2]$. As a consequence, there exists a positive constant $c_6>0$ such that for all $0\le t\le t_2$ and $X\in\mathbb R^{2n}$,
\begin{equation}\label{13112018E4}
	\bigg\Vert\sum_{k>k_0}\frac{(-2t\alpha)^k}{k!}\sqrt{\Reelle Q}(i\overline F)^k\Phi_tX\bigg\Vert^2_{L^2(0,1)}\le c_6 t^{2k_0+2}\vert X\vert^2.
\end{equation}
On the other hand, it follows from Lemma \ref{13112018L1} that there exists a positive constant $c_7>0$ such that for all $0\le t\le1$ and $X\in S^{\perp}$, 
\begin{equation}\label{13112018E5}
	\kappa_t(X)\geq c_7t^{2k_0}\vert X\vert ^2.
\end{equation}
As a consequence of \eqref{13112018E4} and \eqref{13112018E5}, we have that for all $0\le t\le t_2$ and $X\in S^{\perp}$,
\begin{equation}\label{13112018E8}
	\bigg\Vert\sum_{k>k_0}\frac{(-2t\alpha)^k}{k!}\sqrt{\Reelle Q}(i\overline F)^k\Phi_tX\bigg\Vert^2_{L^2(0,1)}
	\le \frac{c_6}{c_7}t^2\kappa_t(X).
\end{equation}
We deduce from \eqref{05112018E16} and \eqref{13112018E8} that there exist some positive constants $c_8>0$ and $0<t_3<t_2$ such that for all $0\le t\le t_3$ and $X\in S^{\perp}$,
\begin{equation}\label{13112018E9}
	m_t(X)\geq\bigg(c_4-\sqrt{c_6c_7^{-1}}t\bigg)^2\kappa_t(X)\geq c_8\kappa_t(X).
\end{equation}
It remains to check that the estimate \eqref{13112018E9} holds for all $X\in\mathbb R^{2n}$. To that end, we will use the following lemma of linear algebra:

\begin{lem}\label{28082019L1} Let $E$ be a real finite-dimensional vector space and $q_1,q_2$ be two nonnegative quadratic forms on $E$. If $E=F\oplus G$ is a direct sum of two vector subspaces such that $q_1\le q_2$ on $F$ and $q_1,q_2$ both vanish on $G$, then $q_1\le q_2$ on $E$.
\end{lem}

\begin{proof}We equip $E$ with a basis. Let $(\cdot,\cdot)$ (resp. $\Vert\cdot\Vert$) be the associated scalar product (resp. Euclidean norm). We denote by $Q_1$ (resp. $Q_2$) the matrix of $q_1$ (resp. $q_2$) in this basis. Since the quadratic forms $q_1$ and $q_2$ are nonnegative, as a consequence of the spectral theorem we have $G\subset\Ker Q_1\cap\Ker Q_2$. Indeed, for all $g\in G$ and $j\in \{1,2\}$, we have$\Vert\sqrt{Q_j} g\Vert^2 = q_j(g) = 0 $, and so $g\in\Ker\sqrt{Q_j} =  \Ker Q_j$. Therefore, for all $(f,g)\in F\times G$, we have
$$
q_2(f+g) = (f+g,Q_2 (f+g)) = (f,Q_2 f) = q_2(f) \leq q_1(f) = q_1(f+g).
$$
Since $E=F+G$, this estimate ends the proof of Lemma \ref{28082019L1}.
\end{proof}

\noindent Let $0\le t\le t_3$. Since $\mathbb R^{2n} = S\oplus S^{\perp}$ and that \eqref{13112018E9} is valid on $S^{\perp}$, it is sufficient to prove that both nonnegative quadratic forms $\kappa_t$ and $m_t$ vanish on the singular space $S$, according to Lemma \ref{28082019L1}. We first notice that by definition, $\kappa_t$ is zero on the singular space $S$. We now prove that this property holds true as well for the quadratic form $m_t$, that is
\begin{equation}\label{09052019E4}
	\forall X\in S,\quad m_t(X) = 0.
\end{equation}
To that end, we use anew the integral representation of $m_t$ given by \eqref{15022019E1},
\begin{equation}\label{13112018E10}
	\forall X\in\mathbb R^{2n},\quad m_t(X) = \int_0^1(e^{-2i\alpha t\overline F}\Phi_tX)^*(\Reelle Q)(e^{-2i\alpha t\overline F}\Phi_tX)\ \mathrm d\alpha.
\end{equation}
According to \eqref{13112018E10}, it is sufficient to prove that 
\begin{equation}\label{13112018E11}
	\forall\alpha\in[0,1],\quad (e^{-2i\alpha t\overline F}\Phi_t)S\subset S+iS,
\end{equation}
since $(\Reelle Q)S = J^{-1}(\Reelle F)S = \{0\}$ from \eqref{05112018E7} and \eqref{23012019E2}. As a consequence of \eqref{05092019E6}, the inclusion $\Phi_tS\subset S+iS$ holds, up to decreasing the positive constant $t_3>0$. Moreover, notice from \eqref{23012019E2} that the space $S+iS$ is stable by the matrix $\overline F$ (since $(\Reelle F)S=\{0\}$ and $(\Imag F)S\subset S$), and therefore by the matrices $e^{-2i\alpha t\overline F}$ for all $0\le\alpha\le1$. This proves that the inclusion \eqref{13112018E11} actually holds. The estimate \eqref{13112018E9} can therefore be extended to all $0\le t\le t_3$ (up to decreasing $t_3>0$) and $X\in\mathbb R^{2n}$. This ends the proof of the estimate \eqref{12112018E11}.

\section{Regularizing effects of semigroups generated by non-selfadjoint quadratic differential operators}\label{Reg}

The aim of this section is to prove Theorem \ref{22102018T1} and Theorem \ref{10122018T1} about the regularizing properties of the semigroups generated by non-selfadjoint quadratic differential operators. Let $q:\mathbb R^{2n}\rightarrow\mathbb C$ be a complex-valued quadratic form with a nonnegative real part $\Reelle q\geq0$. We consider $Q\in\Symm_{2n}(\mathbb C)$ the matrix of $q$ in the canonical basis of $\mathbb R^{2n}$, $F\in M_{2n}(\mathbb C)$ its Hamilton map and $S$ its singular space.

\subsection{Regularizing effects} We begin by proving Theorem \ref{22102018T1}. Let $T>0$ and $(a_t)_{t\in\mathbb R}$, $(b_t)_{-T<t< T}$ be the families of quadratic forms given by Theorem \ref{12102018T1}. We recall that the quadratic forms $a_t$ are nonnegative, the quadratic forms $b_t$ are real-valued and $a_t$, $b_t$ depend analytically on the time-variable $t\in\mathbb R$ and $-T<t<T$ respectively. Moreover, the evolution operators $e^{-tq^w}$ can be factorized as
\begin{equation}\label{19112018E1}
	\forall t\in[0,T),\quad e^{-tq^w} = e^{-ta^w_t}e^{-itb^w_t}.
\end{equation} 
We can assume that the positive constant $0<T<1$ is the one given by Theorem \ref{14102018T2}, which implies that there exists a positive constant $c>0$ such that for all $0\le t\le T$ and $X\in\mathbb R^{2n}$,
\begin{equation}\label{15012019E2}
	a_t(X)\geq c\sum_{j=0}^{k_0}t^{2j}\Reelle q\big((\Imag F)^jX\big),
\end{equation}
where $0\le k_0\le 2n-1$ is the smallest integer such that \eqref{12102018E7} holds. As in Section \ref{Polar}, we denote by $A_t$ and $B_t$ the respective matrices of $a_t$ and $b_t$ in the canonical basis of $\mathbb R^{2n}$. Moreover, we consider anew the time-dependent quadratic form $\kappa_t$ defined in accordance with the convention \eqref{07052019E1} for all $t\geq0$ and $X\in\mathbb C^{2n}$ by
\begin{equation}\label{15012019E1}
	\kappa_t(X) = \sum_{j=0}^{k_0}t^{2j}\Reelle q\big((\Imag F)^jX\big) = \sum_{j=0}^{k_0}t^{2j}\big\vert\sqrt{\Reelle Q}(\Imag F)^jX\big\vert^2.
\end{equation}
The estimate \eqref{15012019E2} reads as: for all $0\le t\le T$ and $X\in\mathbb C^{2n}$, $a_t(X)\geq c\kappa_t(X)$. The aim of this section is to understand the smoothing properties of the evolution operators $e^{-tq^w}$. Since the operators $e^{-itb^w_t}$ are unitary on $L^2(\mathbb R^n)$, we first notice from \eqref{19112018E1} that it is sufficient to study the regularizing properties of the operators $e^{-ta^w_t}$ to derive the ones of the operators $e^{-tq^w}$. Therefore, for some $m\geq1$ and $X_1,\ldots,X_m\in S^{\perp}$, we are interested in the following linear operators
$$\langle X_1,X\rangle^w\ldots\langle X_m,X\rangle^we^{-ta^w_t},$$
where the operators $\langle X_j,X\rangle^w$ are defined in \eqref{23012019E5}. To deal with them, we will use the Fourier integral operator representation of the operators $e^{-ta^w_t}$ and the Egorov formula \eqref{17052019E6}. More precisely, it follows from \eqref{14052019E1} and Proposition \ref{13112018P2} that the operator $e^{-ta^w_t}$ is a Fourier integral operator associated with the nonnegative complex symplectic transformation $e^{-2itJA_t}$, and the Egorov formula \eqref{17052019E6} implies that for all $0\le t\le T$ and $X_0\in\mathbb R^n$,
\begin{equation}\label{09052019E10}
	\langle X_0,X\rangle^we^{-ta^w_t}=e^{-ta^w_t}\langle J^{-1}e^{2itJA_t}JX_0,X\rangle^w = e^{-ta^w_t}\langle e^{2itA_tJ}X_0,X\rangle^w.
\end{equation}
By using \eqref{09052019E10} and the semigroup property of the family of linear operators $(e^{-sa^w_t})_{s\geq0}$, we obtain the following factorization
\begin{align}\label{09052019E11}
	\langle X_1,X\rangle^w\ldots\langle X_m,X\rangle^we^{-ta^w_t} 
	& = \langle X_1,X\rangle^w\ldots\langle X_m,X\rangle^w\underbrace{e^{-\frac tma^w_t}\ldots e^{-\frac tma^w_t}}_{\text{$m$ factors}} \\
	& = \langle Y_{1,t},X\rangle^we^{-\frac tma^w_t}\ldots\langle Y_{m,t},X\rangle^we^{-\frac tma^w_t}, \nonumber
\end{align}
where, for $1\le j\le m$, $Y_{j,t} = e^{\frac{2i(j-1)t}mA_tJ}X_j$. The initial problem is therefore reduced to the analysis of the operators $\langle Y_{j,t},X\rangle^we^{-\frac tma^w_t}.$
The main instrumental result of this section is Lemma \ref{07112018L1} which requires some technical results to be proven. The first of them investigates the anisotropic coercivity properties of the time-dependent quadratic form $\kappa_t$ on $S^{\perp}$ the canonical Euclidean orthogonal complement of the singular space $S$. This is a refinement  of Lemma \ref{13112018L1}.

\begin{lem}\label{23102018L1} There exists a positive constant $c>0$ such that for all $0\le t\le 1$, $X_0\in S^{\perp}\setminus\{0\}$ and $X\in\mathbb C^{2n}$,
$$\kappa_t(X)\geq \frac c{\vert X_0\vert^2}\ t^{2k_{X_0}}\ \big\vert\langle X_0,X\rangle\big\vert^2,$$
where $0\le k_{X_0}\le k_0$ denotes the index of the point $X_0\in S^{\perp}$ with respect to the singular space defined in \eqref{23012019E4}.
\end{lem}

\begin{proof} For all $0\le k\le k_0$, let $r_k$ be the nonnegative quadratic form defined on the phase space by 
\begin{equation}\label{12092018E2}
	r_k(X) = \sum_{j=0}^k\Reelle q\big((\Imag F)^jX\big) = \sum_{j=0}^k\big\vert\sqrt{\Reelle Q}(\Imag F)^jX\big\vert^2\geq0,\quad X\in\mathbb R^{2n}.
\end{equation}
Moreover, we consider $V_k$ the vector subspace defined in \eqref{12102018E8}. We begin by proving that there exists a positive constant $c_k>0$ such that
\begin{equation}\label{23102018E8}
	\forall X\in V^{\perp}_k,\quad r_k(X)\geq c_k\vert X\vert^2.
\end{equation}
If a point $X\in V^{\perp}_k$ satisfies $r_k(X) = 0$, we deduce from \eqref{12092018E2} that
$$\forall j\in\{0,\ldots,k\},\quad \sqrt{\Reelle Q}(\Imag F)^jX = 0,$$
and since $F = JQ$ from \eqref{05112018E7}, this implies that $(\Reelle F)(\Imag F)^jX = 0$ for all $0\le j\le k$, that is $X\in V_k$. It then follows that $X = 0$. The nonnegative quadratic form $r_k$ is therefore positive on the vector subspace $V^{\perp}_k$. The estimate \eqref{23102018E8} is then proved.

Now, we consider $X_0\in S^{\perp}\setminus\{0\}$ and $0\le k_{X_0}\le k_0$ the index of the point $X_0$ with respect to the singular space defined in \eqref{23012019E4}. For all $X\in\mathbb R^{2n}$, we decompose $X = X'+X''$ with $X'\in V^{\perp}_{k_{X_0}}$ and $X''\in V_{k_{X_0}}$. Since $X_0\in V^{\perp}_{k_{X_0}}$ and $r_{k_{X_0}}$ is a nonnegative quadratic form which vanishes on the vector subspace $V_{k_{X_0}}$ from \eqref{14022019E9}, \eqref{12102018E8} and \eqref{12092018E2}, we deduce from \eqref{23102018E8} that
\begin{equation}\label{05072018E6}
	\langle X_0,X\rangle^2 
	= \langle X_0,X'\rangle^2\le\vert X_0\vert^2\vert X'\vert^2
	\le \frac{\vert X_0\vert^2}{c_{k_{X_0}}}\ r_{k_{X_0}}(X')
	= \frac{\vert X_0\vert^2}{c_{k_{X_0}}}\ r_{k_{X_0}}(X).
\end{equation}
Setting $c_0 = \min_{0\le k\le k_0}c_k>0$, we deduce from \eqref{15012019E1}, \eqref{12092018E2} and \eqref{05072018E6} that for all $0\le t\le 1$, $X_0\in S^{\perp}\setminus\{0\}$ and $X\in\mathbb R^{2n}$,
$$\kappa_t(X)\geq t^{2k_{X_0}}\ r_{k_{X_0}}(X)\geq\frac{c_0}{\vert X_0\vert^2}\ t^{2k_{X_0}}\ \langle X,X_0\rangle^2,$$
since $0\le k_{X_0}\le k_0$. It follows that for all $0\le t\le 1$, $X_0\in S^{\perp}\setminus\{0\}$ and $X\in\mathbb C^{2n}$,
\begin{multline*}
	\kappa_t(X) = \kappa_t(\Reelle X) + \kappa_t(\Imag X)
	\geq\frac{c_0}{\vert X_0\vert^2}\ t^{2k_{X_0}}\ \langle \Reelle X,X_0\rangle^2 + \frac{c_0}{\vert X_0\vert^2}\ t^{2k_{X_0}}\ \langle \Imag X,X_0\rangle^2 \\
	= \frac{c_0}{\vert X_0\vert^2}\ t^{2k_{X_0}}\ \big\vert\langle X,X_0\rangle\big\vert^2.
\end{multline*}
This ends the proof of Lemma \ref{23102018L1}.
\end{proof}

The next result will be instrumental to prove Lemma \ref{07112018L1}. Its proof is based on the study of a time-dependent functional.

\begin{lem}\label{23102018L2} For all $s>0$, $t\geq0$ and $u\in\mathscr S(\mathbb R^n)$, the following estimate holds
$$\big\langle a^w_te^{-sa^w_t}u,e^{-sa^w_t}u\big\rangle_{L^2(\mathbb R^n)}\le\frac1{2s} \Vert u\Vert^2_{L^2(\mathbb R^n)}.$$
\end{lem}

\begin{proof} For fixed $t\geq0$ and $u\in\mathscr S(\mathbb R^n)$, we consider the following time-dependent functional defined for all $s\geq0$ by
\begin{equation}\label{23102018E13}
	G(s) = \big\langle sa^w_te^{-sa^w_t}u,e^{-sa^w_t}u\big\rangle_{L^2(\mathbb R^n)}
	+ \frac12\big\Vert e^{-sa^w_t}u\big\Vert^2_{L^2(\mathbb R^n)}.
\end{equation}
The function $G$ is differentiable on $(0,+\infty)$ and its derivative is given for all $s>0$ by
\begin{align*}
	G'(s) = &\ -s\big\langle (a^w_t)^2e^{-sa^w_t}u,e^{-sa^w_t}u\big\rangle_{L^2(\mathbb R^n)} 
	- s\big\langle a^w_te^{-sa^w_t}u,a^w_te^{-sa^w_t}u\big\rangle_{L^2(\mathbb R^n)} \\[5pt]
	&\ - \frac12\big\langle a^w_te^{-sa^w_t}u,e^{-sa^w_t}u\big\rangle_{L^2(\mathbb R^n)}
	- \frac12\big\langle e^{-sa^w_t}u,a^w_te^{-sa^w_t}u\big\rangle_{L^2(\mathbb R^n)}  + \big\langle a^w_te^{-sa^w_t}u,e^{-sa^w_t}u\big\rangle_{L^2(\mathbb R^n)}.
\end{align*}
Since $a^w_t$ is a selfadjoint operator (as its Weyl symbol is real-valued), we obtain that for all $s>0$,
\begin{equation}\label{23102018E14}
	G'(s) = -2s\big\Vert a^w_te^{-sa^w_t}u\big\Vert^2_{L^2(\mathbb R^{n})}\le0.
\end{equation}
We therefore deduce that for all $s\geq0$, $t\geq0$ and $u\in \mathscr S(\mathbb R^n)$,
\begin{equation}
	G(s) = \big\langle sa^w_te^{-sa^w_t}u,e^{-sa^w_t}u\big\rangle_{L^2(\mathbb R^n)} 
	+ \frac12\big\Vert e^{-sa^w_t}u\big\Vert^2_{L^2(\mathbb R^n)}
	\le G(0) = \frac12\Vert u\Vert^2_{L^2(\mathbb R^n)}.
\end{equation}
This ends the proof of Lemma \ref{23102018L2}.
\end{proof}

We need the following lemma whose proof, which is an immediate consequence of H\"ormander's classification of quadratic forms and the symplectic invariance of the Weyl quantization, see \cite[Theorem  18.5.9 and Theorem  21.5.3]{MR2304165}, can be found e.g.\ in \cite[Lemma 2.6]{MR3710672}:

\begin{lem}\label{06072018L1} Let $\tilde q:\mathbb R^{2n}\rightarrow\mathbb R_+$ be a nonnegative quadratic form. Then, the quadratic operator $\tilde q^w(x,D_x)$ is accretive, that is
$$\forall u\in\mathscr S(\mathbb R^n),\quad \big\langle\tilde q^w(x,D_x)u,u\big\rangle_{L^2(\mathbb R^n)}\geq0.$$
\end{lem}

The anisotropic estimates given by Lemma \ref{23102018L1}, combined with Lemma \ref{23102018L2}, provide a first regularizing effect for the evolution operators $e^{-sa^w_t}$.

\begin{lem}\label{07112018L1} There exist some positive constants $0<t_1<T$ and $c>0$ such that for all $0\le\alpha\le1$, $0<t\le t_1$, $s>0$, $X_0\in S^{\perp}$ and $u\in L^2(\mathbb R^n)$,
$$\big\Vert\langle e^{2i\alpha tA_tJ}X_0,X\rangle^we^{-sa^w_t}u\big\Vert_{L^2(\mathbb R^n)}
\le c\vert X_0\vert\ t^{-k_{X_0}}s^{-\frac12}\ \Vert u\Vert_{L^2(\mathbb R^n)},$$
where $0\le k_{X_0}\le k_0$ denotes the index of the point $X_0\in S^{\perp}$ with respect to the singular space defined in \eqref{23012019E4}.
\end{lem}

\begin{proof} We shall first prove that there exist some positive constants $c_0>0$ and $0<t_0<T$ such that for all $0\le\alpha\le1$, $0< t\le t_0$, $X_0\in S^{\perp}$ and $X\in\mathbb R^{2n}$,
\begin{equation}\label{20112018E3}
	\big\vert\langle e^{2i\alpha tA_tJ}X_0,X\rangle\big\vert^2\le c_0\vert X_0\vert^2\ t^{-2k_{X_0}}\ a_t(X),
\end{equation}
where $0\le k_{X_0}\le k_0$ denotes the index of the point $X_0\in S^{\perp}$ with respect to the singular space defined in \eqref{23012019E4}. If the estimate \eqref{20112018E3} holds, the proof of Lemma \ref{07112018L1} is done. Indeed, by denoting $M_{\alpha,t}=\Reelle(e^{2i\alpha tA_tJ})$, we deduce from \eqref{20112018E3} that for all $0\le\alpha\le1$, $0< t\le t_0$, $X_0\in S^{\perp}$ and $X\in\mathbb R^{2n}$,
\begin{equation}\label{20112018E4}
	\langle M_{\alpha,t}X_0,X\rangle^2\le c_0\vert X_0\vert^2\ t^{-2k_{X_0}}\ a_t(X).
\end{equation}
It then follows from \eqref{20112018E4} and Lemma \ref{06072018L1} that for all $0\le\alpha\le1$, $0<t\le t_0$, $s\geq0$, $X_0\in S^{\perp}$ and $u\in\mathscr S(\mathbb R^n)$,
\begin{equation}\label{20112018E5}
	\big\langle\big(\langle M_{\alpha,t}X_0,X\rangle^2\big)^we^{-sa^w_t}u,e^{-sa^w_t}u\big\rangle_{L^2(\mathbb R^n)}
	\le c_0\vert X_0\vert^2\ t^{-2k_{X_0}}\ \big\langle a^w_te^{-sa^w_t}u,e^{-sa^w_t}u\big\rangle_{L^2(\mathbb R^n)}.
\end{equation}
Moreover, the Weyl calculus, see e.g.\ the composition formula (18.5.4) in \cite{MR2304165}, provides that for all $0\le\alpha\le1$ and $0<t\le t_0$,
\begin{equation}
	\langle M_{\alpha,t}X_0,X\rangle^2 = \langle M_{\alpha,t}X_0,X\rangle\ \sharp^w\ \langle M_{\alpha,t}X_0,X\rangle,
\end{equation}
since the symbol $\langle M_{\alpha,t}X_0,X\rangle$ is a linear form, where $\sharp^w$ denotes the Moyal product defined for all $p_1$ and $p_2$ in proper symbol classes by 
$$(p_1\ \sharp^w\ p_2)(x,\xi) = e^{\frac i2\sigma(D_x,D_{\xi};D_y,D_{\eta})}p_1(x,\xi)p_2(y,\eta)\Big\vert_{(x,\xi)=(y,\eta)},$$
with $\sigma$ the symplectic form defined in \eqref{09052019E12}. This implies that for all $0\le\alpha\le1$ and $0<t\le t_0$,
\begin{equation}\label{20112018E6}
	\big(\langle M_{\alpha,t}X_0,X\rangle^2\big)^w = \langle M_{\alpha,t}X_0,X\rangle^w\langle M_{\alpha,t}X_0,X\rangle^w.
\end{equation}
We deduce from \eqref{20112018E5} and \eqref{20112018E6} that for all $0\le\alpha\le1$, $0<t\le t_0$, $s>0$, $X_0\in S^{\perp}$ and $u\in\mathscr S(\mathbb R^n)$,
$$\big\Vert\langle M_{\alpha,t}X_0,X\rangle^we^{-sa^w_t}u\big\Vert^2_{L^2(\mathbb R^n)}
\le c_0\vert X_0\vert^2\ t^{-2k_{X_0}}\ \big\langle a^w_te^{-sa^w_t}u,e^{-sa^w_t}u\big\rangle_{L^2(\mathbb R^n)},$$
and Lemma \ref{23102018L2} then shows that
\begin{equation}\label{20112018E7}
	\big\Vert\langle M_{\alpha,t}X_0,X\rangle^we^{-sa^w_t}u\big\Vert^2_{L^2(\mathbb R^n)}\le\frac{c_0}2\vert X_0\vert^2\ t^{-2k_{X_0}}s^{-1}\ \Vert u\Vert^2_{L^2(\mathbb R^n)}.
\end{equation}
Notice that the estimate \eqref{20112018E7} can be extended to all $u\in L^2(\mathbb R^n)$ since the Schwartz space $\mathscr S(\mathbb R^n)$ is dense in $L^2(\mathbb R^n)$. Similarly, if we denote $N_{\alpha,t}=\Imag(e^{2i\alpha tA_tJ})$, we have that for all $0\le\alpha\le1$, $0<t\le t_0$, $s>0$, $X_0\in S^{\perp}$ and $u\in L^2(\mathbb R^n)$,
\begin{equation}\label{20112018E8}
	\big\Vert\langle N_{\alpha,t}X_0,X\rangle^we^{-sa^w_t}u\big\Vert^2_{L^2(\mathbb R^n)}\le\frac{c_0}2\vert X_0\vert^2\ t^{-2k_{X_0}}s^{-1}\ \Vert u\Vert^2_{L^2(\mathbb R^n)}.
\end{equation}
Finally, we deduce from the triangle inequality that for all $0\le\alpha\le1$, $0<t\le t_0$, $s>0$, $X_0\in S^{\perp}$ and $u\in L^2(\mathbb R^n)$,
$$\big\Vert\langle e^{2i\alpha tA_tJ}X_0,X\rangle^we^{-sa^w_t}u\big\Vert_{L^2(\mathbb R^n)}
\le \big\Vert\langle M_{\alpha,t}X_0,X\rangle^we^{-sa^w_t}u\big\Vert_{L^2(\mathbb R^n)}
+ \big\Vert\langle N_{\alpha,t}X_0,X\rangle^we^{-sa^w_t}u\big\Vert_{L^2(\mathbb R^n)},$$
and the estimates \eqref{20112018E7} and \eqref{20112018E8} imply that
$$\big\Vert\langle e^{2i\alpha tA_tJ}X_0,X\rangle^we^{-sa^w_t}u\big\Vert_{L^2(\mathbb R^n)}\le \sqrt{2c_0}\vert X_0\vert\ t^{-k_{X_0}}s^{-\frac12}\ \Vert u\Vert_{L^2(\mathbb R^n)}.$$
It therefore remains to prove that the estimate \eqref{20112018E3} actually holds. We shall actually prove that there exist some positive constants $c_1>0$ and $0<t_1<T$ such that for all $0\le\alpha\le1$, $0<t\le t_1$, $X_0\in S^{\perp}$ and $X\in\mathbb R^{2n}$,
\begin{equation}\label{20112018E9}
	\big\vert\langle e^{2i\alpha tA_tJ}X_0,X\rangle\big\vert^2\le c_1\vert X_0\vert^2\ t^{-2k_{X_0}}\ \kappa_t(X).
\end{equation}
The estimate \eqref{20112018E3} is then a straightforward consequence of \eqref{15012019E2} and \eqref{20112018E9}. It follows from Lemma \ref{23102018L1} that there exists a positive constant $c_2>0$ such that for all $0\le t\le 1$, $X_0\in S^{\perp}$ and $X\in\mathbb C^{2n}$,
\begin{equation}\label{06112018E5}
	t^{2k_{X_0}}\big\vert\langle X_0,X\rangle\big\vert^2\le c_2\vert X_0\vert^2\kappa_t(X).
\end{equation}
On the other hand, we recall from \eqref{08012019E9} that for all $0\le\alpha\le1$ and $0\le t\le T$,
$$e^{2i\alpha tJA_t} = \exp\Big(-\frac{\alpha}2\Log\big(e^{-2itF}e^{-2it\overline F}\big)\Big).$$
We would like to deduce from Lemma \ref{05112018L2} applied with the functions
\begin{equation}\label{03092019E4}
	G_{\alpha}(M,N) = \exp\Big(-\frac{\alpha}2\Log\big(e^{-2i(M+iN)}e^{-2i(M-iN)}\big)\Big),\quad \alpha\in[0,1],
\end{equation}
that there exist some positive constants $0<t_1<T$ and $c_3>0$ such that for all $0\le\alpha\le1$, $0\le t\le t_1$ and $X\in\mathbb C^{2n}$,
\begin{equation}\label{06112018E7}
	\kappa_t(X)\le c_3\kappa_t\big(e^{2i\alpha tJA_t}X\big).
\end{equation}
This application of Lemma \ref{05112018L2} is made rigorous in Lemma \ref{06092019L1} in the Appendix, which implies that the estimate \eqref{06112018E7} actually holds. Combining \eqref{06112018E5} and \eqref{06112018E7}, we obtain that for all $0\le\alpha\le1$, $0\le t\le t_1$, $X_0\in S^{\perp}$ and $X\in\mathbb C^{2n}$,
$$t^{2k_{X_0}}\ \big\vert\langle X_0,X\rangle\big\vert^2\le c_2c_3\vert X_0\vert^2 \kappa_t\big(e^{2i\alpha tJA_t}X\big),$$
and a straightforward change of variable shows that for all $0\le\alpha\le1$, $0\le t\le t_1$, $X_0\in S^{\perp}$ and $X\in\mathbb R^{2n}$,
$$t^{2k_{X_0}}\ \big\vert\langle e^{2i\alpha tA_tJ}X_0,X\rangle\big\vert^2\le c_2c_3\vert X_0\vert^2 \kappa_t(X).$$
This proves that \eqref{20112018E9} holds and ends the proof of Lemma \ref{07112018L1}.
\end{proof}

We can now derive the proof of Theorem \ref{22102018T1}. To that end, we implement the strategy presented in the beginning of this subsection. Let $m\geq1$ and $X_1,\ldots,X_m\in S^{\perp}$. We denote by $0\le k_{X_j}\le k_0$ the index of the point $X_j\in S^{\perp}$ with respect to the singular space. It follows from \eqref{09052019E11} that for all $0\le t\le T$,
\begin{equation}\label{06112018E3}
	\langle X_1,X\rangle^w\ldots\langle X_m,X\rangle^we^{-ta^w_t} = \langle Y_{1,t},X\rangle^we^{-\frac tma^w_t}\ldots\langle Y_{m,t},X\rangle^we^{-\frac tma^w_t},
\end{equation}
where, for all $1\le j\le m$, $Y_{j,t} = e^{\frac{2i(j-1)t}mA_tJ}X_j.$  
According to Lemma \ref{07112018L1}, there exist some positive constants $0<t_1<T$ and $c>0$ such that for all $0\le\alpha\le1$, $0<t\le t_1$, $s>0$, $X_0\in S^{\perp}$ and $u\in L^2(\mathbb R^n)$,
\begin{equation}\label{20112018E13}
	\big\Vert\langle e^{2i\alpha tA_tJ}X_0,X\rangle^we^{-sa^w_t}u\big\Vert_{L^2(\mathbb R^n)}
	\le c\vert X_0\vert\ t^{-k_{X_0}}s^{-\frac12}\ \Vert u\Vert_{L^2(\mathbb R^n)},
\end{equation}
where $0\le k_{X_0}\le k_0$ denotes the index of the point $X_0\in S^{\perp}$ with respect to the singular space. We deduce from \eqref{20112018E13} that for all $1\le j\le m$, $0<t\le t_1$ and $u\in L^2(\mathbb R^n)$,
\begin{equation}\label{19112018E3}
	\big\Vert\langle Y_{j,t},X\rangle^we^{-\frac tma^w_t}u\big\Vert_{L^2(\mathbb R^n)}
	\le c\vert X_j\vert\ t^{-k_{X_j}-\frac12}\ m^{\frac12}\ \Vert u\Vert_{L^2(\mathbb R^n)}.
\end{equation}
Notice that the constant $c>0$ is independent on the integer $m\geq1$ and the points $X_j\in S^{\perp}$. It now follows from \eqref{06112018E3}, \eqref{19112018E3} and a straightforward induction that for all $0<t\le t_1$ and $u\in L^2(\mathbb R^n)$,
\begin{align*}
	\big\Vert\langle X_1,X\rangle^w\ldots\langle X_m,X\rangle^we^{-ta^w_t}u\big\Vert_{L^2(\mathbb R^n)}
	& \le \frac{c^m}{t^{k_{X_1}+\ldots+k_{X_m}+\frac m2}}\Bigg[\prod_{j=1}^m\vert X_j\vert\Bigg]\ m^{\frac m2}\ \Vert u\Vert_{L^2(\mathbb R^n)} \\
	& \le \frac{e^{\frac m2}c^m}{t^{k_{X_1}+\ldots+k_{X_m}+\frac m2}}\Bigg[\prod_{j=1}^m\vert X_j\vert\Bigg]\ (m!)^{\frac12}\ \Vert u\Vert_{L^2(\mathbb R^n)},
\end{align*}
where we used that $m^m\le e^m m!$. We then deduce from \eqref{19112018E1} that for all $0<t\le t_1$ and $u\in L^2(\mathbb R^n)$,
\begin{align*}
	\big\Vert\langle X_1,X\rangle^w\ldots\langle X_m,X\rangle^we^{-tq^w}u\big\Vert_{L^2(\mathbb R^n)}
	& \le \frac{e^{\frac m2}c^m}{t^{k_{X_1}+\ldots+k_{X_m}+\frac m2}}\Bigg[\prod_{j=1}^m\vert X_j\vert\Bigg]\ (m!)^{\frac12}\ \big\Vert e^{-itb^w_t}u\big\Vert_{L^2(\mathbb R^n)} \\
	& = \frac{e^{\frac m2}c^m}{t^{k_{X_1}+\ldots+k_{X_m}+\frac m2}}\Bigg[\prod_{j=1}^m\vert X_j\vert\Bigg]\ (m!)^{\frac12}\ \Vert u\Vert_{L^2(\mathbb R^n)},
\end{align*}
since the operators $e^{-itb_t}$ are unitary on $L^2(\mathbb R^n)$. This ends the proof of Theorem \ref{22102018T1}.

\subsection{Directions of regularity} We now perform the proof of Theorem \ref{10122018T1}. The family $(a_t)_{t\in\mathbb R}$ still stands for the family given by Theorem \ref{12102018T1} composed of nonnegative quadratic forms $a_t:\mathbb R^{2n}\rightarrow\mathbb R_+$ with coefficients depending analytically on the time-variable $t\in\mathbb R$. As in the previous subsection, the matrix of the quadratic forms $a_t$ in the canonical basis of $\mathbb R^{2n}$ is denoted $A_t$. Moreover, we consider $(U_t)_{t\in\mathbb R}$ the family of metaplectic operators also given by Theorem \ref{12102018T1}. We recall that the evolution operators $e^{-tq^w}$ split as
\begin{equation}\label{15012019E3}
	\forall t\geq0,\quad e^{-tq^w} = e^{-ta^w_t}U_t.
\end{equation}
Let $t>0$, $X_0\in\mathbb R^{2n}$. We assume that the linear operator $\langle X_0,X\rangle^we^{-tq^w}$ is bounded on $L^2(\mathbb R^n)$. We aim at proving that $X_0\in S^{\perp}$. We first notice that since the metaplectic operator $U_t$ is unitary on $L^2(\mathbb R^n)$, it follows from \eqref{15012019E3} that the linear operator $\langle X_0,X\rangle^we^{-ta^w_t}$ is also bounded on $L^2(\mathbb R^n)$. As a consequence, there exists a positive constant $c_{t,X_0}>0$ depending on $t$ and $X_0$ such that
\begin{equation}\label{10122018E1}
	\forall u\in L^2(\mathbb R^n),\quad \big\Vert\langle X_0,X\rangle^we^{-ta^w_t}u\big\Vert_{L^2(\mathbb R^n)}\le c_{t,X_0}\Vert u\Vert_{L^2(\mathbb R^n)}.
\end{equation}
According to the decomposition $\mathbb R^{2n} = S\oplus S^{\perp}$ of the phase space, the orthogonality being taken with respect to the euclidean structure of $\mathbb R^{2n}$, we write $X_0 = X_{0,S} + X_{0,S^{\perp}}$, with $X_{0,S}\in S$ and $X_{0,S^{\perp}}\in S^{\perp}$. For all $\lambda\geq0$, we consider $X_{\lambda}\in S$ the point of the singular space defined by
\begin{equation}\label{10122018E3}
	X_{\lambda} = \lambda X_{0,S} = (x_{\lambda},\xi_{\lambda})\in S\subset\mathbb R^{2n}.
\end{equation}
Moreover, we consider for all $\lambda\geq0$ the Gaussian function $u_{\lambda}\in\mathscr S(\mathbb R^n)$ given for all $x\in\mathbb R^n$ by
\begin{equation}\label{10122018E2}
	u_{\lambda}(x) = e^{i\langle\xi_{\lambda},x\rangle}e^{-\vert x-x_{\lambda}\vert^2}.
\end{equation}
The strategy will be to find upper and lower bounds for the term 
\begin{equation}\label{09052019E2}
	\big\langle\langle X_0,X\rangle^we^{-ta^w_t}u_{\lambda},u_{\lambda}\big\rangle_{L^2(\mathbb R^n)},
\end{equation}
and to consider the asymptotics when $\lambda$ tends to $+\infty$ in order to conclude that the point $X_{0,S}$ has to be equal to zero. An upper bound can be established readily since it follows from \eqref{10122018E1}, \eqref{10122018E2} and the Cauchy-Schwarz inequality that for all $\lambda\geq0$,
\begin{equation}\label{09052019E1}
	\big\vert\big\langle\langle X_0,X\rangle^we^{-ta^w_t}u_{\lambda},u_{\lambda}\big\rangle_{L^2(\mathbb R^n)}\big\vert\le c_{t,X_0}\Vert u_{\lambda}\Vert^2_{L^2(\mathbb R^n)}
	= c_{t,X_0}\Vert u_0\Vert^2_{L^2(\mathbb R^n)}.
\end{equation}
Notice that the right-hand side of the above estimate does not depend on the parameter $\lambda\geq0$. Now, we investigate a lower bound for the term \eqref{09052019E2} by a direct calculus. It follows from the Mehler formula (Corollary \ref{09012019C1}) that the operator $e^{-ta^w_t}$  is a pseudodifferential operator whose symbol is given by
\begin{equation}\label{07012019E2}
	c_te^{-tm_t(X)}\in L^{\infty}(\mathbb R^{2n}),\quad\text{where}\quad c_t = (\det\cos(tJA_t) )^{-\frac12}>0,
\end{equation}
and where $m_t:\mathbb R^{2n}\rightarrow\mathbb R_+$ is the nonnegative quadratic form whose matrix in the canonical basis of $\mathbb R^{2n}$ is the matrix $M_t$ defined in \eqref{08012019E1}. We therefore deduce from \eqref{10122018E2} and \eqref{07012019E2} that the term \eqref{09052019E2} is given for all $\lambda\geq0$ by
\begin{equation}\label{09052019E5}
	\big\langle\langle X_0,X\rangle^we^{-ta^w_t}u_{\lambda},u_{\lambda}\big\rangle_{L^2(\mathbb R^n)}
	= c_t\big\langle\langle X_0,X\rangle^w(e^{-tm_t})^wT_{\lambda}u_0,T_{\lambda}u_0\big\rangle_{L^2(\mathbb R^n)},
\end{equation}
where the operator $T_{\lambda}:L^2(\mathbb R^n)\rightarrow L^2(\mathbb R^n)$ is defined for all $u\in L^2(\mathbb R^n)$ by
\begin{equation}\label{09052019E6}
	T_{\lambda}u = e^{i\langle\xi_{\lambda},\cdot\rangle}u(\cdot-x_{\lambda}).
\end{equation}
We need compute the commutators between the operators $T_{\lambda}$ and the operators $\langle X_0,X\rangle^w$ and $(e^{-tm_t})^w$ respectively. This is done in the following lemma: 

\begin{lem}\label{09052019L1} Let $a\in\mathscr S'(\mathbb R^{2n})$. We have that for all $\lambda\geq0$ and $u\in\mathscr S(\mathbb R^n)$,
$$a^wT_{\lambda}u = T_{\lambda}(L_{\lambda}a)^wu\quad\text{in $\mathscr S'(\mathbb R^n)$},$$
where $L_{\lambda}a\in\mathscr S'(\mathbb R^n)$ is given by $L_{\lambda}a = a(\cdot+X_{\lambda})$.
\end{lem}

\begin{proof} Let $\lambda\geq0$ and $u\in\mathscr S(\mathbb R^n)$ be a Schwartz function. For all $v\in\mathscr S(\mathbb R^n)$, we consider the Wigner function $\mathcal H_{\lambda}(u,v)$ associated with the functions $T_{\lambda}u$ and $T_{\lambda}v$ defined for all $(x,\xi)\in\mathbb R^{2n}$ by
\begin{equation}\label{11012019E1}
	\mathcal H_{\lambda}(u,v)(x,\xi) = \int_{\mathbb R^n}e^{-i\langle y,\xi\rangle}(T_{\lambda}u)\Big(x+\frac y2\Big)(\overline{T_{\lambda}v})\Big(x-\frac y2\Big)\ \mathrm dy.
\end{equation}
It follows from \eqref{09052019E6} and \eqref{11012019E1} that for all $\lambda\geq0$, $v\in\mathscr S(\mathbb R^n)$ and $(x,\xi)\in\mathbb R^{2n}$,
\begin{align}\label{11012019E2}
	\mathcal H_{\lambda}(u,v)(x,\xi) & = \int_{\mathbb R^n}e^{-i\langle y,\xi\rangle}e^{i\langle\xi_{\lambda}, x+\frac y2\rangle}u\Big(x+\frac y2 - x_{\lambda}\Big)
	e^{-i\langle\xi_{\lambda}, x-\frac y2\rangle}\overline v\Big(x-\frac y2 - x_{\lambda}\Big)\ \mathrm dy \\
	& = \int_{\mathbb R^n}e^{-i\langle y,\xi-\xi_{\lambda}\rangle}u\Big(x- x_{\lambda}+\frac y2\Big)\overline v\Big(x- x_{\lambda}-\frac y2\Big)\ \mathrm dy \nonumber \\[5pt]
	& = \mathcal H_0(u,v)(x-x_{\lambda},\xi-\xi_{\lambda}) = (L^{-1}_{\lambda}\mathcal H(u,v))(x,\xi), \nonumber
\end{align}
since $T_0$ is the identity operator. It then follows from \eqref{11012019E2} and the definition of the Weyl calculus that for all $v\in\mathscr S(\mathbb R^n)$,
\begin{align*}
	\big\langle T_{\lambda}^*a^wT_{\lambda}u,\overline v\big\rangle_{\mathscr S'(\mathbb R^n),\mathscr S(\mathbb R^n)} 
	& = \big\langle a^wT_{\lambda}u,\overline{T_{\lambda}v}\big\rangle_{\mathscr S'(\mathbb R^n),\mathscr S(\mathbb R^n)}
	= \langle a,\mathcal H_{\lambda}(u,v)\rangle_{\mathscr S'(\mathbb R^{2n}),\mathscr S(\mathbb R^{2n})} \\[5pt]
	& =  \langle a,L^{-1}_{\lambda}\mathcal H_0(u,v)\rangle_{\mathscr S'(\mathbb R^{2n}),\mathscr S(\mathbb R^{2n})}
	=  \langle L_{\lambda}a,\mathcal H_0(u,v)\rangle_{\mathscr S'(\mathbb R^{2n}),\mathscr S(\mathbb R^{2n})} \\[5pt]
	& =  \langle (L_{\lambda}a)^wu,\overline v\rangle_{\mathscr S'(\mathbb R^n),\mathscr S(\mathbb R^n)}.
\end{align*}
Since the above estimate holds for all Schwartz functions $v\in\mathscr S(\mathbb R^n)$, we proved that $T_{\lambda}^*a^wT_{\lambda}u = (L_{\lambda}a)^wu$ in $\mathscr S'(\mathbb R^n)$. As $T_{\lambda}T_{\lambda}^*$ is the identity operator, we obtain that $a^wT_{\lambda}u = T_{\lambda}(L_{\lambda}a)^wu$ in $\mathscr S'(\mathbb R^n)$. This ends the proof of Lemma \ref{09052019L1}.
\end{proof}

\noindent The quadratic form $m_t$ vanishes on the singular space $S$. Indeed, if $X\in S$, we recall from \eqref{09052019E4} that $m_s(X) = 0$ when $0\le s\ll1$ and since the function $s\in\mathbb R\mapsto m_s(X)$ is analytic, see \eqref{08012019E1} where the matrices $M_s$ are constructed, we deduce that $m_s(X) = 0$ for all $s\geq0$. Since the quadratic forms $m_t$ are positive semidefinite from Lemma \ref{08012019L2} and the points $X_{\lambda}$ are elements of $S$, we deduce that
$$\forall \lambda\geq0, \forall X\in\mathbb R^{2n},\quad (L_{\lambda}m_t)(X) = m_t(X+X_{\lambda}) = m_t(X).$$
We therefore deduce from \eqref{07012019E2} and Lemma \ref{09052019L1} that for all $\lambda\geq0$ and $u\in\mathscr S(\mathbb R^n)$,
\begin{equation}\label{09052019E7}
	(e^{-tm_t})^wT_{\lambda}u = T_{\lambda}(L_{\lambda}e^{-tm_t})^wu = T_{\lambda}(e^{-tm_t})^wu = c_t^{-1} T_{\lambda}e^{-ta^w_t}u,\quad \text{in $\mathscr S'(\mathbb R^n)$}.
\end{equation}
Moreover, \cite[Theorem 4.2]{MR1339714} states that for all $s\geq0$, the evolution operator ${e^{-s\tilde q^w}}$ generated by an accretive quadratic operator $\tilde q^w(x,D_x)$, with $\tilde q:\mathbb R^{2n}\rightarrow\mathbb C$ a complex-valued quadratic form with a nonnegative real-part $\Reelle\tilde q\geq0$, maps $\mathscr S(\mathbb R^n)$ into $\mathscr S(\mathbb R^n)$: i.e.\ $\forall u\in\mathscr S(\mathbb R^n),\quad e^{-s\tilde q^w}u\in\mathscr S(\mathbb R^n).$
This implies that $T_{\lambda}e^{-ta^w_t}u\in\mathscr S(\mathbb R^n)$ for all $\lambda\geq0$ and $u\in\mathscr S(\mathbb R^n)$ and that the equality \eqref{09052019E7} holds in $\mathscr S(\mathbb R^n)$. On the other hand, it follows from Lemma \ref{09052019L1} anew that 
\begin{equation}\label{09052019E8}
	\forall \lambda\ge 0, \forall u\in\mathscr S(\mathbb R^n), \ \langle X_0,X\rangle^wT_{\lambda}u = T_{\lambda}\langle X_0,X+X_{\lambda}\rangle^wu, \quad \text{in $\mathscr S'(\mathbb R^n)$}.
\end{equation}
Since the right-hand side of the above formula belongs to the Schwartz space $\mathscr S(\mathbb R^n)$ for all $\lambda\geq0$ and $u\in\mathscr S(\mathbb R^n)$, the equality \eqref{09052019E8} holds in $\mathscr S(\mathbb R^n)$. As a consequence of \eqref{09052019E5}, \eqref{09052019E7} and \eqref{09052019E8}, we have that for all $\lambda\geq0$,
\begin{multline*}
	\big\langle\langle X_0,X\rangle^we^{-ta^w_t}u_{\lambda},u_{\lambda}\big\rangle_{L^2(\mathbb R^n)}
	 = \big\langle\langle X_0,X\rangle^wT_{\lambda}e^{-ta^w_t}u_0,T_{\lambda}u_0\big\rangle_{L^2(\mathbb R^n)} \\[5pt]
	 = \big\langle T_{\lambda}\langle X_0,X+X_{\lambda}\rangle^we^{-ta^w_t}u_0,T_{\lambda}u_0\big\rangle_{L^2(\mathbb R^n)} 
	 = \big\langle\langle X_0,X+X_{\lambda}\rangle^we^{-ta^w_t}u_0,u_0\big\rangle_{L^2(\mathbb R^n)},
\end{multline*}
since the operators $T_{\lambda}$ are unitary on $L^2(\mathbb R^n)$. Moreover, it follows from \eqref{10122018E3} that for all $\lambda\geq0$ and $X\in\mathbb R^{2n}$,
$$\langle X_0,X+X_{\lambda}\rangle^we^{-ta^w_t} = \langle X_0,X\rangle^we^{-ta^w_t} + \lambda\vert X_{0,S}\vert^2e^{-ta^w_t}.$$
This proves that for all $\lambda\geq0$,
$$\big\langle\langle X_0,X\rangle^we^{-ta^w_t}u_{\lambda},u_{\lambda}\big\rangle_{L^2(\mathbb R^n)} = 
\big\langle\langle X_0,X\rangle^we^{-ta^w_t}u_0,u_0\big\rangle_{L^2(\mathbb R^n)}
+\lambda\vert X_{0,S}\vert^2\big\langle e^{-ta^w_t}u_0,u_0\big\rangle_{L^2(\mathbb R^n)}.$$
Combining the above estimate with \eqref{09052019E1}, we obtain that for all $\lambda\geq0$,
$$\lambda\vert X_{0,S}\vert^2\big\vert\big\langle e^{-ta^w_t}u_0,u_0\big\rangle_{L^2(\mathbb R^n)}\big\vert
\le \big\vert\big\langle\langle X_0,X\rangle^we^{-ta^w_t}u_0,u_0\big\rangle_{L^2(\mathbb R^n)}\big\vert
+ c_{t,X_0}\Vert u_0\Vert^2_{L^2(\mathbb R^n)}.$$
We now only need to check that the term $\langle e^{-ta^w_t}u_0,u_0\rangle_{L^2(\mathbb R^n)}$ is not equal to zero to conclude that $X_{0,S} = 0$, since the right-hand side of the above estimate does not depend on the parameter $\lambda\geq0$.
Since $a_t$ is a nonnegative quadratic form, it follows from Corollary \ref{11022019C1} that the operator $e^{-\frac t2a^w_t}$ is injective. As the Gaussian function $u_0\in\mathscr S(\mathbb R^n)$ is non-zero, we deduce that
\begin{equation}\label{07012019E1}
	\big\langle e^{-ta^w_t}u_0,u_0\big\rangle_{L^2(\mathbb R^n)} = \big\Vert e^{-\frac t2a^w_t}u_0\big\Vert^2_{L^2(\mathbb R^n)}\ne0,
\end{equation}
while using the semigroup property of the family of linear selfadjoint operators $(e^{-sa^w_t})_{s\geq0}$. It therefore follows that $X_{0,S} = 0$ and $X_0\in S^{\perp}$. This ends the proof of Theorem \ref{10122018T1}.

\section{Subelliptic estimates enjoyed by quadratic operators}\label{Sub}

This section is devoted to the proof of Theorem \ref{22102018T2}. Let $q:\mathbb R^{2n}\rightarrow\mathbb C$ be a complex-valued quadratic form with a nonnegative real part $\Reelle q\geq0$. We consider $S$ the singular space of $q$ and $0\le k_0\le 2n-1$ the smallest integer such that \eqref{12102018E7} holds. Let $p_k:\mathbb R^{2n}\rightarrow\mathbb R$ be the nonnegative quadratic form given by \eqref{26102018E1} and $\Lambda^2_k$ be the operator defined in \eqref{25102018E12}, with $0\le k\le k_0$. To prove Theorem \ref{22102018T2}, we will use the interpolation theory as in \cite[Subsection 2.4]{MR3710672} which will allow to derive subelliptic estimates for the quadratic operator $q^w(x,D_x)$ from estimates for the evolution operators $e^{-tq^w}$. In the following, several estimates will involve the operators $\Lambda^4_k$ and we recall from the theory of positive operators, see e.g.\ \cite[Section 4]{MR2523200}, that they are positive operators whose domains are given by
$D(\Lambda^4_k) = \big\{u\in L^2(\mathbb R^n) : \Lambda^4_ku\in L^2(\mathbb R^n)\big\}.$

First of all, we need to prove some additional estimates for the semigroup $(e^{-tq^w})_{t\geq0}$.

\begin{lem}\label{23102018L3} There exist some positive constants $c>0$ and $\mu>0$ such that for all $0\le k\le k_0$, $t>0$ and $u\in L^2(\mathbb R^n)$,
$$\big\Vert\Lambda^4_ke^{-tq^w}u\big\Vert_{L^2(\mathbb R^n)}\le ce^{\mu t} t^{-4k-2}\ \Vert u\Vert_{L^2(\mathbb R^n)}.$$
\end{lem}

\begin{proof} Let $0\le k\le k_0$. It follows from the Gauss decomposition of nonnegative quadratic forms that there exist a positive integer $N_k\geq1$ and some points $X^k_1,\ldots,X^k_{N_k}\in\mathbb R^{2n}$ such that
\begin{equation}\label{23102018E18}
	\forall X\in \mathbb R^{2n}, \ p_k(X) = \sum_{j=1}^{N_k}\langle X^k_j,X\rangle^2.
\end{equation}
However, by definition of $V_k$ and $p_k$ (\eqref{12102018E8} and \eqref{26102018E1}), we know that $p_k$ vanishes on $V_k$. Consequently, for all $1\le j\le N_k$, we have $\langle X^k_j,X\rangle = 0$ for all $X\in V_k$. The points $X^k_j\in\mathbb R^{2n}$ are therefore elements of $V^{\perp}_k\subset S^{\perp}$ and their associated indexes defined in \eqref{23012019E4}, satisfy $0\le k_{X^k_j}\le k$ for all $1\le j\le N_k$. As we have already noticed, the Weyl calculus shows that for all $1\le j\le N_k$,
$$\Op^w\big(\langle X^k_j,X\rangle^2\big) = (\langle X^k_j,X\rangle^w)^2,$$
and we deduce from \eqref{25102018E12}, \eqref{23102018E18} that
\begin{equation}\label{22112018E3}
	\Lambda^4_k = \Big(1+\sum_{j=1}^{N_k}\langle X^k_j,X\rangle^w\langle X^k_j,X\rangle^w\Big)^2 
	= 1 + 2\sum_{j=1}^{N_k} (\langle X^k_j,X\rangle^w)^{2} + \sum_{j=1}^{N_k}\sum_{\ell=1}^{N_k} (\langle X^k_j,X\rangle^w)^{2} (\langle X^k_{\ell},X\rangle^w)^{2}.
\end{equation}
Since the indices $k_{X^k_j}$ are smaller than or equal to $k$, we deduce of Theorem \ref{22102018T1} that there exist some positive constants $c>0$ and $0<t_0<1$ such that for all $1\le j,\ell\le N_k$, $0<t\le t_0$ and $u\in L^2(\mathbb R^n)$,
\begin{equation*}\label{22112018E1}
	\big\Vert\langle X^k_j,X\rangle^w\langle X^k_j,X\rangle^we^{-tq^w}u\big\Vert_{L^2(\mathbb R^n)}
	\le  \sqrt 2c^2 \ t^{-2k-1}\ \vert X^k_j\vert^2\ \Vert u\Vert_{L^2(\mathbb R^n)},
\end{equation*}
and
\begin{equation*}\label{22112018E2}
	\big\Vert\langle X^k_j,X\rangle^w\langle X^k_j,X\rangle^w\langle X^k_{\ell},X\rangle^w\langle X^k_{\ell},X\rangle^we^{-tq^w}u\big\Vert_{L^2(\mathbb R^n)}
	\le 2\sqrt 6c^4\ t^{-4k-2}\ \vert X^k_j\vert^2\vert X^k_{\ell}\vert^2\ \Vert u\Vert_{L^2(\mathbb R^n)},
\end{equation*}
since $X^k_j,X^k_l\in S^{\perp}$. We deduce that there exists a positive constant $c_k>0$ such that for all $0<t\le t_0$ and $u\in L^2(\mathbb R^n)$,
\begin{equation}\label{22112018E4}
	\big\Vert\Lambda^4_ke^{-tq^w}u\big\Vert_{L^2(\mathbb R^n)}
	\le c_k \ t^{-4k-2}\ \Vert u\Vert_{L^2(\mathbb R^n)}.
\end{equation}
Furthermore, it follows from \eqref{22112018E4} and the contraction semigroup property of the family $(e^{-tq^w})_{t\geq0}$ that for all $t>t_0$ and $u\in L^2(\mathbb R^n)$,
$$\big\Vert\Lambda^4_ke^{-tq^w}u \big\Vert_{L^2(\mathbb R^n)} = \big\Vert\Lambda^4_ke^{-t_0q^w}e^{-(t-t_0)q^w}u \big\Vert_{L^2(\mathbb R^n)} 
\le \frac{c_k}{t_0^{4k+2}}\ \Vert e^{-(t-t_0)q^w}u\Vert_{L^2(\mathbb R^n)} \le \frac{c_k}{t_0^{4k+2}}\ \Vert u\Vert_{L^2(\mathbb R^n)}.$$
Consequently, there exists a positive constant $\mu_k>0$ such that for all $t>0$ and $u\in L^2(\mathbb R^n)$,
$$\big\Vert\Lambda^4_ke^{-tq^w}u\big\Vert_{L^2(\mathbb R^n)}\le c_ke^{\mu_k t} \ t^{-4k-2} \ \Vert u\Vert_{L^2(\mathbb R^n)}.$$
This ends the proof of Lemma \ref{23102018L3}.
\end{proof}

By using some results of interpolation theory, we can now derive Theorem \ref{22102018T2} from Lemma~\ref{23102018L3}. Let $0\le k\le k_0$. We consider $\mathscr H_k$ the Hilbert space defined by
$$\mathscr H_k = D(\Lambda^4_k) = \big\{u\in L^2(\mathbb R^n) : \Lambda^4_k u\in L^2(\mathbb{R}^n)\big\},$$
naturally equipped with the scalar product
$\langle u,v\rangle_{\mathscr H_k} = \big\langle\Lambda^4_k u,\Lambda^4_k v\big\rangle_{L^2(\mathbb{R}^n)}.$
We deduce from Lemma \ref{23102018L3} that there exist some positive constants $c>0$ and $\mu>0$ such that for all $t>0$ and $u\in L^2(\mathbb{R}^n)$,
\begin{equation}\label{23082018E3}
	\big\Vert\Lambda^4_k e^{-tq^w}u\big\Vert_{L^2(\mathbb{R}^n)}\le ce^{\mu t} t^{-4k-2} \ \Vert u\Vert_{L^2(\mathbb{R}^n)}.
\end{equation}
Considering the operator $p^w(x,D_x) = q^w(x,D_x) + \mu$,
the estimate \eqref{23082018E3} can be written as
\begin{equation}\label{27082018E1}
	\forall t>0, \forall u\in L^2(\mathbb R^n),\quad \big\Vert e^{-tp^w}u\big\Vert_{\mathscr H_k}\le  c \ t^{-4k-2}\ \Vert u\Vert_{L^2(\mathbb{R}^n)}.
\end{equation}
It follows from \eqref{27082018E1} and the strong continuity of the semigroup $(e^{-tp^w})_{t\geq0}$ that for all $u\in L^2(\mathbb R^n)$, $t_0>0$ and $t>0$, we have 
$$\big\Vert e^{-(t+t_0)p^w}u - e^{-t_0p^w}u\big\Vert_{\mathcal H_k}
= \big\Vert e^{-t_0p^w}\big(e^{-tp^w}u - u\big)\big\Vert_{\mathcal H_k} \\
\le  c \ t_0^{-4k-2}\big\Vert e^{-tp^w}u - u\big\Vert_{L^2(\mathbb R^n)}
\underset{t\rightarrow0}{\rightarrow}0.$$
This proves that for all $u\in L^2(\mathbb R^n)$, the function $t\in(0,+\infty)\mapsto e^{-tp^w}u\in\mathcal H_k$ is continuous, and therefore measurable.
Moreover, we deduce from \cite[pp. 425-426]{MR1339714} that the operator $p^w(x,D_x)$ equipped with the domain $D(q^w)$ is maximal accretive. Corollary 5.13 in \cite{MR2523200} therefore shows that the following continuous inclusion holds between the domain of the quadratic operator $q^w(x,D_x)$ and $(L^2(\mathbb R^n),\mathscr H_k)_{1/(4k+2),2}$ the space obtained by real interpolation between $L^2(\mathbb R^n)$ and $\mathscr H_k$:
\begin{align}\label{26072017E2}
	D(q^w)\subset \big(L^2(\mathbb R^n),\mathscr H_k\big)_{1/(4k+2),2}.
\end{align}
Since $\mathscr H_k$ is the domain of the operator $\Lambda^4_k$ and that $\Lambda^2_k$ is a positive selfadjoint operator, we deduce from Theorem 4.36 in \cite{MR2523200} that
\begin{equation}\label{26072017E4}
	\big(L^2(\mathbb R^n),\mathscr H_k\big)_{1/(4k+2)} 
	= \big(D((\Lambda^2_k)^0),D((\Lambda^2_k)^2)\big)_{1/(4k+2),2}
	= D\big((\Lambda^2_k)^{\frac2{4k+2}}\big) 
	= D\big(\Lambda^{\frac 2{2k+1}}_k\big).
\end{equation}
We therefore obtain from \eqref{26072017E2} and \eqref{26072017E4} that the following continuous inclusion holds
$$D(q^w)\subset D\big(\Lambda_k^{\frac2{2k+1}}\big).$$
This implies that there exists a positive constant $c_k>0$ such that
$$\forall u\in D(q^w),\quad \big\Vert\Lambda_k^{\frac2{2k+1}}u\big\Vert_{L^2(\mathbb R^n)}\le c_k\big[\Vert p^w(x,D_x)u\Vert_{L^2(\mathbb R^n)} + \Vert u\Vert_{L^2(\mathbb R^n)}\big],$$
and we deduce from the definition of $p^w$ that
$$\forall u\in D(q^w),\quad \big\Vert\Lambda_k^{\frac2{2k+1}}u\big\Vert_{L^2(\mathbb R^n)}\le c_k(1+\mu)\big[\Vert q^w(x,D_x)u\Vert_{L^2(\mathbb R^n)} + \Vert u\Vert_{L^2(\mathbb R^n)}\big].$$

\section{Appendix}\label{Appendix}

\subsection{About the polar decomposition}\label{polardecomposition} To begin this appendix, we recall the basics about the polar decomposition of a bounded operator on a Hilbert space. As a prerequisite, we recall that if $H$ is an Hilbert space and $T\in\mathcal L(H)$ is a nonnegative selfadjoint bounded linear operator, there exists a unique nonnegative selfadjoint bounded operator $\sqrt T\in\mathcal L(H)$ such that $(\sqrt T)^2 = T$, see e.g.\ \cite[Theorem 4.4.2]{MR2668420}. From there, we define the absolute value of any bounded operator $T\in\mathcal L(H)$ as the selfadjoint operator defined by $\vert T\vert = \sqrt{T^*T}$. The operator $\vert T\vert$ satisfies $\Ker\vert T\vert = \Ker T$. Moreover, we recall that a bounded operator $U\in\mathcal L(H)$ is a partial isometry if $\Vert Ux\Vert_H = \Vert x\Vert_H$ for all $x\in (\Ker U)^{\perp}$. We can now state the standard polar decomposition theorem whose proof can be found e.g.\ in \cite[Theorem 4.4.3]{MR2668420}:

\begin{thm}\label{15052019T1} Let $H$ be an Hilbert space and $T\in\mathcal L(H)$ be a bounded linear operator. Then, there exist a unique nonnegative selfadjoint bounded linear operator $S\in\mathcal L(H)$ and a partial isometry $U\in\mathcal L(H)$ such that $T = US$ and $\Ker U = \Ker T$. Moreover, the operator $S$ is given by $S = \vert T\vert$.
\end{thm}

However, the decomposition given by Theorem \ref{15052019T1} is not useful for us. We are more interested here with decompositions of the type $T = \vert T\vert U$. Let us assume that $T\in\mathcal L(H)$ is written as 
\begin{equation}\label{20052019E1}
	T = SU,
\end{equation}
with $S\in\mathcal L(H)$ a nonnegative selfadjoint injective bounded linear operator and $U\in\mathcal L(H)$ be a unitary operator. By passing to the adjoint, we deduce that $T^* = U^*S$. Since the operator $U^*\in\mathcal L(H)$ remains unitary on $H$ and that $\Ker U^* = \Ker T^* = \{0\}$, the operator $T^*$ being injective as a composition of two injective operators, we deduce from Theorem \ref{15052019T1} that such a couple $(U,S)$ is uniquely defined and $S = \vert T^*\vert$. With an abuse of terminology, we call the decomposition \eqref{20052019E1}, when it exists (it will always be the case in this paper), with the bounded linear operators $S$ and $U$ respectively nonnegative selfadjoint injective and unitary, the polar decomposition of the operator $T$.

To end this subsection, let us check that formula \eqref{01022019E1}, namely
$e^{-tq^w} = e^{-ta^w_t}e^{-itb^w_t},$
with $t\geq0$ fixed, the operators $e^{-ta^w_t}$ and $e^{-itb^w_t}$ being defined in \eqref{14052019E1}, is the polar decomposition of the evolution operator $e^{-tq^w}$ generated by the accretive quadratic operator $q^w(x,D_x)$, as defined just before. The operator $e^{-ta^w_t}$ is injective from Corollary \ref{11022019C1} since the quadratic form $a_t$ is nonnegative. In order to check that this operator is also nonnegative and selfadjoint on $L^2(\mathbb R^n)$, we recall that the adjoint of any evolution operator $e^{-s\tilde q^w}$ generated by an accretive operator $\tilde q^w(x,D_x)$ is given by $(e^{-s\tilde q^w})^* = e^{-s(\overline{\tilde q})^w}$, see e.g.\ \cite[Chapter 1, Corollary 10.6]{MR710486} and \cite[p. 426]{MR1339714}. This formula implies that $(e^{-ta^w_t})^* = e^{-ta^w_t}$, since the quadratic form $a_t$ is real-valued. The operator $e^{-ta^w_t}$ is therefore selfadjoint on $L^2(\mathbb R^n)$. By using this  selfadjointness together with the semigroup property of the family of contractive operators $(e^{-sa^w_t})_{s\geq0}$, we deduce that
$$\forall u\in L^2(\mathbb R^n),\quad \langle e^{-ta^w_t}u,u\rangle_{L^2(\mathbb R^n)} = \big\Vert e^{-\frac t2a^w_t}u\big\Vert^2_{L^2(\mathbb R^n)}\geq0,$$
which proves that the operator $e^{-ta^w_t}$ is also nonnegative. Finally, the operator $e^{-itb^w_t}$ is unitary on $L^2(\mathbb R^n)$ since the quadratic form $b_t$ is real-valued. 

\subsection{A symplectic lemma}\label{symp} We now prove that any matrix of the form $e^{JQ}$, with $J$ the real symplectic matrix defined in \eqref{08112018E1} and $Q$ a complex symmetric matrix, is symplectic. Before that, let us recall that when $\mathbb K = \mathbb R$ or $\mathbb C$, the symplectic group $\Symp_{2n}(\mathbb K)$ is the subgroup of $\GL_{2n}(\mathbb K)$ composed of all matrices $M\in\GL_{2n}(\mathbb K)$ such that $M^TJM = J$, or equivalently $JM = (M^T)^{-1}J$, where $J$ is again the matrix defined in \eqref{08112018E1}.

\begin{lem}\label{09112018L1} For all $Q\in\Symm_{2n}(\mathbb C)$, we have $e^{JQ}\in\Symp_{2n}(\mathbb C)$.
\end{lem}

\begin{proof} Since the matrix $J$ satisfies $J^2=-I_{2n}$ and $J^T=-J$, and the matrix $Q$ is symmetric, we first notice that for all $t\geq0$,
$$\partial_t \big[(e^{tJQ})^TJe^{tJQ}\big] = (JQ e^{tJQ})^TJe^{tJQ} + (e^{tJQ})^TJJQe^{tJQ} = (e^{tJQ})^TQe^{tJQ} - (e^{tJQ})^TQe^{tJQ}=0.$$
Moreover, $(e^{0JQ})^TJe^{0JQ} = J$, which proves that for all $t\geq0$, $(e^{tJQ})^TJe^{tJQ} = J$. In particular, the matrix $e^{JQ}$ is symplectic. This ends the proof of Lemma \ref{09112018L1}.
\end{proof}

\subsection{About Fourier integral operators}\label{FIO} Fourier integral operators associated with nonnegative complex linear transformations play a key role in this paper to manipulate the evolution operators $e^{-tq^w}$ generated by quadratic forms $q:\mathbb R^{2n}\rightarrow\mathbb C$ with nonnegative real parts $\Reelle q\geq0$. In this subsection, we recall their definition and their basic properties following \cite[Section 5]{MR1339714} and \cite[Section 2]{MR3880300}. Let $T\in\Symp_{2n}(\mathbb C)$ be a nonnegative complex symplectic linear  transformation, that is, a complex symplectic transformation satisfying
$$\forall X\in\mathbb C^{2n},\quad i\big(\sigma(\overline{TX},TX)-\sigma(\overline X,X)\big)\geq0,$$
with $\sigma$ the canonical symplectic form on $\mathbb C^{2n}$ defined in \eqref{09052019E12}. Associated with this nonnegative symplectic linear transformation is its twisted graph $\lambda_T = \big\{(TX,X') : X\in\mathbb C^{2n}\big\}\subset\mathbb C^{2n}\times\mathbb C^{2n},$
where $X' = (x,-\xi)\in\mathbb C^{2n}$ if $X = (x,\xi)\in\mathbb C^{2n}$, which defines a nonnegative Lagrangian plane of $\mathbb C^{2n}\times\mathbb C^{2n}$ equipped with the symplectic form $\sigma_1((z_1,z_2),(\zeta_1,\zeta_2)) := \sigma(z_1,\zeta_1) + \sigma(z_2,\zeta_2),$ for $ (z_1,z_2),(\zeta_1,\zeta_2)\in\mathbb C^{2n}\times\mathbb C^{2n}.$
The set $\widetilde{\lambda_T} = \big\{(z_1,z_2,\zeta_1,\zeta_2) : (z_1,\zeta_1,z_2,\zeta_2)\in\lambda_T\big\}\subset\mathbb C^{4n},$
is then a nonnegative Lagrangian plane of $\mathbb C^{4n}$ equipped with the canonical symplectic form on $\mathbb C^{4n}$ (see \eqref{09052019E12}). According to \cite[Proposition 5.1 and Proposition 5.5]{MR1339714}, there exists a complex-valued quadratic form 
\begin{equation}\label{17052019E2}
	p(x,y,\theta) = \langle(x,y,\theta),P(x,y,\theta)\rangle,\quad (x,y)\in\mathbb R^{2n}, \theta\in\mathbb R^N,
\end{equation}
where
\begin{equation}
	P = \begin{pmatrix}
	P_{x,y;x,y} & P_{x,y;\theta} \\
	P_{\theta;x,y} & P_{\theta;\theta}
	\end{pmatrix}\in M_{2n+N}(\mathbb C),
\end{equation}
is a symmetric matrix satisfying the conditions:
\begin{enumerate}[label=\textbf{\arabic*.},leftmargin=* ,parsep=2pt,itemsep=0pt,topsep=2pt]
	\item $\Imag P\geq0$,
	\item The row vectors of the submatrix $\begin{pmatrix} P_{\theta;x,y} & P_{\theta;\theta} \end{pmatrix}\in\mathbb C^{N\times(2n+N)}$ are linearly independent over $\mathbb C$,
\end{enumerate}
parametrizing the nonnegative Lagrangian plane
$$\widetilde{\lambda_T}=\bigg\{\bigg(x,y,\frac{\partial p}{\partial x}(x,y,\theta),\frac{\partial p}{\partial y}(x,y,\theta)\bigg) : \frac{\partial p}{\partial\theta}(x,y,\theta) = 0\bigg\}.$$
By using some integrations par parts as in \cite[p. 442]{MR1339714}, this quadratic form $p$ allows to define the tempered distribution 
\begin{equation}\label{17052019E5}
	K_T = \frac1{(2\pi)^{\frac{n+N}2}}\sqrt{\det\begin{pmatrix} -ip''_{\theta,\theta} & p''_{\theta,y} \\ p''_{x,\theta} & ip''_{x,y} \end{pmatrix}}\int_{\mathbb R^N}e^{ip(x,y,\theta)}\ \mathrm d\theta\in\mathscr S'(\mathbb R^{2n}),
\end{equation}
as an oscillatory integral. Notice here that we do not prescribe the sign of the square root so the tempered distribution $K_T$ is defined up to its sign. Apart form this sign uncertainty, it is checked in \cite[p. 444]{MR1339714} that this definition only depends on the nonnegative complex symplectic transformation $T$, and not on the choice of the parametrization of the nonnegative Lagrangian $\widetilde{\lambda_T}$ by the quadratic form $p$. Associated with the nonnegative complex symplectic linear transformation $T$ is therefore the Fourier integral operator 
$$\mathscr K_T:\mathscr S(\mathbb R^n)\rightarrow\mathscr S'(\mathbb R^n),$$
defined by the kernel $K_T\in\mathscr S'(\mathbb R^{2n})$ as
$$\forall u,v\in\mathscr S(\mathbb R^n),\quad\langle \mathscr K_Tu,v\rangle_{\mathscr S'(\mathbb R^n),\mathscr S(\mathbb R^n)} 
= \langle K_T,u\otimes v\rangle_{\mathscr S'(\mathbb R^{2n}),\mathscr S(\mathbb R^{2n})}.$$
The first properties of this class of Fourier integral operators is summarized in the following proposition which is taken from \cite[Proposition 2.1]{MR3880300}:

\begin{prop}\label{17052019P1} Associated with any nonnegative complex symplectic linear transformation $T$ is a Fourier integral operator $\mathscr K_T:\mathscr S(\mathbb R^n)\rightarrow\mathscr S'(\mathbb R^n)$ whose kernel (determined up to its sign) is the tempered distribution $K_T\in\mathscr S'(\mathbb R^{2n})$ defined in \eqref{17052019E5} and whose adjoint $\mathscr K^*_T = \mathscr K_{\overline T^{-1}}:\mathscr S(\mathbb R^n)\rightarrow\mathscr S'(\mathbb R^n)$ is the Fourier integral operator associated with the nonnegative complex symplectic linear transformation $\overline T^{-1}$. The Fourier integral operator $\mathscr K_T$ defines a continuous mapping on the Schwartz space
$$\mathscr K_T:\mathscr S(\mathbb R^n)\rightarrow\mathscr S(\mathbb R^n),$$
which extends by duality as a continuous map on the space of tempered distributions
$$\mathscr K_T:\mathscr S'(\mathbb R^n)\rightarrow\mathscr S'(\mathbb R^n),$$
satisfying the Egorov formula
\begin{equation}\label{17052019E6}
	\forall X_0\in\mathbb C^{2n},\forall u\in\mathscr S'(\mathbb R^n),\quad\langle X_0,X\rangle^w\mathscr K_Tu = \mathscr K_T\langle J^{-1}T^{-1}JX_0,X\rangle^wu,
\end{equation}
where $J$ is the matrix defined in \eqref{08112018E1} and where for all $Y_0=(y_0,\eta_0)\in\mathbb C^{2n}$, 
$$\langle Y_0,X\rangle^w = \langle y_0,x\rangle + \langle\eta_0,D_x\rangle.$$
Furthermore, the Fourier integral operator
$$\mathscr K_T:L^2(\mathbb R^n)\rightarrow L^2(\mathbb R^n),$$
is a bounded operator on $L^2(\mathbb R^n)$ whose operator norm satisfies $\Vert\mathscr K_T\Vert_{\mathcal L(L^2(\mathbb R^n))}\le 1$.
\end{prop}

\noindent The Egorov formula is presented in the following way in \cite[Proposition 2.1]{MR3880300}:
\begin{equation}\label{17052019E7}
	\forall(y_0,\eta_0)\in\mathbb C^{2n}, \forall u\in\mathscr S'(\mathbb R^n),\quad (\langle x_0,D_x\rangle - \langle\xi_0,x\rangle)\mathscr K_Tu = \mathscr K_T(\langle y_0,D_x\rangle - \langle\eta_0,x\rangle)u,
\end{equation}
with $(x_0,\xi_0) = T(y_0,\eta_0)$. However, the formulas \eqref{17052019E6} and \eqref{17052019E7} are equivalent since
$$\langle x_0,D_x\rangle - \langle\xi_0,x\rangle = \langle J^{-1}X_0,X\rangle^w = \langle J^{-1}TY_0,X\rangle^w\quad\text{and}\quad \langle y_0,D_x\rangle - \langle\eta_0,x\rangle = \langle J^{-1}Y_0,X\rangle^w.$$

The next proposition, coming from \cite[Proposition 5.9]{MR1339714}, shows that the composition of two Fourier integral operators associated with nonnegative complex symplectic linear transformations remains a Fourier integral operator associated with a nonnegative complex symplectic linear transformation. It has a key role in this paper in Section \ref{Polar}. The sign uncertainty that appears is anew due to the fact that the Schwartz distributions $K_T$ defined in \eqref{17052019E5} are determined up to their sign. However, this sign uncertainty is not an issue in this work.

\begin{prop}\label{17052019P3} If $T_1$ and $T_2$ are two nonnegative complex symplectic linear transformations in $\mathbb C^{2n}$, then $T_1T_2$ is also a nonnegative complex symplectic linear transformation and
$$\mathscr K_{T_1T_2} = \pm\mathscr K_{T_1}\mathscr K_{T_2}.$$
\end{prop}

Finally, we are interested in the real case:

\begin{dfn}\label{17052019D1} A Fourier integral operator $\mathscr K_T$ associated with a real symplectic linear transformation $T$ is called metaplectic.
\end{dfn}

The metaplectic operators stand out among the other Fourier integral operators $\mathscr K_T$ as illustrated in the following proposition which comes from \cite[Theorem 5.12]{MR1339714}:

\begin{prop}\label{17052019P2} Let $\mathscr K_T$ be a Fourier integral operator associated with a nonnegative complex symplectic transformation $T$. The operator $\mathscr K_T:L^2(\mathbb R^n)\rightarrow L^2(\mathbb R^n)$ is invertible if and only if $\mathscr K_T$ is a metaplectic operator, that is,  if and only if $T$ is a real symplectic transformation. In this case, the operator $\mathscr K_T:L^2(\mathbb R^n)\rightarrow L^2(\mathbb R^n)$ defines a bijective isometry on $L^2(\mathbb R^n)$.
\end{prop}

To finish, let us recall the metaplectic invariance of the Weyl calculus:

\begin{thm} Let  $T$ be a real symplectic transformation and $\mathscr K_T$ the associated metaplectic operator. Then, the following identity holds for all tempered distributions $a\in\mathscr S'(\mathbb R^n)$,
$$\mathscr K^{-1}_Ta^w(x,D_x)\mathscr K_T = (a\circ T)^w(x,D_x).$$
\end{thm}

The general result of metaplectic invariance of the Weyl calculus can be found e.g.\ in \cite[Theorem 18.5.9]{MR2304165}. Notice that the Egorov formula \eqref{17052019E6} is a particular case of this Theorem for linear forms since \eqref{17052019E6} can be also written in the following way
$$\forall X_0\in\mathbb C^{2n},\quad \mathscr K^{-1}_T\langle X_0,X\rangle^w\mathscr K_T = \langle X_0,TX\rangle^w,$$
by using that $(J^{-1}T^{-1}J)^T = T$, which is a straightforward property of real symplectic matrices.

\subsection{Splitting of the harmonic oscillator semigroup} In this subsection, we give a decomposition of the harmonic oscillator semigroup. To obtain this splitting, we will make use once again of the theory of Fourier integral operators in the very same way as in Section \ref{Polar}. Let us mention as an anecdote that the identity \eqref{11022019E8} involved in the proof of the following proposition has played a major role and has been widely used in image processing in order to make rotations, see e.g.\ \cite P. This identity is also key here for our purpose. As a byproduct of this splitting, we obtain the injectivity property of the evolution operators generated by accretive quadratic operators associated with nonnegative quadratic forms.

\begin{prop}\label{11022019P1} Let $\mathcal H = -\partial^2_x+x^2$, with $x\in\mathbb R$, be the harmonic oscillator. Then, the semigroup $(e^{-t\mathcal H})_{t\geq0}$ generated by the operator $\mathcal H$ admits the following decomposition:
\begin{equation}\label{11022019E16}
	\forall t\geq0,\quad e^{-t\mathcal H} = e^{-\frac12(\tanh t)x^2}e^{\frac12\sinh(2t)\partial^2_x}e^{-\frac12(\tanh t)x^2}.
\end{equation}
This implies in particular that the evolution operators $e^{-t\mathcal H}$ are injective.
\end{prop}

\begin{proof} First, let us consider the following well know factorization (see e.g.\ \cite{P}) for $t\in(-\pi,\pi)$ 
\begin{equation}\label{11022019E8}
	\begin{pmatrix}
		1 & 0 \\
		\tan\frac t2 & 1
	\end{pmatrix}
	\begin{pmatrix}
		1 & -\sin t \\
		0 & 1
	\end{pmatrix}
	\begin{pmatrix}
		1 & 0 \\
		\tan\frac t2 & 1
	\end{pmatrix}
	=
	\begin{pmatrix}
		\cos t & -\sin t \\
		\sin t & \cos t
	\end{pmatrix}.
\end{equation}
Identifying each one of these matrices with an exponential and extending analytically this relation to $t\in i\mathbb R$ we deduce that
$$\forall t\in\mathbb R,\quad e^{-itJ} = e^{-(i\tanh\frac t2)JQ_x}e^{-(i\sinh t)JQ_\xi} e^{-(i\tanh\frac t2)JQ_x},$$
where $J$ is the symplectic matrix in dimension $2$, $Q_x$ is the matrix of $x^2$ and $Q_{\xi}$ is the matrix of $\xi^2$. Consequently, applying  Proposition \ref{13112018P2}, we deduce that for all $t\geq0$,
\begin{equation}\label{11022019E14}
	\varepsilon_te^{-\frac12(\tanh t)x^2}e^{\frac12\sinh(2t)\partial^2_x}e^{-\frac12(\tanh t)x^2} = e^{-t(x^2-\partial^2_x)},
\end{equation}
with $\varepsilon_t\in\{-1,1\}$ for all $t\geq0$. It only remains to prove that $\varepsilon_t = 1$ for all $t\geq0$ to establish \eqref{11022019E16}. To that end, we consider $u_0\in\mathscr S(\mathbb R)$ the Gaussian function defined for all $x\in\mathbb R$ by $u_0(x) = e^{-x^2}$. We first notice that for all $t\geq0$,
\begin{equation}\label{11022019E13}
	e^{-\frac12(\tanh t)x^2}e^{\frac12\sinh(2t)\partial^2_x}e^{-\frac12(\tanh t)x^2}u_0>0.
\end{equation}
Indeed, this estimate is trivial when $t=0$ by definition of $u_0$. When $t>0$, we observe that for all $u\in\mathscr S(\mathbb R)$ such that $u>0$, the function $e^{-\frac12(\tanh t)x^2}u>0$ is also positive, and on the other hand, we notice by using the explicit formula for the Fourier transform of Gaussian functions that
$$e^{\frac12\sinh(2t)\partial^2_x}u = \sqrt{\frac{2\pi}{\sinh(2t)}}\exp\bigg(-\frac{x^2}{2\sinh(2t)}\bigg)\ast u>0,$$
where $\ast$ denotes the convolution product. This proves that \eqref{11022019E13} holds. Now, let us consider the function $\varphi$ defined for all $t\geq0$ by
\begin{equation}\label{10052019E1}
	\varphi(t) = \varepsilon_te^{-\frac12(\tanh t)x^2}e^{\frac12\sinh(2t)\partial^2_x}e^{-\frac12(\tanh t)x^2}u_0\in\mathscr S(\mathbb R^n).
\end{equation}
The rest of the proof consists in checking that $\varphi (t)>0$ for all $t\geq0$. This property combined with \eqref{11022019E13} will prove that $\varepsilon_t>0$ for all $t\geq0$. Since $\varepsilon_t\in\{-1,1\}$, it will then follows that $\varepsilon_t = 1$ for all $t\geq0$. We first deduce from \cite{MR1339714} (Theorem 4.2) that the function $t\geq0\mapsto e^{-t(x^2-\partial^2_x)}u_0\in\mathscr S(\mathbb R^n)$ is continuous which implies from \eqref{11022019E14} and \eqref{10052019E1} the continuity of the function $\varphi$ from $[0,+\infty)$ to $\mathscr S(\mathbb R)$. As a consequence of \eqref{11022019E13} and \eqref{10052019E1}, the Schwartz function $\varphi(t)$ is not the zero function for all $t\geq0$. Let $x\in\mathbb R$. The previous discussion implies that the function $t\geq0\mapsto\varphi(t)(x)\in\mathbb R^*$ is continuous and does not vanish. Moreover, it follows from \eqref{11022019E14} and \eqref{10052019E1} that $\varphi(0)(x) = u_0(x)>0$. We deduce that $\varphi(t)(x)>0$ for all $t\geq0$. As a consequence, $\varphi(t)>0$ for all $t\geq0$. This proves that \eqref{11022019E16} holds. The injectivity of the operators $e^{-t\mathcal H}$ is then a straightforward consequence of \eqref{11022019E16} since the operators $e^{-\frac12(\tanh t)x^2}$ and $e^{\frac12\sinh(2t)\partial^2_x}$ are themselves injective. This ends the proof of Proposition \ref{11022019P1}.
\end{proof}

Notice that the injectivity property of the evolution operators $e^{-t\mathcal H}$ can also be readily proved by using the Hermite basis of $L^2(\mathbb R^n)$ and a direct calculus.

\begin{cor}\label{11022019C1} Let $q:\mathbb R^{2n}\rightarrow\mathbb R$ be a nonnegative quadratic form $q\geq0$. Then, for all $t\geq0$, the evolution operator $e^{-tq^w}$ generated by the accretive quadratic operator $q^w(x,D_x)$ is injective.
\end{cor}

\begin{proof} We deduce from \cite[Theorem 21.5.3]{MR2304165} that there exists a real linear symplectic transformation $\chi:\mathbb R^{2n}\rightarrow\mathbb R^{2n}$ such that for all $(x,\xi)\in\mathbb R^{2n}$,
\begin{equation}
	(q\circ\chi)(x,\xi) = \sum_{j=1}^k\lambda_j(\xi^2_j + x^2_j) + \sum_{j=k+1}^{k+l}x^2_j,
\end{equation}
with $k,l\geq0$ and $\lambda_j>0$ for all $1\le j\le k$. By the symplectic invariance of the Weyl quantization, \cite{MR2304165} (Theorem 28.5.9), we can find a metaplectic operator $\mathcal T$ satisfying
\begin{equation}\label{11022019E15}
	q^w(x,D_x) = \mathcal T^{-1}\Big(\sum_{j=1}^k\lambda_j(D^2_{x_j} + x^2_j) + \sum_{j=k+1}^{k+l}x^2_j\Big)\mathcal T.
\end{equation}
Let $t\geq0$. It follows from \eqref{11022019E15} that the evolution operator $e^{-tq^w}$ is written as
\begin{equation}\label{11022019E17}
	e^{-tq^w} = \mathcal T^{-1}\Big(\prod_{j=1}^ke^{-t\lambda_j(D^2_{x_j} + x^2_j)}\Big)\Big(\sum_{j=k+1}^{k+l}e^{-tx^2_j}\Big)\mathcal T.
\end{equation}
We deduce from \eqref{11022019E17} and Proposition \ref{11022019P1} that the operator $e^{-tq^w}$ is the composition of injective operators, so is itself injective. This ends the proof of Corollary \ref{11022019C1}.
\end{proof}

\subsection{Spectrum localization} The following result provides a localization for the spectrum of matrices of the form $JA$, with $J$ the symplectic matrix defined in \eqref{08112018E1} and $A$ a Hermitian positive semidefinite matrix.

\begin{lem}\label{07012019L1} Let $A\in\Her_n(\mathbb C)$ be a Hermitian positive semidefinite matrix and $J\in\Symp_{2n}(\mathbb R)$ be the symplectic matrix given by \eqref{08112018E1}. Then, the spectrum of the matrix $JA$ is purely imaginary, that is $\sigma(JA)\subset i\mathbb R$. 
\end{lem}

\begin{proof} We first assume that the matrix $A$ is Hermitian positive definite. Under this assumption, we observe that
$\sqrt A(JA)(\sqrt A)^{-1} = \sqrt AJ\sqrt A$. The matrix $JA$ is therefore conjugated to a skew-Hermitian matrix and its spectrum is then purely imaginary. When $A$ is only Hermitian positive semidefinite, we can consider $(A_p)_p$ a sequence of Hermitian positive definite matrices that converges to $A$. Since the eigenvalues of a complex matrix are continuous with respect to this matrix according to \cite[Theorem II.5.1]{MR0203473}, and that $\sigma(JA_p)\subset i\mathbb R$ from the beginning of the proof, we deduce that the eigenvalues of the matrix $JA$ are purely imaginary. This ends the proof of Lemma~\ref{07012019L1}.
\end{proof}

\subsection{Taylor expansion in a non-commutative setting} In the next lemma, we prove a composition result of Taylor expansions for functions taking values in non-commutative rings. It will be useful in the end of Subsection \ref{perturbation}. Notice that we consider holomorphic functions in a neighborhood of $0$, but the proof works the same way near any point of $\mathbb C$.  Let us recall that $\mathbb C\langle X,Y \rangle$ denotes the ring of non-commutative polynomials in $X$ and $Y$, and that for all nonnegative integer $k\geq0$, we consider $\mathbb C_{k,0}\langle X,Y \rangle$ the finite-dimensional subspace of $\mathbb C\langle X,Y\rangle$ of non-commutative polynomials of degree smaller than or equal to $k$ vanishing in $(0,0)$. In the following, given $\rho>0$, the notation $\mathbb D(0,\rho)$ denotes the open disk in $\mathbb C$ centered in $0$ of radius $\rho$, and $B((0,0),\rho)$ stands for the open ball in $M_{2n}(\mathbb R)\times M_{2n}(\mathbb R)$ centered in $(0,0)$ of radius $\rho$ with respect to the norm $\Vert\cdot\Vert_{\infty}$ defined in \eqref{29032021E1}.

\begin{lem}\label{04092019L1} Let $f:\mathbb D(0,\rho)\rightarrow\mathbb C$ be an analytic function, with $\rho>0$. We consider $P\in\mathbb C_{k,0}\langle X,Y \rangle$, with $k\geq0$ a nonnegative integer, and $R:B((0,0),\rho)\rightarrow M_{2n}(\mathbb C)$ a function satisfying that there exists a positive constant $C>0$ such that for all $(M,N)\in B((0,0),\rho)$ we have
\begin{equation}\label{04092019E1}
	\Vert R(M,N)\Vert\le C\Vert(M,N)\Vert^{k+1}_{\infty}.
\end{equation}
Then, there exists $\rho'\in(0,\rho)$, depending continuously on $P$ and $C$, such that the function 
$$f\circ(P+R):(M,N)\mapsto\sum_{j=0}^{+\infty}\frac{f^j(0)}{j!}(P(M,N)+R(M,N))^j,$$
is well defined on $B((0,0),\rho')$. Furthermore, there exists a continuous map $\Psi : \mathbb{C}_{k,0}\langle X,Y \rangle \to \mathbb C_{k,0}\langle X,Y \rangle$ and a function $R':B((0,0),\rho')\rightarrow M_{2n}(\mathbb C)$ such that for all $(M,N)\in B((0,0),\rho')$,
$$f(P(M,N)+R(M,N)) = f(0)I_{2n}+\Psi(P)(M,N)+R'(M,N),$$
with
$$\Vert R'(M,N)\Vert\le\Gamma_{C,P}\Vert(M,N)\Vert^{k+1}_{\infty},$$
$\Gamma_{C,P}>0$ denoting a positive constant which depends continuously on $C$ and $P$.
\end{lem}

\begin{proof} Since the functions $P$ and $R$ tend to $(0,0)$ as $(M,N)$ goes to $(0,0)$, if $\rho'\in(0,\rho)$ is chosen sufficiently small, then for $(M,N)\in B((0,0),\rho')$, we have $\Vert P(M,N)\Vert<\rho/4$ and $\Vert R(M,N)\Vert<\rho/4$. Consequently, the function $f\circ(P+R)$ is well defined on $B((0,0),\rho')$. Let $(M,N)\in B((0,0),\rho')$. Realizing a Taylor expansion of the function $f$ (considered as a map on $M_{2n}(\mathbb C)$) in $P(M,N)$, we get that
$$f\circ(P+R)(M,N) = f\circ P(M,N) + \int_0^1\mathrm df(P(M,N)+\alpha R(M,N))(R(M,N))\ \mathrm d\alpha,$$
where $\mathrm df$ denotes the differential of the function $f$. The second term in the right-hand side of the above equality is a remainder term. Indeed, since  $\Vert P(M,N)+\alpha R(M,N)\Vert<\rho/2$ for all $0\le\alpha\le1$ with our choice of $\rho'\in(0,\rho)$, we deduce from \eqref{04092019E1} that this term satisfies
$$\bigg\vert\int_0^1\mathrm df(P(M,N)+\alpha R(M,N))(R(M,N))\ \mathrm d\alpha\bigg\vert\le C\bigg(\sup_{\Vert L\Vert<\rho/2}\Vert\mathrm df(L)\Vert_{\mathcal L(M_{2n}(\mathbb C))}\bigg)\Vert(M,N)\Vert^{k+1}_{\infty},$$
with $\mathcal L(M_{2n}(\mathbb C))$ the space of bounded operators on $M_{2n}(\mathbb C)$. Consequently, we focus on the term $f\circ P(M,N)$. Since the function $f$ is analytic on $\mathbb D(0,\rho)$, we can consider $(a_j)_{j\geq0}\in\mathbb C^\mathbb N$ the coefficients of the Taylor expansion of $f$ and write
$$\forall z\in\mathbb D(0,\rho),\quad f(z) = \sum_{j=0}^{+\infty}a_jz^j.$$
Naturally, $f\circ P(M,N)$ can be decomposed as
$$f\circ P(M,N) = f(0)I_{2n} + Q(P(M,N)) + P(M,N)^{k+1}\sum_{j=0}^{+\infty}a_{j+k+1}P(M,N)^j,$$
where $Q\in\mathbb C_k[X]$ is a polynomial of degree smaller than or equal to $k$ vanishing in $0$ and depending only on $f$, given by $Q(X) = \sum_{j=1}^ka_jX^j$. The third term in the right-hand side of the above equality is also a remainder term. Indeed, since the polynomial $P$ vanishes in $(0,0)$, there exists a positive constant $M_P>0$ depending continuously (and only) on $P$ such that 
$$\Vert P(M,N)\Vert\le M_P\Vert(M,N)\Vert_{\infty}.$$
With the previous choice of $\rho'\in(0,\rho)$, $\Vert P(M,N)\Vert<\rho/4$, and we obtain that
$$\Big\Vert P(M,N)^{k+1}\sum_{j=0}^{+\infty}a_{j+k+1}P(M,N)^j\Big\Vert\le M_P^{k+1}\Vert(M,N)\Vert^{k+1}_{\infty}\sum_{j=0}^{+\infty}\vert a_{j+k+1}\vert\Big(\frac{\rho}4\Big)^j.$$
Notice that the sum in the right-hand side is finite since the function $f$ is analytic on $\mathbb D(0,\rho)$. Finally, we just have to observe that $Q\circ P\in\mathbb C_{k^2,0}\langle X,Y\rangle$ is a non-commutative polynomial vanishing in $(0,0)$ and depending continuously on $P$. The sum of its terms of degree smaller than or equal to $k$ defines $\Psi(P)$ and its higher order terms are remainder terms bounded by $\Vert(M,N)\Vert^{k+1}_{\infty}$, up to a constant also depending continuously on $P$. This ends the proof of Lemma~\ref{04092019L1}.
\end{proof}

\subsection{A perturbation result}\label{perturbation} To end this appendix, we give the proof of a quite technical lemma which is instrumental in Section \ref{Polar} and Section \ref{short}. Let $q:\mathbb R^{2n}\rightarrow\mathbb C$ be a complex-valued quadratic form with a nonnegative real-part $\Reelle q\geq0$. We consider $Q\in\Symm_{2n}(\mathbb C)$ the matrix of $q$ in the canonical basis of $\mathbb R^{2n}$, $F$ the Hamilton map of $q$ and $S$ its singular space. Let $0\le k_0\le 2n-1$ be the smallest integer such that \eqref{12102018E7} holds. Moreover, we consider the time-dependent quadratic form $\kappa_t:\mathbb C^{2n}\rightarrow\mathbb R$ defined in accordance with the convention \eqref{07052019E1} for all $t\geq0$ and $X\in\mathbb C^{2n}$ by
\begin{equation}\label{13052019E9}
	\kappa_t(X) = \sum_{k=0}^{k_0} t^{2k} \Reelle q\big((\Imag F)^k X\big) = \sum_{k=0}^{k_0}t^{2k}\big\vert\sqrt{\Reelle Q}(\Imag F)^kX\big\vert^2.
\end{equation}
The following lemma investigates the perturbations of the quadratic form $\kappa_t$:

\begin{lem}\label{05112018L2} Let $(G_{\alpha})_{0\le\alpha\le1}$ be a family of functions $G_{\alpha}:B((0,0),\rho)\rightarrow M_{2n}(\mathbb C)$, with $\rho>0$, satisfying on the one hand that there exist a family $(P_{\alpha})_{0\le\alpha\le1}$ of non-commutative polynomials $P_{\alpha}\in\mathbb C_{k_0,0}\langle X,Y\rangle$ depending continuously on the parameter $0\le\alpha\le1$, a family $(R_{\alpha})_{0\le\alpha\le1}$ of functions $R_{\alpha}:B((0,0),\rho)\rightarrow M_{2n}(\mathbb C)$ and a positive constant $C>0$ such that for all $0\le\alpha\le1$ and $(M,N)\in B((0,0),\rho)$,
\begin{equation}\label{04092019E9}
	G_{\alpha}(M,N) = I_{2n} + P_{\alpha}(M,N) + R_{\alpha}(M,N),
\end{equation}
with
\begin{equation}\label{04092019E10}
	\Vert R_{\alpha}(M,N)\Vert\le C\Vert(M,N)\Vert^{k_0+1}_{\infty},
\end{equation}
and on the other hand that for all $0\le\alpha\le1$ and $t\geq0$ such that $(t\Reelle F,t\Imag F)\in B((0,0),\rho)$,
\begin{equation}\label{03092019E6}
	G_{\alpha}(t\Reelle F,t\Imag F)(S+iS)\subset S+iS.
\end{equation}
Then, there exist some positive constants $c>0$ and $0<T\le 1$ such that for all $0\le t\le T$, $0\le\alpha\le1$ and $X\in\mathbb C^{2n}$,
$$\kappa_t\big(G_{\alpha}\big(t\Reelle F,t\Imag F\big)X\big)\geq c\kappa_t(X).$$
\end{lem}

\begin{proof} By definition \eqref{13052019E9} of the time-dependent quadratic form $\kappa_t$, the estimate we want to prove writes for all $0\le\alpha\le1$, $0\le t\ll1$ small enough and $X\in\mathbb C^{2n}$ as
\begin{equation}\label{13052019E13}
	\sum_{k=0}^{k_0}t^{2k}\big\vert\sqrt{\Reelle Q}(\Imag F)^kG_{\alpha}(t\Reelle F,t\Imag F)X\big\vert^2\geq c\sum_{k=0}^{k_0}t^{2k}\big\vert\sqrt{\Reelle Q}(\Imag F)^kX\big\vert^2.
\end{equation}
By using the two classical inequalities that hold for all $m\geq1$ and $a_1,\ldots,a_m\geq0$,
\begin{equation}\label{13052019E11}
	\sqrt{a_1+\ldots+a_m}\le\sqrt{a_1}+\ldots+\sqrt{a_m},
\end{equation}
and
\begin{equation}\label{13052019E12}
	(a_1+\ldots+a_m)^2\le 2^{m-1}(a^2_1+\ldots+a^2_m),
\end{equation}
we notice that in order to prove the estimate \eqref{13052019E13}, it is in fact sufficient to establish that for all $0\le\alpha\le1$, $0\le t\ll1$ small enough and $X\in\mathbb C^{2n}$,
\begin{equation}\label{13052019E10}
	\sum_{k=0}^{k_0}t^k\big\vert\sqrt{\Reelle Q}(\Imag F)^kG_{\alpha}(t\Reelle F,t\Imag F)X\big\vert\geq c\sum_{k=0}^{k_0}t^k\big\vert\sqrt{\Reelle Q}(\Imag F)^kX\big\vert.
\end{equation}
Indeed, we deduce from \eqref{13052019E11} and \eqref{13052019E12} that when \eqref{13052019E10} holds, we have that for all $0\le\alpha\le1$, $0\le t\ll1$ small enough and $X\in\mathbb C^{2n}$,
\begin{align}\label{13052019E20}
	&\ \sum_{k=0}^{k_0}t^{2k}\big\vert\sqrt{\Reelle Q}(\Imag F)^kG_{\alpha}(t\Reelle F,t\Imag F)X\big\vert^2 \\
	\geq &\ \frac1{2^{k_0}}\bigg(\sum_{k=0}^{k_0}t^k\big\vert\sqrt{\Reelle Q}(\Imag F)^kG_{\alpha}(t\Reelle F,t\Imag F)X\big\vert\bigg)^2
	\geq\frac{c^2}{2^{k_0}}\bigg(\sum_{k=0}^{k_0}t^k\big\vert\sqrt{\Reelle Q}(\Imag F)^kX\big\vert\bigg)^2 \nonumber \\
	= &\ \frac{c^2}{2^{k_0}}\bigg(\sum_{k=0}^{k_0}\sqrt{t^{2k}\big\vert\sqrt{\Reelle Q}(\Imag F)^kX\big\vert^2}\bigg)^2
	\geq\frac{c^2}{2^{k_0}}\sum_{k=0}^{k_0}t^{2k}\big\vert\sqrt{\Reelle Q}(\Imag F)^kX\big\vert^2, \nonumber
\end{align}
which is the required estimate. We therefore focus on proving the estimate \eqref{13052019E10}. First of all, let us write the functions $G_{\alpha}$ under a more manageable form. Since the non-commutative polynomials $P_{\alpha}\in\mathbb C_{k_0}\langle X,Y\rangle$ have a degree smaller than or equal to $k_0$, vanish on $(0,0)$ and depend continuously on the parameter $0\le\alpha\le1$, there exist some continuous functions $\sigma_{j,m}:[0,1]\rightarrow\mathbb C$, with $1\le j\le k_0$ and $m\in\{0,1\}^j$, such that for all $0\le\alpha\le1$,
\begin{equation}\label{04092019E8}
	P_{\alpha}(X,Y) = \sum_{j=1}^{k_0}\sum_{m\in\{0,1\}^j}\sigma_{j,m}(\alpha)X^{m_1}Y^{1-m_1}\ldots X^{m_j}Y^{1-m_j}.
\end{equation}
With an abuse of notation, we denote the above non-commutative product by
$$X^{1-m_1}Y^{1-m_1}\ldots X^{m_j}Y^{1-m_j} = \prod_{\ell=1}^jX^{m_{\ell}}Y^{1-m_{\ell}}.$$
We deduce from \eqref{04092019E9} and \eqref{04092019E8} that for all $0\le\alpha\le1$, $0\le k\le k_0$ and $(M,N)\in B((0,0),\rho)$,
\begin{equation}\label{03092019E7}
	G_{\alpha}(M,N) = I_{2n} + \sum_{j=1}^k\sum_{m\in\{0,1\}^j}\sigma_{j,m}(\alpha)\prod_{\ell=1}^jM^{m_{\ell}}N^{1-m_{\ell}} + R_{\alpha,k}(M,N),
\end{equation}
where the remainder terms $R_{\alpha,k}(M,N)$ are given by
\begin{equation}\label{04092019E11}
	R_{\alpha,k}(M,N) = \sum_{j=k+1}^{k_0}\sum_{m\in\{0,1\}^j}\sigma_{j,m}(\alpha)\prod_{\ell=1}^jM^{m_{\ell}}N^{1-m_{\ell}} + R_{\alpha}(M,N).
\end{equation}
Since the functions $\sigma_{j,m}$ are continuous on $[0,1]$, we deduce from \eqref{04092019E10} and \eqref{04092019E11} that there exists a positive constant $C_0>0$ such that for all $0\le\alpha\le1$, $0\le k\le k_0$ and $(M,N)\in B((0,0),\rho)$,
\begin{equation}\label{03092019E8}
	\Vert R_{\alpha,k}(M,N)\Vert\le C_0\Vert(M,N)\Vert^{k+1}_{\infty}.
\end{equation}
We can now tackle the proof of the estimate \eqref{13052019E10}. We begin by studying the matrices $t^k(\Imag F)^kG_{\alpha}(t\Reelle F,t\Imag F)$. Let $T_0>0$ be such that $(t\Reelle F,t\Imag F)\in B((0,0),\rho)$ for all $0\le t\le T_0$. It follows from \eqref{03092019E7} that for all $0\le\alpha\le1$, $0\le k\le k_0$ and $0\le t\le T_0$, 
$$G_{\alpha}\big(t\Reelle F,t\Imag F\big) = I_{2n} + \sum_{j=1}^{k_0-k}\sum_{m\in\{0,1\}^j}\sigma_{j,m}(\alpha)t^j\prod_{\ell = 1}^j(\Reelle F)^{m_{\ell}}(\Imag F)^{1-m_{\ell}} + R_{\alpha,k_0-k}(t\Reelle F,t\Imag F).$$
We deduce that for all $0\le\alpha\le1$, $0\le k\le k_0$ and $0\le t\le T_0$,
\begin{multline}\label{14112018E3}
	t^k(\Imag F)^kG_{\alpha}\big(t\Reelle F,t\Imag F\big)
	= t^k(\Imag F)^k \\
	+ \sum_{j=1}^{k_0-k}\sum_{m\in\{0,1\}^j}\sigma_{j,m}(\alpha)t^{k+j}(\Imag F)^k\prod_{\ell = 1}^j(\Reelle F)^{m_{\ell}}(\Imag F)^{1-m_{\ell}}
	+ t^k(\Imag F)^kR_{\alpha,k_0-k}(t\Reelle F,t\Imag F).
\end{multline}
Let $0\le\alpha\le1$, $0\le k\le k_0$ and $1\le j\le k_0-k$. Isolating the term associated with the tuple $0\in\{0,1\}^j$ whose coordinates are all equal to $0$, we split the following sum in two
\begin{multline}\label{13052019E4}
	\sum_{m\in\{0,1\}^j}\sigma_{j,m}(\alpha)t^{k+j}(\Imag F)^k\prod_{\ell = 1}^j(\Reelle F)^{m_{\ell}}(\Imag F)^{1-m_{\ell}} \\
	= \sigma_{j,0}(\alpha)t^{k+j}(\Imag F)^{k+j}
	+ \sum_{m\in\{0,1\}^j\setminus\{0\}}\sigma_{j,m}(\alpha)t^{k+j}(\Imag F)^k\prod_{\ell = 1}^j(\Reelle F)^{m_{\ell}}(\Imag F)^{1-m_{\ell}}.
\end{multline}
For all $m\in\{0,1\}^j\setminus\{0\}$, we can write
\begin{equation}\label{13052019E5}
	\prod_{\ell = 1}^j(\Reelle F)^{m_{\ell}}(\Imag F)^{1-m_{\ell}} = A_m(\Reelle F)(\Imag F)^{n_m},
\end{equation}
where $n_m$ is a nonnegative integer satisfying $0\le n_m\le j-1$ and $A_m\in M_{2n}(\mathbb R)$ is a real matrix product of $j-1-n_m$ matrices belonging to $\{\Reelle F,\Imag F\}$. It follows from \eqref{13052019E4} and \eqref{13052019E5} that for all $0\le\alpha\le1$, $0\le k\le k_0$ and $0\le t\le T_0$,
\begin{multline}\label{13052019E6}
	\sum_{j=1}^{k_0-k}\sum_{m\in\{0,1\}^j}\sigma_{j,m}(\alpha)t^{k+j}(\Imag F)^k\prod_{\ell = 1}^j(\Reelle F)^{m_{\ell}}(\Imag F)^{1-m_{\ell}} \\
	= \sum_{j=1}^{k_0-k}\sigma_{j,0}(\alpha)t^{k+j}(\Imag F)^{k+j}
	+ \sum_{j=1}^{k_0-k}\sum_{m\in\{0,1\}^j\setminus\{0\}}\sigma_{j,m}(\alpha)t^{k+j}(\Imag F)^kA_m(\Reelle F)(\Imag F)^{n_m}.
\end{multline}
Moreover, the second term in the right-hand side of the above equality can be written as
\begin{align}\label{13052019E7}
	&\ \sum_{j=1}^{k_0-k}\sum_{m\in\{0,1\}^j\setminus\{0\}}\sigma_{j,m}(\alpha)t^{k+j}(\Imag F)^kA_m(\Reelle F)(\Imag F)^{n_m} \\
	= &\ \sum_{j=1}^{k_0-k}\sum_{p=0}^{j-1}\sum_{\substack{m\in\{0,1\}^j\setminus\{0\}\\[1pt] n_m = p}}\sigma_{j,m}(\alpha)t^{k+j}(\Imag F)^kA_m(\Reelle F)(\Imag F)^p \nonumber \\
	= &\ \sum_{p=0}^{k_0-k-1}t^{p+1}\bigg(\sum_{j=p+1}^{k_0-k}\sum_{\substack{m\in\{0,1\}^j\setminus\{0\}\\[1pt]n_m = p}}\sigma_{j,m}(\alpha)t^{k+j-p-1}(\Imag F)^kA_m\bigg)(\Reelle F)(\Imag F)^p \nonumber \\
	=: &\ \sum_{p=0}^{k_0-k-1}t^{p+1}B_{\alpha,p,k}(t)(\Reelle F)(\Imag F)^p. \nonumber
\end{align}
We deduce from \eqref{14112018E3}, \eqref{13052019E6} and \eqref{13052019E7} that for all $0\le\alpha\le1$, $0\le k\le k_0$ and $0\le t\le T_0$,
\begin{multline}\label{13052019E8}
	t^k(\Imag F)^kG_{\alpha}\big(t\Reelle F,t\Imag F\big) = t^k(\Imag F)^k + \sum_{j=1}^{k_0-k}\sigma_{j,0}(\alpha)t^{k+j}(\Imag F)^{k+j} \\
	+ \sum_{p=0}^{k_0-k-1}t^{p+1}B_{\alpha,p,k}(t)(\Reelle F)(\Imag F)^p + t^k(\Imag F)^kR_{\alpha,k_0-k}(t\Reelle F,t\Imag F).
\end{multline}
The triangle inequality therefore implies that for all $0\le\alpha\le1$, $0\le k\le k_0$, $0\le t\le T_0$ and $X\in\mathbb C^{2n}$,
\begin{equation}\label{13052019E15}
\begin{split}
	&t^k\big\vert\sqrt{\Reelle Q}(\Imag F)^kG_{\alpha}\big(t\Reelle F,t\Imag F\big)X\big\vert  \\
	\geq&\bigg\vert t^k\sqrt{\Reelle Q}(\Imag F)^kX+ \sum_{j=1}^{k_0-k}t^{k+j}\sigma_{j,0}(\alpha)\sqrt{\Reelle Q}(\Imag F)^{k+j}X\bigg\vert \\
	&- \bigg\vert\sum_{p=0}^{k_0-k-1}t^{p+1}\sqrt{\Reelle Q}B_{\alpha,p,k}(t)(\Reelle F)(\Imag F)^pX\bigg\vert 
	- t^k\big\vert\sqrt{\Reelle Q}(\Imag F)^kR_{\alpha,k_0-k}(t\Reelle F,t\Imag F)X\big\vert.
	\end{split}
\end{equation}
Our aim is now to control the two first terms appearing in the right-hand side of the above estimate. To that end, we begin by noticing that since $(\sigma_{j,m})_{1\le j\le k_0,m\in\{0,1\}^j}$ is a finite family of continuous functions defined on $[0,1]$, and by definition of the terms $B_{\alpha,p,k}(t)$ in \eqref{13052019E7}, there exists a positive constant $c_0>0$ such that for all $0\le\alpha\le1$, $0\le k\le k_0$, $1\le j\le k-k_0$ and $m\in\{0,1\}^j$,
\begin{equation}\label{13052019E14}
	\big\vert\sigma_{j,m}(\alpha)\big\vert + \big\Vert\sqrt{\Reelle Q}B_{\alpha,p,k}(t)J\sqrt{\Reelle Q}\big\Vert\le c_0.
\end{equation}
Then, the first term can be controlled in the following way: from \eqref{13052019E14} and Lemma \ref{lem_l1_ineq}, we have that for all $0\le k\le k_0-1$ and $\eta_k\in(\mathbb R^*_+)^{k_0-k}$, there exists a positive constant $\gamma_{\eta_k}>0$, such that for all $0\le\alpha\le1$, $0\le t\le T_0$ and $X\in\mathbb C^{2n}$,
\begin{multline}\label{19112018E8}
	\bigg\vert t^k\sqrt{\Reelle Q}(\Imag F)^kX + \sum_{j=1}^{k_0-k}t^{k+j}\sigma_{j,0}(\alpha)\sqrt{\Reelle Q}(\Imag F)^{k+j}X\bigg\vert \\
	\geq\gamma_{\eta_k}t^k\big\vert\sqrt{\Reelle Q}(\Imag F)^kX\big\vert - c_0\sum_{j=1}^{k_0-k}(\eta_k)_jt^{k+j}\big\vert\sqrt{\Reelle Q}(\Imag F)^{k+j}X\big\vert.
\end{multline}
Notice that when $k=k_0$, the sum appearing in the left-hand side of the estimate \eqref{19112018E8} is reduced to zero, which motivates to set $\gamma_{\eta_{k_0}} = 1$. By using that $F = JQ$ and \eqref{13052019E14}, we derive the following estimate for the second term for all $0\le\alpha\le1$, $0\le k\le k_0$, $0\le t\le T_0$ and $X\in\mathbb C^{2n}$,
\begin{multline}\label{14112018E9}
	\bigg\vert\sum_{p=0}^{k_0-k-1}t^{p+1}\sqrt{\Reelle Q}B_{\alpha,p,k}(t)(\Reelle F)(\Imag F)^pX\bigg\vert \\
	\le\sum_{p=0}^{k_0-k-1}t^{p+1}\big\vert\sqrt{\Reelle Q}B_{\alpha,p,k}(t)J\sqrt{\Reelle Q}\sqrt{\Reelle Q}(\Imag F)^pX\big\vert
	\le c_0\sum_{p=0}^{k_0}t^{p+1}\big\vert\sqrt{\Reelle Q}(\Imag F)^pX\big\vert.
\end{multline}
We deduce from \eqref{13052019E15}, \eqref{19112018E8} and \eqref{14112018E9} that for all $0\le\alpha\le1$, $0\le t\le T_0$ and $X\in\mathbb C^{2n}$,
\begin{multline*}
	p_{\alpha,t}(X) \geq \sum_{k=0}^{k_0}\gamma_{\eta_k}t^k\big\vert\sqrt{\Reelle Q}(\Imag F)^kX\big\vert
	- c_0\sum_{k=0}^{k_0-1}\sum_{j=1}^{k_0-k}(\eta_k)_jt^{k+j}\big\vert\sqrt{\Reelle Q}(\Imag F)^{k+j}X\big\vert \\
	- c_0(k_0+1)\sum_{p=0}^{k_0}t^{p+1}\big\vert\sqrt{\Reelle Q}(\Imag F)^pX\big\vert
	- \sum_{k=0}^{k_0}t^k\big\vert\sqrt{\Reelle Q}(\Imag F)^kR_{\alpha,k_0-k}(t\Reelle F,t\Imag F)X\big\vert,
\end{multline*}
where the functions $p_{\alpha,t}$ are the ones appearing in the left-hand side of the estimate \eqref{13052019E10}, defined for all $0\le\alpha\le1$, $0\le t\le T_0$ and $X\in\mathbb C^{2n}$ by
\begin{equation}\label{13052019E18}
	p_{\alpha,t}(X) = \sum_{k=0}^{k_0}t^k\big\vert\sqrt{\Reelle Q}(\Imag F)^kG_{\alpha}(t\Reelle F,t\Imag F)X\big\vert.
\end{equation}
We make the change of indexes $j' = k$ and $k' = k+j$ in the following sum
$$\sum_{k=0}^{k_0-1}\sum_{j=1}^{k_0-k}(\eta_k)_jt^{k+j}\big\vert\sqrt{\Reelle Q}(\Imag F)^{k+j}X\big\vert
= \sum_{k=1}^{k_0}\bigg(\sum_{j=0}^{k-1}(\eta_j)_{k-j}\bigg)t^k\big\vert\sqrt{\Reelle Q}(\Imag F)^kX\big\vert.$$
Considering the quantity 
\begin{equation}\label{20112018E1}
	\varepsilon_{\eta,k,t} = \gamma_{\eta_k} - c_0\sum_{j=0}^{k-1}(\eta_j)_{k-j} - c_0(k_0+1)t,
\end{equation}
and the remainder term
\begin{equation}\label{19112018E10}
	\Sigma_{\alpha,t}(X) = \sum_{k=0}^{k_0}t^k\big\vert\sqrt{\Reelle Q}(\Imag F)^kR_{\alpha,k_0-k}(t\Reelle F,t\Imag F)X\big\vert,
\end{equation}
we deduce that for all $0\le\alpha\le1$, $0\le t\le T_0$ and $X\in\mathbb C^{2n}$, $p_{\alpha,t}(X)$ satisfies the estimate
\begin{equation}\label{ou_yeah}
	p_{\alpha,t}(X)\geq(\gamma_{\eta_0}-c_0(k_0+1)t)\big\vert\sqrt{\Reelle Q}X\big\vert
	+ \sum_{k=1}^{k_0}\varepsilon_{\eta,k,t}t^k\big\vert\sqrt{\Reelle Q}(\Imag F)^kX\big\vert - \Sigma_{\alpha,t}(X).
\end{equation}
Now, we determine the $\eta_k\in(\mathbb R^*_+)^{k_0-k}$. We would like to have $c_0(\eta_j)_{k-j} = \frac{\gamma_{\eta_k}}{k+1}$. Therefore, we define for all $0\le k\le k_0-1$ and $1\le j\le k_0-k$,
\begin{equation}\label{13112018E13}
	(\eta_k)_j = \gamma_{\eta_{k+j}} (c_0(k+j+1) )^{-1}.
\end{equation}
This construction seems implicit but, in fact, it is not. Indeed, to define $\eta_k$, we just need to know $\gamma_{\eta_{\ell}}$ for the indexes $k+1\le\ell\le k_0$ and since $\gamma_{\eta_{k_0}}=1$, we can proceed by induction. With this construction \eqref{13112018E13} of $\eta_k$, we have that for all $1\le k\le k_0$ and $0\le t\le T_0$,
$$\varepsilon_{\eta,k,t} = \frac{\gamma_{\eta_k}}{k+1} - c_0(k_0+1)t.$$
We deduce from this construction and \eqref{ou_yeah} that for all $0\le\alpha\le1$, $0\le t\le T_0$ and $X\in\mathbb C^{2n}$,
$$p_{\alpha,t}(X)\geq\sum_{k=0}^{k_0}\bigg(\frac{\gamma_{\eta_k}}{k+1} - c_0(k_0+1)t\bigg)t^k\big\vert\sqrt{\Reelle Q}(\Imag F)^kX\big\vert - \Sigma_{\alpha,t}(X).$$
Therefore, there exist some positive constants $c_1>0$ and $0<T_1<T_0$ such that for all $0\le\alpha\le1$, $0\le t\le T_1$ and $X\in\mathbb C^{2n}$,
\begin{equation}\label{13052019E16}
	p_{\alpha,t}(X) \geq c_1\sum_{k=0}^{k_0}t^k\big\vert\sqrt{\Reelle Q}(\Imag F)^kX\big\vert - \Sigma_{\alpha,t}(X).
\end{equation}
Now, we prove that the reminder term $\Sigma_{\alpha,t}$ can be controlled by $\sum_{k=0}^{k_0}t^k\vert\sqrt{\Reelle Q}(\Imag F)^kX\vert$. To that end, we begin by observing from 
\eqref{03092019E8} and \eqref{19112018E10} that $0\le\alpha\le1$, $0\le t\le T_1$ and $X\in\mathbb C^{2n}$,
\begin{align}\label{03092019E10}
	\Sigma_{\alpha,t}(X) & \le C_0\sum_{k=0}^{k_0}t^k\big\Vert\sqrt{\Reelle Q}(\Imag F)^k\big\Vert\big\Vert(t\Reelle F,t\Imag F)\big\Vert^{k_0-k+1}_{\infty}\vert X\vert \\
	& = t^{k_0+1}\bigg(C_0\sum_{k=0}^{k_0}\big\Vert\sqrt{\Reelle Q}(\Imag F)^k\big\Vert\big\Vert(\Reelle F,\Imag F)\big\Vert^{k_0-k+1}_{\infty}\bigg)\vert X\vert. \nonumber
\end{align}
Then, the inequality \eqref{13052019E11}, the estimate \eqref{03092019E10} and Lemma \ref{13112018L1} imply that there exists a positive constant $c_2>0$ such that for all $0\le\alpha\le1$, $0\le t\le\min(1,T_1)$ and $X\in (S+iS)^{\perp}$,
$$\Sigma_{\alpha,t}(X)\le c_2t\bigg(\sum_{k=0}^{k_0}t^k\big\vert\sqrt{\Reelle Q}(\Imag F)^kX\big\vert\bigg),$$
where the orthogonality is taken with respect to the Hermitian structure of $\mathbb C^{2n}$. This estimate combined with \eqref{13052019E16} shows the existence of positive constants $c_3>0$ and $0<T_2<T_1$ such that for all $0\le\alpha\le1$, $0\le t\le T_2$ and $X\in (S+iS)^{\perp}$,
\begin{equation}\label{13112018E17}
	p_{\alpha,t}(X) \geq (c_1-c_2t)\sum_{k=0}^{k_0}t^k\big\vert\sqrt{\Reelle Q}(\Imag F)^kX\big\vert\geq c_3\sum_{k=0}^{k_0}t^k\big\vert\sqrt{\Reelle Q}(\Imag F)^kX\big\vert.
\end{equation}
Now, it only remains to check that the estimate \eqref{13112018E17} can be extended to all $X\in\mathbb C^{2n}$. To that end, we notice that for all $0\le k\le k_0$, $X\in S+iS$ and $Y\in \mathbb C^{2n}$,
\begin{equation}\label{28082019E1}
	\sqrt{\Reelle Q}(\Imag F)^k(X+Y) = \sqrt{\Reelle Q}(\Imag F)^kY,
\end{equation}
since $\sqrt{\Reelle Q}(\Imag F)^k(S+iS) = \{0\}$ by definition \eqref{12102018E7} of the singular space $S$. This implies that for all $0\le t\le T_2$ and $X\in\mathbb C^{2n}$ written $X = X_{S+iS} + X_{(S+iS)^{\perp}}$, with $X_{S+iS}\in S+iS$ and $X_{(S+iS)^{\perp}}\in (S+iS)^{\perp}$ according to the decomposition $\mathbb C^{2n} = (S+iS)\oplus(S+iS)^{\perp}$, the orthogonality being taken with respect to the Hermitian structure of $\mathbb C^{2n}$, we have
\begin{equation}\label{13052019E17}
	\sum_{k=0}^{k_0}t^k\big\vert\sqrt{\Reelle Q}(\Imag F)^kX\big\vert = \sum_{k=0}^{k_0}t^k\big\vert\sqrt{\Reelle Q}(\Imag F)^kX_{(S+iS)^{\perp}}\big\vert.
\end{equation}
Moreover, it follows from the assumption \eqref{03092019E6} that for all $0\le\alpha\le1$, $0\le t\le T_2$, $S+iS$ is a stable subspace of  $G_{\alpha}(t\Reelle F,t\Imag F)$. Consequently, we deduce from \eqref{13052019E18}, \eqref{28082019E1} that for all $0\le\alpha\le1$, $0\le t\le T_2$ and $X\in\mathbb C^{2n}$,
\begin{equation}\label{13052019E21}
	p_{\alpha,t}(X) = p_{\alpha,t}(X_{(S+iS)^{\perp}}).
\end{equation}
As a consequence of \eqref{13052019E17} and \eqref{13052019E21}, the estimate \eqref{13112018E17} can be extended to all $0\le\alpha\le1$, $0\le t\le T_2$ and $X\in\mathbb C^{2n}$. This ends the proof of Lemma \ref{05112018L2}.
\end{proof}

The two following lemmas are used to prove Lemma \ref{05112018L2}.

\begin{lem}\label{13112018L1} There exists a positive constant $c>0$ such that for all $0\le t\le 1$ and $X\in (S+iS)^{\perp}$,
$$\kappa_t(X)\geq ct^{2k_0}\vert X\vert^2,$$
where the orthogonality is taken with respect to the Hermitian structure of $\mathbb C^{2n}$.
\end{lem}

\begin{proof} We begin by observing that for all $0\le t\le 1$ and $X\in\mathbb C^{2n}$,
\begin{equation}\label{14012019E6}
	\kappa_t(X) \geq t^{2k_0}\sum_{k=0}^{k_0}\big\vert\sqrt{\Reelle Q}(\Imag F)^kX\big\vert^2.
\end{equation}
It follows from \eqref{12102018E9}, \eqref{12092018E2} and \eqref{23102018E8} that there exists a positive constant $c>0$ such that for all $X\in S^{\perp}$,
$$\kappa_t(X)\geq t^{2k_0}\sum_{k=0}^{k_0}\big\vert\sqrt{\Reelle Q}(\Imag F)^kX\big\vert^2\geq ct^{2k_0}\vert X\vert^2,$$
since $V^{\perp}_{k_0} = S^{\perp}$. Moreover, if $X\in(S+iS)^{\perp}$, then $\Reelle X,\Imag X\in S^{\perp}$ and since $\kappa_t$ is a nonnegative quadratic form, we deduce that 
$$ct^{2k_0}\vert X\vert^2 = ct^{2k_0}\vert\Reelle X\vert^2 + ct^{2k_0}\vert\Imag X\vert^2\le\kappa_t(\Reelle X) + \kappa_t(\Imag X) = \kappa_t(X).$$
This ends the proof of Lemma \ref{13112018L1}.
\end{proof}

\begin{lem}\label{lem_l1_ineq} Let $m\in \mathbb N^*$ and $\eta\in(\mathbb R^*_+)^m$. Then, we have that for all $x,y_1,\dots,y_m\in \mathbb C^m$,
$$\Big\vert x + \sum_{j=1}^m y_j\Big\vert\geq\frac{\vert x\vert}{1+ \eta_{\min}^{-1}} - \sum_{j=1}^m \eta_j\vert y_j\vert, \ \ \ \ \  \mathrm{with} \ \eta_{\min} = \min_{1\le j\le m}\eta_j.$$
\end{lem}

\begin{proof} Let $x,y_1,\dots,y_m\in \mathbb C^m$. We consider $\alpha = \frac1{1+ \eta_{\min}}$ and distinguish two cases: \\[5pt]
\textbf{1.} On the one hand, if $\alpha\vert x\vert\geq\sum_{j=1}^m\vert y_j\vert$, we have that
$$\Big\vert x + \sum_{j=1}^m y_j\Big\vert + \sum_{j=1}^m\eta_j\vert y_j\vert 
\geq \Big\vert x + \sum_{j=1}^m y_j\Big\vert 
\geq \vert x\vert - \sum_{j=1}^m\vert y_j\vert
\geq \vert x\vert(1-\alpha) = \frac{\vert x\vert}{1+ \eta_{\min}^{-1}}.$$
\textbf{2.} On the other hand, when $\alpha\vert x\vert \le \vert y_1\vert+\dots+ \vert y_m\vert$, it follows that
$$\Big\vert x + \sum_{j=1}^m y_j\Big\vert + \sum_{j=1}^m\eta_j\vert y_j\vert
\geq\sum_{j=1}^m\eta_j\vert y_j\vert
\geq\alpha\eta_{\min}\vert x\vert = \frac{\vert x\vert}{1+ \eta_{\min}^{-1}}.$$
This ends the proof of Lemma \ref{lem_l1_ineq}.
\end{proof}

To end this subsection, let us detail why Lemma \ref{05112018L2} can be applied to the functions $G$ and $G_{\alpha}$ respectively defined in \eqref{04092019E4} and \eqref{03092019E4}.

\begin{lem}\label{25122025L7} The function $G$ defined in \eqref{04092019E4} satisfies the assumptions of Lemma \ref{05112018L2}.
\end{lem}

\begin{proof} Let us recall that the function $G$ is given by
\begin{equation}\label{05092019E1}
	G(M,N) = 2 (\sqrt{e^{-2i(M+iN)}e^{-2i(M-iN)}} + I_{2n})^{-1}.
\end{equation}
The matrix exponential being defined as the sum of an absolutely convergent series, the product of the two exponentials is given by the following Cauchy product for all $(M,N)\in M_{2n}(\mathbb R)\times M_{2n}(\mathbb R)$,
\begin{equation}\label{05092019E4}
	e^{-2i(M+iN)}e^{-2i(M-iN)} = \sum_{j=0}^{+\infty}\frac{(-2i)^j}{j!}\sum_{\ell=0}^j\binom j{\ell}(M+iN)^{\ell}(M-iN)^{j-\ell}.
\end{equation}
Let us consider the non-commutative polynomial $P$ defined by
$$P(X,Y)=\sum_{j=1}^{k_0}\frac{(-2i)^j}{j!}\sum_{\ell=0}^j\binom j{\ell}(X+iY)^{\ell}(X-iY)^{j-\ell}\in\mathbb C_{k_0,0}\langle X,Y\rangle.$$
We also consider the function $R:(M,N)\in M_{2n}(\mathbb R)\times M_{2n}(\mathbb R)\rightarrow M_{2n}(\mathbb C)$ defined for all $(M,N)\in~M_{2n}(\mathbb R)\times M_{2n}(\mathbb R)$ by
$$R(M,N) = \sum_{j=k_0+1}^{+\infty}\frac{(-2i)^j}{j!}\sum_{\ell=0}^j\binom j{\ell}(M+iN)^{\ell}(M-iN)^{j-\ell}.$$
With these notations, the product of exponentials takes the following form for all $(M,N)\in~M_{2n}(\mathbb R)\times M_{2n}(\mathbb R)$,
\begin{equation}\label{05092019E5}
	e^{-2i(M+iN)}e^{-2i(M-iN)} = I_{2n} + P(M,N) + R(M,N).
\end{equation}
Notice that the term $R(M,N)$ is a remainder since for all $\rho>0$ there exists a positive constant $c>0$ such that for all $(M,N)\in B((0,0),\rho)$,
$$\Vert R(M,N)\Vert\le c\Vert(M,N)\Vert^{k_0+1}_{\infty}.$$
Now applying Lemma \ref{04092019L1} with $\rho =1$ (it could be chosen arbitrarily) and the analytic function 
\begin{equation}\label{06092019E1}
	f:z\in\mathbb D(1,1)\mapsto((\sqrt z+1)/2)^{-1},
\end{equation}
we deduce that there exists $\rho'\in(0,1)$ such that the function $G$ is well defined on $B((0,0),\rho')$ and satisfies the assumptions \eqref{04092019E9} and \eqref{04092019E10} of Lemma \ref{05112018L2} on $B((0,0),\rho')$ (with no dependence with respect to the parameter $0\le\alpha\le1$ here).

Always in order to apply Lemma \ref{05112018L2} to the function $G$, it remains to check that for all $t\geq0$ such that $(t\Reelle F,t\Imag F)\in B((0,0),\rho')$, 
$$G(t\Reelle F,t\Imag F)(S+iS)\subset S+iS.$$
Notice that by definition, we have $G(t\Reelle F,t\Imag F) = \Phi_t$ %from the definitions \eqref{14022019E8} and \eqref{05092019E1} of the matrices $\Phi_t$ and of the function $G$ respectively.
 The inclusion we aim at proving is therefore equivalent to the following one for all $t\geq0$ such that $(t\Reelle F,t\Imag F)\in B((0,0),\rho')$,
\begin{equation}\label{05092019E6}
	\Phi_t(S+iS)\subset S+iS.
\end{equation}
Since the matrix function $((\sqrt\cdot + I_{2n})/2)^{-1}$ is analytic on $B(I_{2n},1)$ (from the analyticity of the function \eqref{06092019E1} on $\mathbb D(1,1)$), there exists a sequence of complex numbers  $(\sigma_j)_{j\geq1}$ such that
$$\forall A\in B(I_{2n},1),\quad 2 (\sqrt A+ I_{2n})^{-1} = I_{2n} + \sum_{j=1}^{+\infty}\sigma_j(A-I_{2n})^j.$$
It follows that the matrix $\Phi_t$ is the sum of the following series for all $t\geq0$ such that $(t\Reelle F,t\Imag F)\in B((0,0),\rho')$,
\begin{equation}\label{05092019E7}
	\Phi_t = I_{2n} + \sum_{j=1}^{+\infty}\sigma_j(e^{-2itF}e^{-2it\overline F}-I_{2n})^j.
\end{equation}
Since $(\Reelle F)S=\{0\}$ and $(\Imag F)S\subset S$ from \eqref{23012019E2}, the two inclusions $F(S+iS)\subset S+iS$ and $\overline F(S+iS)\subset S+iS$ hold. They imply in particular that $e^{-2itF}(S+iS)\subset S+iS$ and $e^{-2it\overline F}(S+iS)\subset S+iS$ for all $t\geq0$. The inclusion \eqref{05092019E6} is then a consequence of this observation and \eqref{05092019E7}.
\end{proof}

\begin{lem}\label{06092019L1} The family of functions $(G_{\alpha})_{0\le\alpha\le1}$ defined in \eqref{03092019E4} satisfies the assumptions of Lemma \ref{05112018L2}.
\end{lem}

\begin{proof} We recall that the matrix functions $G_{\alpha}$ are defined for all $0\le\alpha\le1$ by
\begin{equation}\label{05092019E10}
	G_{\alpha}(M,N) = \exp\Big(-\frac{\alpha}2\Log\big(e^{-2i(M+iN)}e^{-2i(M-iN)}\big)\Big).
\end{equation}
Similarly to the previous study of the function $G$ in the proof of Lemma \ref{25122025L7}, we deduce that there exists $\rho>0$ and $C>0$ such that the function
$(M,N)\mapsto \Log\big(e^{-2i(M+iN)}e^{-2i(M-iN)}\big),$
is well defined on $B((0,0),\rho)$ and can be written as
$$\forall (M,N)\in B((0,0),\rho),\quad \Log\big(e^{-2i(M+iN)}e^{-2i(M-iN)}\big)= P(M,N)+R(M,N),$$
where $P\in \mathbb C_{k_0,0}\langle X,Y\rangle$ and $R$ is a remainder term
$$\forall (M,N)\in B((0,0),\rho),\quad \Vert R(M,N)\Vert\le C\Vert (M,N)\Vert^{k_0+1}_{\infty}.$$
Now, observing that the set $\{-(\alpha/2)P:0\le\alpha\le1\}$ is bounded, we deduce from Lemma \ref{04092019L1} applied with $f=\exp$ that there exists $\rho'\in(0,\rho)$ and $C'>0$ (independent of $\alpha$) such that for all $0\le\alpha\le1$, the function $G_{\alpha}$ is well defined on $B((0,0),\rho')$ and there exists $R_{\alpha}:B((0,0),\rho')\to M_{2n}(\mathbb{C})$ satisfying
$$\forall (M,N)\in B((0,0),\rho'),\quad \Vert R_{\alpha}(M,N)\Vert\le C'\Vert (M,N)\Vert^{k_0+1}_{\infty},$$
such that
$$\forall (M,N)\in B((0,0),\rho'),\quad G_{\alpha}(M,N) = I_{2n} + \Psi(-\frac{\alpha}2P)(M,N) + R_{\alpha}(M,N).$$
Since $\Psi$ is a continuous map, the family of functions $(G_{\alpha})_{0\le\alpha\le1}$ satisfies the assumptions \eqref{04092019E9} and \eqref{04092019E10} of Lemma \ref{05112018L2} on $B((0,0),\rho')$.

It remains to check that for all $0\le\alpha\le1$ and $t\geq0$ such that $(t\Reelle F,t\Imag F)\in B((0,0),\rho')$,
\begin{equation}\label{05092019E8}
	G_{\alpha}(t\Reelle F,t\Imag F)(S+iS)\subset S+iS.
\end{equation}
Let $0\le\alpha\le1$ and $t\geq0$ such that $(t\Reelle F,t\Imag F)\in B((0,0),\rho)$ fixed. Since the complex function $\exp(-(\alpha/2)\Log\cdot)$ is analytic on the disk $\mathbb D(1,1)$, the matrix function $\exp(-(\alpha/2)\Log\cdot)$ is analytic on $B(I_{2n},1)$. Thus, there exists a sequence $(\sigma_{\alpha,j})_{j\geq0}$ of complex numbers such that 
$$\forall A\in B(I_{2n},1),\quad \exp\Big(-\frac{\alpha}2\Log A\Big)=\sum_{j=0}^{+\infty}\sigma_{\alpha,j}(A-I_{2n})^j.$$
We deduce from this series expansion that
\begin{equation}\label{05092019E9}
	G_{\alpha}(t\Reelle F,t\Imag F) = \sum_{j=0}^{+\infty}\sigma_{\alpha,j}(e^{-2itF}e^{-2it\overline F}-I_{2n})^j.
\end{equation}
However, we have already noticed that the vector space $S+iS$ is stable by the matrices $e^{-2itF}$ and $e^{-2it\overline F}$. The inclusion \eqref{05092019E8} is therefore a consequence of this observation and \eqref{05092019E9}.
\end{proof}

\end{document}